\documentclass[reqno,twoside,11pt]{amsart}

\usepackage{amsmath,amsfonts,calrsfs,fullpage,amssymb,xcolor,verbatim,eucal,yfonts,mathrsfs}

\usepackage{mathtools}
\mathtoolsset{showonlyrefs}
\usepackage{url} 

\usepackage{bbm}

\newtheorem{Theorem}{Theorem}[section]

\newtheorem{Proposition}[Theorem]{Proposition}
\newtheorem{Lemma}[Theorem]{Lemma}
\newtheorem{lemma}[Theorem]{Lemma}

\newtheorem{Remark}[Theorem]{Remark}

\newtheorem{Hypothesis}{Hypothesis}

\makeatletter
\@addtoreset{equation}{section}

\makeatother

\def\a{\alpha}
\def\le{\left}
\def\r{\right}

\def\ds{\displaystyle}

\def\e{\epsilon}

\newcommand\norm[2]{\vert #1 \vert_{#2}}
\newcommand\norma[3]{\vert #2 \vert_{#1}}
\newcommand\normb[2]{\vert #2 \vert_{#1}} 
\newcommand\Normb[2]{\Vert #2 \Vert_{#1}} 
\newcommand{\bis}{{\prime\prime}}

\renewcommand{\d}{d}

\newcommand{\newu}{\mathrm{u}}

\newcommand{\newv}{\mathrm{v}} 

\newcommand\scalar[3]{\langle #2,#3\rangle_{#1} }
\newcommand\scalarb[3]{\bigl\langle #2,#3\bigr\rangle_{#1} }
\newcommand\WP{w^Q}

\newcommand{\lb}{\langle}
\newcommand{\rb}{\rangle}
\newcommand{\tr}{\mathrm{tr}}

\newcommand{\hs}{\hspace{-2truecm}}
\newcommand{\hsl}{\hspace{-1truecm}}
\newcommand{\hsp}{\hspace{2truecm}}
\newcommand{\hslp}{\hspace{1truecm}}
\newcommand{\hsllp}{\hspace{0.5truecm}}

\newcommand{\newspace}{\mathscr{L}(E,D(A_0))}
\newcommand{\newspaceH}{\mathscr{L}(E,H)}

\newcommand{\newM}{M}
\newcommand{\newsigma}{{ \sigma}}%

\newcommand\adda[1]{{\color{blue} #1}}
\newcommand\coma[1]{{\color{red} {#1}}}

\newcommand\dela[1]{}

\newcommand{\Fb}{\mathbf{F}}
\newcommand{\zb}{\mathbf{z}}

\newcommand{\tolong}{\longrightarrow}

\usepackage{todonotes}

\newcommand{\todozb}[1]
{\todo[color=teal!60,inline]{#1}}

\title{\bf Stochastic wave equations with constraints: well-posedness and Smoluchowski-Kramers diffusion approximation}\thanks{There are no data associated with this paper}\date{}

\author[Z. Brze{\'z}niak]{Zdzis{\l}aw Brze{\'z}niak}
\address{Department of Mathematics\\
University of York,}
\email{zdzislaw.brzezniak@york.ac.uk}

\author[S. Cerrai]{Sandra Cerrai}
\address{Department of Mathematics\\
University of Maryland\\
}
\email{cerrai@umd.edu}
\thanks{S. Cerrai was partially supported by the NSF grant  DMS-1954299 - {\em Multiscale analysis of infinite-dimensional stochastic systems}}

\subjclass[2010]{}


\begin{document}
\maketitle

\begin{abstract}

We investigate the well-posedness of a class of stochastic second-order in time damped evolution equations in Hilbert spaces, subject to the constraint that the solution lies within the unitary sphere. Then, we focus on a specific example, the stochastic damped wave equation in a bounded domain of a $d$-dimensional Euclidean space, endowed with the Dirichlet boundary condition, with the added constraint that the $L^2$-norm of the solution is equal to one. We introduce a small mass $\mu>0$ in front of the second-order derivative in time and examine the validity of a Smoluchowski-Kramers diffusion approximation. We demonstrate that, in the small mass limit, the solution converges to the solution of a stochastic parabolic equation subject to the same constraint. We further show that an extra noise-induced drift emerges, which  in fact does not account for the Stratonovich-to-It\^{o} correction term.

\end{abstract}

\section{Introduction}\label{sec-intro}

The objective of this paper is twofold. Firstly, we aim to establish the existence and uniqueness of global solutions to stochastic second-order in time damped evolution equations in Hilbert spaces, while imposing the constraint that the norm of the solution is equal to one.  Secondly, we focus on a specific case of such equations, namely the stochastic damped wave equation in a bounded domain of a $d$-dimensional Euclidean space, subject to the Dirichlet boundary condition, with the  constraint that the $L^2$-norm of the solution is equal to one.
 In this case, we introduce an additional parameter $\mu$, called mass, to such equation and aim to prove that the solution $u_\mu$ converges to the solution of a certain stochastic heat equation, with Dirichlet boundary conditions,  satisfying the constraint that the $L^2$-norm of the solution is equal to one, as well. Unlike in  all the examples studied in the existing literature, the limiting equation we obtain may not be of the Stratonovich form. However,  one can give an independent proof of the existence and uniqueness of solutions to such limiting problem by employing  methods similar to those used recently in  \cite{Brz+Huss_2020}.

The present paper is the first one to consider the problem of the well-posedness for the following class of evolution equations in a separable Hilbert space $H$
	\begin{equation}\label{eq-2.26-intro}
	 u_{tt}(t)+  \norma{H}{ u_t(t)}{H}^2u(t)=-A_0^2u(t)+\norma{H}{ A_0u(t)}{H}^2-\gamma u_t(t) +\sigma(u(t),u_t(t))\,\circ\, dW(t),
	\end{equation}
	subject to the finite-codimension constraint of living on $M=S_H(0,1)$, the unitary sphere of $H$, with the initial data    $(u_0,v_0)$ in $TM$,  the tangent bundle of $M$.
	Here $A_0$ is a non-negative self-adjoint operator on $H$ with domain $D(A_0)$, $\gamma$ is a positive constant, $W(t)$ is a cylindrical Wiener process with  reproducing kernel Hilbert space (RKHS) $K$ having gamma-radonifying embedding in some Banach space $E$, and $\sigma$ is a locally Lipschitz mapping defined on $\mathscr{H} := D(A_0)\times H$ with values in $\mathscr{L}(E,H)$, such that $\sigma(u,v)$  projects $E$  onto $T_uM$. Due to its Stratonovich formulation, when written in It\^o form equation \eqref{eq-2.26-intro}
contains a non-trivial Stratonovich-to-It\^o trace term.

It is worth noticing that equation \eqref{eq-2.26-intro} is a stochastic version of the constrained deterministic  equation
\[u_{tt}(t)+  \norma{H}{ u_t(t)}{H}^2u(t)=-A_0^2u(t)+\norma{H}{ A_0u(t)}{H}^2-\gamma u_t(t),\]
where  the terms $\norma{H}{ u_t(t)}{H}^2u(t)$ and $\norma{H}{ A_0u(t)}{H}^2$ are added to $u_{tt}(t)$ and $-A_0^2u(t)$, respectively, in order to ensure that the solution stays on the manifold $M$. Actually, if one adds to the deterministic equation above any stochastic perturbation such as
\[\sigma(u(t),u_t(t))\,\circ\, dW(t),\]
under the assumption that for every $(u,v) \in\,\mathscr{H}$ the mapping $\sigma(u,v)$  projects $E$ into the tangent space $T_uM$, due to the presence of the Stratonovich integral the invariance property holds also for the stochastic equation \eqref{eq-2.26-intro}.

It is important to note that the model presented in equation \eqref{eq-2.26-intro} differs from the recent study of stochastic geometric wave equations by the first-named author and coauthors in \cite{Brz+GNR_2022}. In those works, the solution is constrained to a manifold within the Euclidean space (e.g., the sphere in $\mathbb{R}^3$), whereas here, the solution of \eqref{eq-2.26-intro} is restricted to a functional manifold - the set of all square-integrable functions with $L^2$-norm one. Although both classes of SPDEs exhibit nonlinearities with cubic growth, their nature differs significantly: in the former case, the nonlinearities are local, while in the latter, they are non-local. In particular, this fundamental distinction necessitates different mathematical techniques for their analysis.

However, the investigation of deterministic and stochastic constrained partial differential equations (PDEs) is not a new field of study. In this regard, we would like to mention the papers \cite{Caff+Lin_2009} by Caffarelli and Lin, and \cite{Rybka_2006} by Rybka, where deterministic heat flows in Hilbert manifolds were explored. The motivation behind the former paper was to find a gradient flow approach to a specific minimization problem.
A similar inquiry was undertaken for the stochastic 2-D Navier-Stokes equation by the first named author and Dhariwal in \cite{Brz+Dhariwal_2021}. This work was preceded by the paper \cite{Brz+DM_2018} by these two authors and Mariani, as well as the paper \cite{CPR_2009} by Caglioti, Pulvirenti, and Rousset, whose motivation  was the occurrence of different dissipation timescales. 
Both in \cite{Brz+DM_2018} and in \cite{CPR_2009}, the study focused on the deterministic 2-D Navier-Stokes equations (NSEs).
As well known, the $L^2$-norm of the solutions $u^\nu$ to such equations, with viscosity $\nu >0$, converges   to $0$, as $t\to \infty$,
while the  $L^2$-norm of the (strong) solution $u$ of the limiting Euler equation, which corresponds to  $\nu=0$, remains constant. Thus, it was proposed in \cite{CPR_2009} to consider a modification of the Navier-Stokes equations in which the  $L^2$-norm of  the solution  remains constant in time as well.
Finally, we would like to mention the recent work of Hairer and Rosati \cite{hr}, which examines the projected process of vector-valued linear SPDEs. This process corresponds to the angular component of the solution constrained to the unit sphere in $L^2$, and the study investigates its ergodic behavior.

A physical motivation for our model arises from the so-called relativistic limit of the Klein-Gordon equation; see, for example, \cite{Cirincione+Chernoff_1981} and \cite{Machihara_2002} for mathematical treatments, and \cite{Zee_2010} and \cite{Unknown_2024} for physical discussions. It follows from these and many related papers that the solutions $u^c$
  to the relativistic Klein-Gordon equation - with $c$
 denoting the speed of light - converge, in an appropriate sense, to solutions of the Schr\"odinger equation. Since the $L^2$-norm of Schr\"odinger solutions is conserved, it is natural to consider a modification of the Klein-Gordon equation that also preserves the 
$L^2$-norm of its solutions, in the hope that such modified equations will provide a better approximation to the limiting Schr\"odinger dynamics, at least over intermediate time scales.

\medskip

The main result concerning the existence and uniqueness of solutions for equation \eqref{eq-2.26-intro} is formulated in Theorem \ref{theorem 2.5}, and an extension of that result to the case of more regular initial data is presented in Theorem \ref{thm-2.7-regular data}. Both theorems, whose proofs  are presented in Section \ref{sec-proof of existence}, show that the stochastic constrained {\em wave} equation \eqref{eq-2.26-intro} admits a unique global solution living in the tangent bundle $TM$. Notice that the coefficients in equation \eqref{eq-2.26-intro} are only locally Lipschitz and have cubic growth. This means that the global well-posedness of the equation cannot be proven directly.  Namely, we first consider the equation in its mild formulation and prove that there exists a local maximal solution that is defined up to a certain stopping time $\tau$. Next, we prove that such solution $z=(u,v)$ stays on the tangent bundle $TM$. Finally we prove suitable a-priori bounds for the solution that allow to show that the solution is global and unique.

\smallskip

In the second part of the paper, the well posedness result for the abstract stochastic equation \eqref{eq-2.26-intro}  is applied to the specific case of a stochastic  damped wave equation in a bounded domain $D\subset \mathbb{R}^d$, endowed with the Dirichlet boundary condition and constrained to live in $M$, the unitary sphere in $H:=L^2(D)$. Namely, we consider the equation
\begin{equation}\label{eq1-infin-intro}
\left\{
\begin{array}{l}
\displaystyle{\mu\, \partial_t^2u_\mu(t,\xi)+\mu\normb{H}{ \partial_t u_\mu(t)}^2u_\mu(t,\xi)}\\[10pt]
\displaystyle{\hspace{+1.5truecm}=\Delta u_\mu(t,\xi)+\normb{H}{ \nabla u_\mu(t)}^2u_\mu(t,\xi)-\gamma\, \partial_t u_\mu(t,\xi)+\sigma(u_\mu(t))\,\partial_t \WP(t,\xi)}\\[10pt]
\displaystyle{u_\mu(0,\xi)=u_0(\xi),\ \ \ \  \partial_t u_\mu(0,\xi)=v_0(\xi),\ \ \ \ \ \ \   u_\mu(t,\xi)=0,\ \ \xi \in\,\partial D,}\end{array}
\right.\end{equation}
depending on a positive parameter $\mu$, where $\gamma$ is a positive constant, $\WP$ is a  Wiener process on $H$, with reproducing kernel Hilbert space $K$ and covariance operator $Q$, and the mapping $\sigma:H^1_0(D)\to\mathscr{L}(H)$ is such that  $\sigma(u)$ projects $H$ onto the tangent space $T_uM$.
 Notice that, since  here the diffusion coefficient depends on the unknown position $u_\mu$ and not on the velocity $\partial_t u_\mu$,  the Stratonovich trace term is equal to zero and the Stratonovich and the It\^o formulations coincide.

 As we mentioned above, our aim is studying the limiting behavior of the solution $u_\mu$ of equation  \eqref{eq1-infin-intro}, when the mass $\mu$ goes to zero.
 Namely, we fix an arbitrary condition $(u_0,v_0)$ that is sufficiently smooth and lives in the tangent bundle of $M$   and we show that $u_\mu$ converges in probability, in a suitable functional space, to
the unique solution of the equation
\begin{equation}
\label{limit-eq-intro}
\left\{\begin{array}{l}
\ds{\gamma\,\partial_t u(t,\xi)=\Delta u(t,\xi)+ \norm{\nabla u(t)}{H}^2 u(t,\xi)-\frac 12 \Vert \sigma(u(t))\Vert_{\mathscr{T}_2(K,H)}^2 u(t) +\sigma(u(t))\partial_t \WP(t,\xi),}\\[18pt]
\ds{u(0,\xi)=u_0(\xi),\ \ \ \ u(t,\xi)=0,\ \ \ \xi \in\,\partial D.	}
\end{array}\right.
	\end{equation}
	In particular, this means that in the diffusion-approximation limit   the term
	\[\mu\normb{H}{ \partial_t u_\mu(t)}^2u_\mu(t)\]
	converges, as $\mu\to 0$, to the non-trivial term
	\begin{equation} \label{fine100}-\frac 12 \Vert \sigma(u(t))\Vert_{\mathscr{T}_2(K,H)}^2 u(t),\end{equation}
	which depends on the diffusion coefficient $\sigma$ through its $\mathscr{T}_2(K,H)$ norm, where $\mathscr{T}_2(K,H)$ denotes the space of Hilbert-Schmidt operators from $K$ to $H$ (see Section 2 for all details).
	It is important to stress that the coefficient \eqref{fine100} does not coincides with the Stratonovich-to-It\^o correction term. Moreover, as we will show later with a concrete example, the solution to equation \eqref{limit-eq-intro} does  not coincide with the solution of the constrained  parabolic equation perturbed by Stratonovich-type noise
	\begin{equation}
\label{limit-eq-intro-other}
\left\{\begin{array}{l}
\ds{\gamma\,\partial_t u(t,\xi)=\Delta u(t,\xi)+ \norm{\nabla u(t)}{H}^2 u(t,\xi) +\sigma(u(t))\circ\partial_t \WP(t,\xi),}\\[18pt]
\ds{u(0,\xi)=u_0(\xi),\ \ \ \ u(t,\xi)=0,\ \ \ \xi \in\,\partial D,}
\end{array}\right.
	\end{equation}
which is the only example of constrained stochastic heat equation considered in the existing literature so far. In particular, our limiting result provides a new example of a constrained stochastic parabolic problem, which arises in a concrete situation such as  the small mass limit for equation \eqref{eq1-infin-intro}.

	 While this paper is the first one handling the case of SPDEs with constraints, a series of papers have investigated the validity of the so-called Smoluchowski-Kramers approximation, that describes the limiting behavior of the solution $u_\mu$, as the  mass $\mu$  vanishes.
For the finite dimensional case, the existing literature is quite broad and  we refer in particular to \cite{f}, \cite{fh}, \cite{hhv}, \cite{hmdvw} and  \cite{spi} (see also \cite{CF3}, \cite{CWZ} and \cite{lee} for systems subject to a magnetic field and \cite{hu} and \cite{Nguyen} for some related multiscaling problems).
We should also mention here that a simple model of the Smoluchwski-Kramers phenomenon  for stochastic SDEs has been investigated by Nelson in Chapters 9 and 10 of his famous book \cite{Nelson_1967}.
In recent years there has been an intense activity dealing with the Smoluchowski-Kramers diffusion approximation of   infinite dimensional systems. To this purpose, we refer to \cite{CF1}, \cite{CF2}, \cite{salins} and \cite{Lv2} for the  case of constant damping term (see also \cite{CS3} where systems subject to a magnetic field are studied), and to \cite{CX} and \cite{CXIE} for the case of  state-dependent damping. As a matter of fact, these two situations are quite different, as in the case of non-constant friction  a noise-induced term emerges from the small mass limit.

The study of the Smoluchowski-Kramers approximation does not  reduce only to the proof of the limit of the solutions $u_\mu$.
Actually, it is crucial to ascertain the stability of such an approximation in relation to other significant asymptotic characteristics exhibited by the two systems, such as their long-term behaviors, for instance.   To this purpose,  in \cite{CGH} and \cite{CF1} it is shown that  the statistically invariant states of the stochastic damped wave equation (in case of constant friction) converge in a suitable sense to the invariant measure of the limiting  equation. In the same spirit, the papers  \cite{sal}, \cite{sal2} and \cite{CXIE} are devoted to the analysis of the interplay between the small mass and the small noise limit. In particular, \cite{CXIE} studies the validity of a large deviation principle for the trajectories of the solution, while \cite{sal} and \cite{sal2} deal with the study of the convergence of the quasi-potential, that describes, as known, the asymptotics of the exit times and the large deviation principle for the invariant measure.

The current paper is the first one addressing the small mass limit for constrained infinite-dimensional systems. To the best of our knowledge, the only other paper in the existing literature that investigates this particular problem is \cite{bhw}, which focuses    on general manifolds in the finite-dimensional case. As mentioned in the introduction of \cite{bhw},   "Brownian motion of micro and nanoparticles occurring in complex environments can often be represented as two-dimensional or one-dimensional manifolds embedded within a three-dimensional space. For
example, the motion of proteins on cellular membranes occurs effectively on two-dimensional manifolds". Thus, the physical motivation of \cite{bhw} was to present Brownian motion
on a manifold as the zero-mass limit of an inertial system.
In  the present paper, we will be able to address an analogous problem in the case of space-dependent systems.

The transition from a finite number of degrees of freedom to an infinite number presents considerable challenges and complexities. The strategy we follow in our proof is somehow standard: we first prove suitable uniform bounds  with respect to $\mu \in\,(0,1)$ for the family of solutions $\{u_\mu\}_{\mu \in\,(0,1)}$, then, thanks to those bounds, we prove that the family $\{u_\mu\}_{\mu \in\,(0,1)}$ is tight in a suitable functional space and, finally, we identify any limiting point for the family $\{u_\mu\}_{\mu \in\,(0,1)}$ with the unique solution of the limiting problem \eqref{limit-eq-intro}. Nevertheless,  the demonstration of these steps is quite challenging and necessitates the introduction of novel arguments and techniques.

  Specifically, we must establish uniform bounds, with respect to $\mu$, for the solutions of \eqref{eq1-infin-intro} within functional spaces possessing higher regularity than $\mathscr{H}=H^1_0(D)\times H$, and the presence of cubic terms in the equation adds an extraordinary level of complexity to proving such bounds.
 The need for delicate uniform bounds for the solution \( u_\mu \), with respect to \( \mu \in (0,1) \), is not specific to this paper but is a key feature of all the existing literature on the Smoluchowski-Kramers approximation mentioned above. However, in this work, we require a priori bounds in spaces of higher regularity than those considered in previous papers, such as \cite{CF1}, \cite{CF2}, \cite{CGH}, and \cite{CX}. The first important reason for this necessity is that we need to take the limit of \( \vert\nabla u_\mu(t)\vert_H^2 u_\mu(t) \), and no integration by parts can be applied to handle the nonlocal term \( \vert\nabla u_\mu(t)\vert_H^2 \). To ensure the required tightness of \( \{u_\mu\}_{\mu \in (0,1)} \) in \( H^1(D) \), among other considerations, we require a priori bounds in \( H^\alpha(D) \) for \( \alpha > 1 \). It is important to note that, in this context, obtaining bounds in \( H^\alpha(D) \) for \( \alpha \in (1,2) \) is not any easier than obtaining bounds in \( H^2(D) \), due to the specific nature of equation \eqref{eq1-infin-intro}. Another reason for requiring bounds in higher-regularity spaces is that, in the proof of the limit, we need uniform bounds for the solution in \( L^2(0,T;H^2(D)) \). Given the peculiar and highly nontrivial nature of equation \eqref{eq1-infin-intro}, achieving this requires obtaining bounds for \( (u_\mu, \sqrt{\mu}\,\partial_t u_\mu) \) in \( H^3(D) \times H^2(D) \). Finally, we emphasize that the cubic nature of the nonlinearities introduces additional challenges, as we must establish uniform bounds for the fourth moments of the solution and its time derivative. For this purpose, see Lemma \ref{Lemma6.1} and Lemma \ref{Lemma7.1}.

\vspace{0.4truecm}

Before concluding this introduction, we provide a brief overview of the contents of our paper. The first two sections are dedicated to the examination of the well-posedness of the abstract problem \eqref{eq-2.26-intro}. In Section \ref{sec-main}, we present the notation and assumptions, and describe how the abstract damped wave equation can be introduced in the deterministic setting. We then introduce the stochastically forced version of the equation and establish a series of preliminary results concerning the diffusion coefficient $\sigma$. Finally, we state the two main results concerning the existence and uniqueness of solutions: Theorem \ref{theorem 2.5} and Theorem \ref{thm-2.7-regular data}. Section \ref{sec3} is dedicated to providing the proofs of these two theorems.

In the remaining seven sections of our paper, we delve into the examination of the validity of the Smoluchowski-Kramers diffusion approximation for the system \eqref{eq1-infin-intro}.
Section \ref{sec4} is dedicated to introducing the necessary notation and assumptions. In Theorem \ref{teo8.3}, we present the main result of this study.
In Section \ref{sec5}, we provide a concrete example to illustrate that our limiting equation \eqref{limit-eq-intro} and equation \eqref{limit-eq-intro-other} are, in fact, two distinct equations.
The subsequent two sections focus on establishing the required uniform bounds for the solution $(u_\mu, \partial_t u_\mu)$ of equation \eqref{eq1-infin-intro}. In Section \ref{sec6}, we prove bounds in $H^1_0(D) \times H$, while in Section \ref{sec7}, we establish bounds in $H^2(D) \times H^1_0(D)$.
Section \ref{sec8} addresses the proof of the tightness of $\{\mathcal{L}(u_\mu)\}_{\mu \in (0,1)}$ within the appropriate functional space. Finally, in Section \ref{sec9}, we conclude the proof of Theorem \ref{thm-2.7-regular data} by identifying any limit point of the family $\{\mathcal{L}(u_\mu)\}_{\mu \in (0,1)}$ as the unique solution of equation \eqref{limit-eq-intro}.

\section{The well-posedness: notations, assumptions and main results}\label{sec-main}

	Let us briefly introduce the basic notations.
	We will denote by $H$ a separable Hilbert space endowed with an inner product $\scalar{H}{\cdot}{\cdot}$ and the corresponding norm $\norma{H}{\cdot}{}$.
If $E$ and $F$ are Banach spaces,  the class of all bounded linear operators from $E$ to $F$  will be denoted by $\mathscr{L}(E,F)$. We will use a shortcut notation
$\mathscr{L}(E)$ for $\mathscr{L}(E,E)$. It is known that $\mathscr{L}(E,F)$ is also a Banach space.  We will denote by $\mathscr{L}_2(E\times E;F)$  the Banach space of all bounded bilinear operators from $E\times E=:E^2$ to $F$. If $K$ is another Hilbert space,  we will denote
	by $\mathscr{T}_2(K,H)$, or $\gamma(K,H)$,   the Hilbert space of all Hilbert-Schmidt operators from $K$ to $H$, endowed with the natural inner product and norm.
It is known that $\mathscr{T}_2(K,H) \hookrightarrow \mathscr{L}(K,H)$ continuously.
If $\{e_j\}_{j \in \mathbb{N}}$ is   an orthonormal basis of a separable Hilbert space $K$ which  is continuously embedded into a Banach space $E$ and
\[
\sum_{j=1}^{\infty} \norm{e_j}{E}^2 <\infty,
\]
then for every $\Lambda \in \mathscr{L}_2(E\times E;H)$ we put
\begin{equation}\label{eqn-tr K}
\tr_K(\Lambda)= \sum_{i=1}^{\infty} \Lambda(e_j,e_j).
\end{equation}

If $X$ is a normed vector space, $a\in X$ and $r>0$, then we will denote by $B_X(a,r)$, respectively $S_X(a,r)$,   the open ball, respectively the sphere,   in  $X$ of radius $r$ and center $a$.

\medskip

In what follows, we shall  assume that   $A_0$ is  a  non-negative self-adjoint operator on $H$ and we shall denote its  domain by $D(A_0)$. If we put
\begin{equation}\label{eqn-A_1}
	A_1 := \sqrt{A_0^2 + \delta I},
	\end{equation}
 with $\delta=0$ if $A_0$ is invertible (i.e. $A_0$ is injective, surjective and the inverse $A_0^{-1}$ is bounded),  and $\delta=1$ otherwise,
then  $A_1$, with  $D(A_1)= D(A_0)$, is  a strictly positive self-adjoint  operator on $H$ and   $0$ belongs to $\rho(A_1)$, the resolvent set of the operator $A_1$. In particular,
the inverse $A_1^{-1}:H \to H$ is bounded.
Whenever we will use the space $D(A_1)$ we will always mean that it is endowed with the  norm $\normb{H}{A_1\cdot }$ and the corresponding inner product.

	We will denote by $\mathscr{H}$  the Hilbert space
	\begin{equation}\label{eqn-H-mathscr}
	\mathscr{H} := D(A_1)\times H=D(A_0)\times H,
	\end{equation}
	endowed with the following  inner product
 \[
\begin{split}
	\scalar{\mathscr{H}}{z_1}{z_2}= \scalar{H}{A_1u_1}{A_1u_2}+\scalar{H}{v_1}{v_2},
	\end{split}\]
	The corresponding norm $|\cdot|_{\mathscr{H}}$ satisfies
	\[
\begin{split}
	\norm{z}{\mathscr{H}}^2= \normb{H}{ A_1 u}^2  + \normb{H}{ v}^2  ,\;\;\; z=(u,v)\in \mathscr{H}.
\end{split}
	\]

	We will also use the following scale of Hilbert spaces
 \begin{equation}\label{eqn-H-mathscr-alpha}
	\mathscr{H}_\alpha :=D(A_1^{1+\alpha})\times D(A_1^\alpha), \;\;\; \alpha \geq 0.
	\end{equation}
Each space $\mathscr{H}_\alpha$ is 	endowed with an inner product defined for every $z_i=(u_i,v_i)\in \mathscr{H}_\alpha$ by
	 \[
	\scalar{\mathscr{H}_\alpha}{z_1}{z_2}
:= \scalar{H}{A_1^{1+\alpha}u_1}{A_1^{1+\alpha}u_2}+\scalar{H}{A_1^\alpha v_1}{A_1^\alpha v_2}.
	\]
		Note that obviously $\mathscr{H}_0=\mathscr{H}$, with equal norms and inner products.

Next, we introduce the  linear operator
$\mathscr{A}$ in the space $\mathscr{H}$ as follows,
\begin{equation}\label{eqn-mathscr A}
\mathscr{A} z
:=(v,-A_0^2u), \;\;\; \ z=(u,v)\in D(\mathscr{A}):=\mathscr{H}_{1}.
\end{equation}
It is well known that $\mathscr{A}$ generates a $C_0$ group (of exponential growth) $\mathscr{S}=(\mathscr{S}(t))_{t\in  \mathbb{R}}$ on $\mathscr{H}$, see e.g. \cite{Brz+Masl+Seidl_2005} and references therein.
If $A_0$ is invertible (so that  we take $\delta=0$ in \eqref{eqn-A_1}) then  $\mathscr{S}$ is a unitary group.
The restriction $\mathscr{A}_\alpha$ of the operator $\mathscr{A}$	 defined by
\[\mathscr{A}_\alpha z=\mathscr{A}z, \;\;\; z\in D(\mathscr{A}_\alpha):=\mathscr{H}_{\alpha+1}\\
\]
is the generator of  a $C_0$ group $\mathscr{S}_\alpha=(\mathscr{S}_\alpha(t))_{t\in  \mathbb{R}}$ on $\mathscr{H}_\alpha$ and
$\mathscr{S}_\alpha(t)$ is the restriction of $\mathscr{S}(t)$ to the space $\mathscr{H}_\alpha$. In what follows, we will not make this distinction and denote all these objects
without the subscript $\alpha$, unless our approach could lead to ambiguity.

\medskip
	
	Let us also consider a separable Hilbert space $K$ and  a separable Banach space $E$ such that $K \subset  E$ continuously and the embedding
\begin{equation}\label{eqn-K to E}
  i: K \hookrightarrow E \mbox{ is gamma-radonifying}.
\end{equation}
  By the Kwapie\'n-Szyma\'nski Theorem \cite{Kwapien+Szym_1980} there exists
an  orthonormal  basis  $\{e_j\}_{j \in \mathbb{N}}$ such that
\[
\sum_{j\geq 1}  \norm{i e_j}{E}^{2} <\infty.
\]
We assume that  $W_j=(W_j(t):t\geq 0)$, $j\in\mathbb{N}$ is a sequence of iid
real Wiener processes
defined on some filtered probability space $(\Omega, \mathscr{F}, (\mathscr{F}_t)_{t\geq 0},\mathbb{P})$ satisfying the usual assumptions.
 Let also
\[
W(t)=\sum_{j= 1}^\infty W_j(t) i e_j
\]
 be a $E$-valued Wiener process. The Reproducing Kernel Hilbert space of the law of $W(1)$ is equal to the space $K$ and the process $W$ can also be viewed as a canonical $K$-cylindrical  Wiener process.

We assume now that  $\gamma$ is a positive constant (we call it the damping coefficient) and fix a mapping $\sigma_0:\mathscr{H}\to\mathscr{L}(E,H)$.  	Our aim is to study a certain {\em constrained} version of the following abstract damped wave equation
		\begin{align}\label{eq-2.1}
    & u_{tt}(t) + A_0^2\,u(t) = -\gamma u_t(t)+\sigma_0(u(t),u_t(t))\,dW(t),
    \end{align} with the following initial conditions
    	\begin{align}
     u(0) =u_0 \in\,D(A_0),\;\;\; \ u_t(0)=v_0 \in\,H.\label{eq-2.2}
	\end{align}

	By  {\em constrained} we mean that we want our solution $u$  to stay on $M$, where $M$ is  the unit sphere  in $H$,
i.e.
\[
M=S_{H}(0,1):=\bigl\{ x \in H: \normb{H}{ x } =1\bigr\}.
\]
In particular,  we need to assume that  the initial data $u_0$  satisfy the same  condition, i.e.
	\begin{equation}
	\normb{H}{ u_0 }=1. \label{eq-2.5}
	\end{equation}
In what follows, it is convenient to use the tangent bundle of $M$ which in the present framework can be defined as
\[
\mathscr{M}:=TM=\bigcup_{u \in\,M}\{ (u,v)\in M\times T_uM \},\ \ \ \ \ \ T_uM= \{v \in\,H\,:\,\scalar{H}{u}{v}=0\}.
\]
Note that  $\mathscr{M}$ is a Hilbert manifold modeled on the Hilbert space $H\times H$. Moreover,
$\mathscr{M}$ is a closed subspace of $\mathscr{H}$. We will endow the former set  with a metric inherited from  the latter space.

	One can heuristically see that if $u(t) \in\,M$, for all $t\geq 0$,   then  the following property is also verified
	\[
	\scalar{H}{u(t)}{u_t(t)}=0 ,\;\;\;t\geq 0. \label{eq-2.6}
	\]
	Hence, in addition to \eqref{eq-2.5}, we also need to assume the following condition on the initial data 	\begin{equation}
	\scalar{H}{u_0}{v_0}=0. \label{eq-2.7}
	\end{equation}

It is quite obvious that a solution to equation \eqref{eq-2.1} with initial   conditions \eqref{eq-2.2}   will not necessarily stay on the manifold $M$, even though the initial data $(u_0,v_0)$ satisfy the compatibility conditions
\eqref{eq-2.5} and \eqref{eq-2.7}, see  \cite{Brz+DHM_2018}.
We will show below that it is possible to resolve this conundrum by modifying equation \eqref{eq-2.1}.
 In order to find the appropriate modification one can think of its deterministic part as the equation
	\begin{equation}
	 u_{tt}(t) + \nabla_u \Phi_0(u(t)) = -\gamma\, u_t(t), \label{eq-2.8}
	\end{equation}
	where the gradient $\nabla_u$  is understood in the $H$-sense and the {\em energy function} $\Phi_0$ is defined by
	\[
	\Phi_0(u) =\frac{1}{2}\normb{H}{ A_0u}^2 ,\;\;\; u\in D(A_0).
	\]
	Recalling that  $M$ is the unit sphere $S_H(0,1)$, if $\Phi$ is the {\em restricted} energy functional
	\[
	\Phi(u)=\Phi_0(u) ,\;\;\;u\in D(A_0)\cap M, \]
	we can  replace the term $\nabla_u \Phi_0$ by $\widetilde{\nabla}_u\Phi_0$.
	Here,  the gradient $\widetilde{\nabla}_u$ is  understood as the gradient of $M$ with respect to the {\em metric} on $M$ inherited from $H$, i.e.
	\[
	\widetilde{\nabla}_a\Phi= \Pi_a(\nabla_a \Phi_0) \in T_a M, \quad a\in M\cap D(A_0), \]
	where, for every $a\in M$, we denote by  $\Pi_a \in \mathscr{L}(H)$ 	 the orthogonal  projection onto $T_aM$, i.e.
	\begin{equation}
\Pi_a:H \ni u \mapsto u-\scalar{H}{u}{a}\, a  \in H. \label{eq-2.14}
	\end{equation}

	Since   we have
	\[
	\nabla_a\Phi_0 = A_0^2\,a, \quad a\in D(A_0^2), \]
	we infer that for $a \in D(A_0)\cap M$,
\[
\begin{split}
	 \widetilde{\nabla}_a\Phi= \Pi_a(A_0^2\,a)  =A_0^2\,a-\scalar{H}{A_0^2\,a}{a} a   =A_0^2\,a-\normb{H}{ A_0\,a}^2a.
\end{split}
	\]
	Moreover, the {\em acceleration} term $u_{tt}(t)$ has also to be modified in a similar fashion so that $u(t)$ stays on $M$, i.e. we need to replace it by
	\[
	\Pi_{u(t)}(u_{tt}(t)) = u_{tt}(t)-\scalar{H}{u_{tt}(t)}{u(t)} u(t).
	\]
Now, since we are assuming that $\normb{H}{ u(t)}=1$, it is immediate to check that
	\[
	\scalar{H}{u_{tt}(t)}{u(t)}=-\normb{H}{ u_t(t)}^2, \;\; t \geq 0,
	\]
	so that
	\[
	\Pi_{u(t)}( u_{tt}(t))=  u_{tt}(t) +  \normb{H}{ u_t(t)}^2u(t). \]

 	In this way we obtain the following  constrained version of equation \eqref{eq-2.8}
	\begin{equation} \label{eq-2.18}
	u_{tt}(t)+  \normb{H}{ u_t(t)}^2u(t)=-A_0^2\,u(t)+\normb{H}{ A_0u(t)}^2u(t)-\gamma\, u_t(t).
	\end{equation}
	The above heuristic argument can be made rigorous, as shown in the following two theorems, whose proofs are postponed to next section, where we will consider the more general stochastic case.
	
	\begin{Theorem} \label{theorem 2.1}
Assume that $f \in L^2_{\mathrm{loc}}([0,\infty; H))$.
		 If the initial data  $(u_0,v_0)\in \mathscr{H}$ satisfy the compatibility  conditions \eqref{eq-2.5} and \eqref{eq-2.7},  then there exists a unique function
		$
		(u,v)\in C(\bigl[0,\infty\bigr);\mathscr{H})$
		such that
\begin{equation}\label{eqn-u on M}
  		u\in C^1(\bigl[0,\infty\bigr);M),\ \ \ \  u_t(t)=v(t),\ \ t\geq 0,
\end{equation}
 and $u$ is a mild solution to  the following equation
 \begin{equation}\label{eq-2.18-external}
	 u_{tt}(t)+ \normb{H}{ u_t(t)}^2u(t)=-A_0^2\,u(t)+\normb{H}{ A_0\,u(t)}^2u(t)-\gamma\, u_t(y)+ \Pi_{u(t)}f(t).
	\end{equation}
Moreover, if we denote
\begin{equation} \label{eqn-Psi-energy functional}	\Psi_{}(u,v):= \frac{1}{2}\left(\normb{H}{ A_0u}^2+ \norma{H}{ v}{H}^2\right), \;\;\; \ (u,v)\in \mathscr{H},
	\end{equation}
then  for every $t\geq 0$ we have
\[
\begin{split}
  \Psi_{}(z(t))&=\Psi_{}(z_0)-\gamma  \int_0^t \norma{H}{ v(s)}{H}^2\, ds + \int_0^t \scalar{H}{v(s)}{f(s)}\, ds.\\
\end{split}	\]
 		\end{Theorem}
				The above result can be strengthened in the following way.
	\begin{Theorem} \label{thm-regular data}
	Assume that $\alpha \geq 0$ and fix  $f \in L^2_{\mathrm{loc}}([0,\infty;D(A_0^\alpha))$.  If the initial  data $(u_0,v_0)\in \mathscr{H}_\alpha$ satisfy conditions \eqref{eq-2.5} and \eqref{eq-2.7}, then
the unique solution
		 to problem  \eqref{eq-2.18-external}, guaranteed by Theorem
		 \ref{theorem 2.1}, satisfies
\[
		(u,v)\in C(\bigl[0,\infty\bigr);\mathscr{H}_\alpha),\;\;\; u\in C^1(\bigl[0,\infty\bigr);H_{\alpha+1}).
\]
\end{Theorem}

Notice that the  proofs of the above two results require the following assertions about the  non-linearities appearing in  equations  \eqref{eq-2.18} and \eqref{eq-2.18-external}.
\begin{Lemma}\label{prop-nonlinearities}
The following function
\begin{align*}
F:\mathscr{H}\ni z=(u,v) \mapsto \norma{H}{ v}{H}^2u+\norma{H}{ A_0u}{H}^2u  \in H,
\end{align*}
is well defined and is a  homogenous continuous polynomial of degree $3$.
In particular, it is Lipschitz-continuous  on balls. Moreover, the same result holds if $\alpha \geq 0$ and the spaces
$H$ and $\mathscr{H}$ are replaced respectively by $D(A_0^\alpha)$ and $\mathscr{H}_\alpha$.
\end{Lemma}
The proof of this  result is obvious and we omit it. We just observe that the corresponding continuous trilinear function is  given by
\begin{align*}
F:\mathscr{H}^3\ni (z^1,z^2,z^3) \mapsto \scalar{H}{v_1}{v_2} u_3+\scalar{H}{A_0\,u_1}{A_0\,u_2} u_3 \in  H,
\end{align*}
where $z^i=(u_i,v_i) \in\,\mathscr{H}$, for $i=1,2,3$.

\subsection{The stochastic constrained wave equation}
\label{sec-SCWEs}

First of all, we need to introduce the diffusion coefficient $\sigma$. We begin with a function
	$\sigma_0:\mathscr{H} \rightarrow \newspaceH$
	which we assume to satisfy the following conditions.
\begin{Hypothesis}
\label{assumption-sigma_0-1}
The mapping
\begin{equation}
\sigma_0:\mathscr{H} \rightarrow \newspaceH
\end{equation}

is  of linear growth and Lipschitz-continuous on balls, i.e.  there exist $L\geq 0$ and a sequence $(L_n)_{n=1}^\infty $ of nonnegative real numbers  such that
	\begin{align}
	\begin{split}   \label{eq-2.21}
	&\Vert \sigma_0(z)\Vert_{\newspaceH}^2 \leq L(1+\norm{ z }{\mathscr{H}}^2),\;\; z\in \mathscr{H}, \\[10pt]
	&\Vert \sigma_0(z_2)-\sigma_0(z_2) \Vert_{\newspaceH}^2 \leq L_n|z_2-z_1|_{\mathscr{H}}^2,\;\; z_1,z_2\in B_{\mathscr{H}}(0,n), \; n \in \mathbb{N}^\ast.
	\end{split} \end{align}
\end{Hypothesis}

\begin{Remark}\label{rem-Stratonovich}
{\em Because of the Stratonovich integral we are going to use, we need to work with the space $\newspaceH$ instead of the usual  $\mathscr{T}_2(K,H)$. Note that in view of assumption \eqref{eqn-K to E}, the former space is naturally emendable into the latter. To be precise, if $L\in \newspaceH$ then $L\circ i \in \mathscr{T}_2(K,H)$ and the corresponding linear map is continuous.\hfill\(\Box\)}
\end{Remark}

	As in the deterministic case, we have to modify $\sigma_0$ by taking its tangential component.
Thus, we define the function
\begin{equation}
{\sigma}:\mathscr{H} \rightarrow \newspaceH
	\end{equation}
by setting for every   $z=(u,v)\in \mathscr{H}$
	\begin{equation}
	\label{eqn-sigma-def}	
	\begin{split}
	\sigma(z)&:= \Pi_u \circ \bigl(\sigma_0(z)\cdot\bigr)= \sigma_0(z)-\scalar{H}{\sigma_0(z)\cdot}{u}\, u
\in \newspaceH,
\end{split}
	\end{equation}
where $\Pi_u$ is the projection defined in \eqref{eq-2.14}.
Note that in general
\[
\sigma \not= \sigma_0.
\]
In the above, we used the following notation  for $ B\in \newspaceH$ and $u_1,u_2 \in H$,
\[
\scalar{H}{B\cdot}{u_1} u_2=\{E \ni k\mapsto\scalar{H}{Bk}{u_1}\, u_2 \in H\} \in \mathscr{L}(E,H).
\]
Similarly, for $ B\in \mathscr{T}_2(K,H)$ and $u_1,u_2 \in H$, we denote
\[
\scalar{H}{B\cdot}{u_1} u_2=\{K \ni k\mapsto\scalar{H}{Bk}{u_1}\, u_2 \in H\} \in \mathscr{L}(K,H).
\]
In the latter case, it is obvious  that  for all  $u_1,u_2 \in H$,   $\scalar{H}{B\cdot}{u} u$ is a bounded linear operator from $K$ to $H$.

Moreover, if $B\in \mathscr{T}_2(K,H)$ and $u_1,u_2\in H$, then
\begin{equation}
\scalar{H}{B\cdot}{u_1}\, u_2\in \mathscr{T}_2(K,H)
\end{equation}
and
		\begin{equation} \label{eq-2.23}
		\Vert \scalar{H}{B\cdot}{u_1} u_2 \Vert_{\mathscr{T}_2(K,H)} \leq \norma{H}{ u_1}{H}\, \norm{ u_2}{H}\Vert B\Vert_{\mathscr{T}_2(K,H)},\;\;\; u_1,u_2\in H.
		\end{equation}
Actually, if
$\{e_j\}_{j \in\,\mathbb{N}}$  is an orthonormal basis of $K$, then  we have
		\begin{align*}
		\sum_{j=1}^{\infty} \norm{\scalar{H}{Be_j}{u_1} \, u_2}{H}^2
		\leq \norma{H}{ u_1}{H}^2 \norma{H}{ u_2}{H}^2 \, \sum_{j=1}^{\infty} \norm{Be_j}{H}^2  =  \norma{H}{ u_1}{H}^2 \norma{H}{ u_2}{H}^2 \,\Vert B\Vert_{\mathscr{T}_2(K,H)}^2.
		\end{align*}

We begin by noticing that, since $\Pi_u$ is an orthogonal projection, as a consequence of the definition \eqref{eqn-sigma-def} for every   $z=(u,v)\in \mathscr{H}$ and every $e \in E$ we have
	\begin{equation}
	\label{eqn-sigma-1}	
	\normb{H}{\sigma(z)e} \leq  \normb{H}{\sigma_0(z)e}.
	\end{equation}
Moreover, $\sigma$ satisfies the following properties.

	\begin{Lemma} \label{lemma 2.2}
	The  function  $\sigma$  is    Lipschitz-continuous on balls and has cubic growth,  as a function defined on $\mathscr{H}$ with values in $\newspaceH$.
More precisely, for every $z =(u,v) \in \mathscr{H}$
\begin{align}\begin{split}
\label{eqn-growth tilde sigma-1-H}
		&\Vert \sigma(z)\Vert_{\newspaceH} \leq  \Vert \sigma_0(z)\Vert_{\newspaceH} \leq L(1+\norm{ z }{\mathscr{H}}),
\\[10pt]
&\Vert \sigma(z)\Vert_{\newspaceH}
		\leq  \sqrt{L}(1+\norm{z}{\mathscr{H}})(1+\normb{D(A_0)}{ u}
\normb{H}{ u}), \;\; .
		\end{split}\end{align}
		Moreover, for every $z\in \newM$ and $e \in\,E$
\begin{equation}   \label{zs1}
  \sigma(z)\,e \in T_uM.
\end{equation}

\end{Lemma}
	
	\begin{proof}
Let us begin by observing that  property \eqref{zs1} is an immediate consequence of definition \eqref{eqn-sigma-def}. Moreover,
 \eqref{eqn-growth tilde sigma-1-H} follows from    \eqref{eq-2.21} and \eqref{eqn-sigma-1}.

Next,  by the definition of the function $\sigma$ we have
\begin{align}\begin{split}
\label{eqn-singma-sigma_0}
\normb{D(A_0)}{ {\sigma }(z)(e) } &\leq \normb{D(A_0)}{ \sigma_0 (z)e} +\vert \scalar{H}{\sigma_0(z)e}{u}  \vert  \normb{D(A_0)}{ u}
\\[10pt]
&\leq \normb{D(A_0)}{ \sigma_0 (z)e} + \normb{H}{\sigma_0(z)e}\normb{H} {u}   \normb{D(A_0)}{ u}
\\[10pt]
&\leq \Vert \sigma_0(z)\Vert_{\mathscr{L}(E,D(A_0))} \norm{e}{E}  +\Vert \sigma_0(z)\Vert_{\mathscr{L}(E,H)} \norm{e}{E} \normb{D(A_0)}{ u}
\normb{H}{ u}
\\[10pt]
&\leq \sqrt{L}(1+\norm{ z }{\mathscr{H}})  \bigl(1+  \normb{D(A_0)}{ u}
\normb{H}{ u} \bigr)\vert e\vert_E,
\end{split}\end{align}
and hence \eqref{eqn-singma-sigma_0} follows.

In order to prove that  the map $ \sigma$ 		is Lipschitz-continuous on balls it is sufficient to prove
that the second term on the RHS of \eqref{eqn-sigma-def}, i.e. the map
\begin{equation}
		 (u,v) \in\,\mathscr{H}  \mapsto  \scalar{H}{\sigma_0(u,v)\cdot}{u} u \in \newspaceH,
  \end{equation}
		is Lipschitz-continuous on balls. For this purpose, if we fix
$		z_i = (u_i,v_i)\in \mathscr{H}$, $i=1,2$, we have
		\begin{align*}
		\scalar{H}{ \sigma_0(z_2)\cdot}{u_2} u_2-&\scalar{H}{ \sigma_0(z_1)\cdot}{u_1} u_1  =\scalar{H}{ (\sigma_0(z_2)-\sigma_0(z_1))\cdot}{u_2} u_2
\\[10pt]
&\hslp+\scalar{H}{ \sigma(z_1)\cdot}{u_2-u_1} u_2 + \scalar{H}{ \sigma_0(z_1)\cdot}{u_1}(u_2-u_1).
		\end{align*}
This implies  that
		\begin{align*}
		\Vert\scalar{H}{\sigma_0(z_2)\cdot}{u_2} u_2&-\scalar{H}{\sigma_0(z_1)\cdot}{u_1} u_1 \Vert_{\newspaceH}
\leq \normb{H}{ u_2}^2 \Vert \sigma_0(z_2)-\sigma_0(z_1)\Vert_{\newspaceH}
\\[10pt]
		&\hsl
		+\Vert  \sigma_0(z_1)\Vert_{\newspaceH}\normb{H}{ u_2-u_1}\left(\norma{H}{ u_1}{H}+\norma{H}{ u_2}{H}\right).
		\end{align*}
Therefore,  since the map $\sigma_0: \mathscr{H} \rightarrow \newspaceH$ is Lipschitz-continuous on balls and of linear growth, the result follows.
	\end{proof}

Once we have defined the diffusion operator $\sigma$, the stochastic version of equation \eqref{eq-2.18} can be written as
	\begin{equation}\label{eq-2.26}
	 u_{tt}(t)+  \norma{H}{ u_t(t)}{H}^2u(t)=-A_0^2u(t)+\norma{H}{ A_0u(t)}{H}^2 \,u(t)-\gamma u_t(t) +\sigma(u(t),u_t(t))\,\circ\, dW(t),
	\end{equation}
	with the initial data satisfying   $(u_0,v_0)\in \mathscr{M}$.
 \begin{Remark} \label{remark 2.3bis}
{\em In this paper we have made the choice to consider the stochastic differential  in the Stratonovich sense.
For general constraints, this ensures that the constraint is preserved for all $t\geq 0$
In the specific case of our equation, where the solution is constrained to the unit sphere in $L^2$, we could have also used the It\^o formulation while still maintaining the constraint. However, since the analysis in this case does not require different arguments or techniques, we focus solely on noise in the Stratonovich sense for the sake of brevity.}

\end{Remark}
	
	\begin{Remark} \label{remark 2.3}
{\em \begin{enumerate}
 \item[1.] The definition we gave in \eqref{eqn-sigma-def} for the function $\sigma$
is not unique. We only require that condition \eqref{zs1} is satisfied. As a matter of fact,
formula \eqref{eqn-sigma-def} is simply one of many that fulfills this assumption. Another one is the following
\begin{equation}\label{eqn-new sigma}
 \widehat{\sigma}(z) = \phi (\norma{H}{ u }{H})\, \sigma(z), \;\;z\in \mathscr{H},
\end{equation}
where  $\phi: [0,\infty) \to [0,\infty)$ is an auxiliary $C_0^\infty$  bump function such that
\[
\phi(r)=\begin{cases}
1, &\mbox{ if }\ \  \norm{r-1}{} \leq \frac14, \\[10pt]
0, &\mbox{ if }\ \  \norm{r-1}{} \geq  \frac12.
\end{cases}
\]

\item[2.] If we assume that the function $\sigma_0$  depends only on the first component $u$ of $z=(u,v)$,
 i.e. there exists a function
	\begin{equation}
	g_0:H \rightarrow \newspaceH \label{eqn-g_0}
	\end{equation}
such that
	\begin{equation}
	\sigma_0(z)=g_0(u), \;\; z=(u,v) \in  \mathscr{H}, \label{eqn-sigma_0-g_0}
	\end{equation}
then we can  	 modify $g_0$ by taking its tangent part. This means that with a slight abuse of notation, we can define
	\begin{align*}
	{g}(u):= \Pi_u(g_0(u))
= g_0(u)-\scalar{H}{g_0(u)\cdot}{u} u  \in \newspaceH, \;\; u \in {H}.
	\end{align*}
Then, the function $\sigma$ associated by formula \eqref{eqn-sigma-def} with the function $\sigma_0$ defined  in formula \eqref{eqn-sigma_0-g_0} satisfies
\[
\sigma(z)=g(u), \;\; z=(u,v) \in  \newspaceH.
\]

In this case, it is possible  to weaken assumptions on $g_0$ and to assume    that $g_0$ is only defined on $M$ and Lipschitz-continuous. Then by the classical Kirszbraun Theorem \cite{Kirszbraun_1934}, see also
\cite{Brz+Rana_2022} and \cite{Brz+Carr_2003}, we can find a (globally) Lipschitz-continuous and bounded extension of $g_0$ from $M$ to the whole $H$.
\item[3.] We could have added a force $f$ term to the above equation as in the deterministic  equation \eqref{eq-2.18-external}, assuming only that $f$ is an $H$-valued progressively measurable process
such that  $f \in L^2_{\mathrm{loc}}([0,\infty; H))$, $\mathbb{P}$-almost surely. But for the sake of simplicity of exposition we have not done so.  Indeed, in this paper we concentrate on different issues.
\hfill\(\Box\)
\end{enumerate}
}
\end{Remark}

  As we already mentioned  in the present paper we decided to study the equation above  in the Stratonovich sense. In fact, we can rewrite the Stratonovich term using the standard It\^o differential,
see e.g. \cite{Brz+Elw_2000}.
To this purpose, with the notations we have introduced above,  if we denote $z=(u, u_t)$, we  rewrite the second order in time equation \eqref{eq-2.26} as a system of two equations of first order in time
\begin{equation}\label{eqn-strong solution}
  \left\{
  \begin{array}{l}
 \ds{ dz(t)=\mathscr{A}_{}z(t)\,dt+ \bigl( -\norma{H}{  u_t(t)}{H}^2u(t)+\norma{H}{ A_0u(t)}{H}^2u(t)  -\gamma\, u_t(t)     \bigr) \, dt
} \\[10pt]
 \ds{\ \ \ \ \ \ \ \ \ \ \  + (0,\sigma(z(t)) \bigr)\,\circ\,dW(t), }\\[10pt]
 \ds{z(0)=(u_0,v_0).}	
  \end{array}\right.
 \end{equation}
 \dela{\[
 \left\{
  \begin{split}
 \ds{ dz(t)&=\mathscr{A}_{}z(t)\,dt+ \bigl( -\norma{H}{ v(t)}{H}^2u(t)+^{-1}\norma{H}{ A_0u(t)}{H}^2u(t)  -^{-1}\gamma v(t)     \bigr) \, dt
} \\[10pt]
 \ds{  &+ (0,\sigma(z(t)) \bigr)\,\circ\,dW(t), }\\[10pt]
 \ds{z(0)&=(u_0,v_0).}	
  \end{split}\right.
  \]}

For a $C^1$-class function $G: \mathscr{H}\to \mathscr{L}(E,\mathscr{H})$, we define, see \cite[Definition 3.1]{Brz+Elw_2000},
\[
\begin{split}
\int_0^t G(z(s)) \,\circ\,dW(s)&:=\int_0^t G(z(s)) \,dW(s)
+\frac12  \int_0^t \tr_{K} \bigl[G^\prime(z(s))G(z(s))\bigr]  \,ds,
\end{split}
\]
with $\tr_K$  defined  as in \eqref{eqn-tr K}.
Note that (see comments after \cite[Definition 3.1]{Brz+Elw_2000})  for all $z \in \mathscr{H}$ we have $d_zG=G^\prime(z) \in \mathscr{L}(\mathscr{H}; \mathscr{L}(E,\mathscr{H}))$ so that
\[G^\prime(z)\,G(z):=d_zG \cdot G(z) \in \mathscr{L}(E; \mathscr{L}(E,\mathscr{H}))\equiv \mathscr{L}_2(E\times E,\mathscr{H}),\] i.e.
\[
G^\prime(z)G(z)(e_1,e_2)=\Bigl[ d_zG\bigl( G(z)e_1\bigr)\bigr]e_2, \;\; (e_1,e_2) \in E\times E.
\]
 This means that $\tr_{K} \bigl[ G^\prime(z(s)G(z(s))\bigr]$ is a well defined element of $\mathscr{H}$ and satisfies
\begin{equation}\label{eqn-tr K-G'G}
\tr_K \Bigl[ G^\prime(z)G(z) \Bigr]= \sum_{i=1}^{\infty}\Bigl[ d_zG\bigl( G(z)e_j\bigr)\Bigr]e_j,
\end{equation}
where $\{e_j\}_{j \in\,\mathbb{N}}$ is   an orthonormal basis  of  $K$. In particular, if
\begin{align}
&G: \mathscr{H} \ni z=(u,v) \mapsto \left\{E \ni k \mapsto  \left(
	0, 	\sigma(z) k \right) \in \mathscr{H} \right\} \in \mathscr{L}(E,\mathscr{H}),
\end{align}
 where $\sigma: \mathscr{H} \to \mathscr{L}(E,H)$, then for every $z=(u,v)$ and $w=(x,y)$ in $\mathscr{H}$ we have the following expression for the Fr\'echet derivative of $G$,
 \begin{align*}
 [d_zG](w)=[d_{(u,v)}G](x,y)&=\bigl(0, [\partial_u \sigma(z)](x)+ [\partial_v \sigma(z)](y)  \bigr),
 \end{align*}
 where
 \begin{align*}
 \partial_u \sigma(z) \in\,  \mathscr{L}   \bigl(D(A_0),\mathscr{L}(E,H)\bigr),\;\;\;
    \partial_v \sigma(z)  \in\,  \mathscr{L}   \bigl(H,\mathscr{L}(E,H)\bigr),
  \end{align*}
 are the directional   Fr\'echet derivatives of function $\sigma$ at $z$.  Therefore, we deduce that for every $z\in \mathscr{H}$ we have
 \[
G^\prime(z) G(z)= [d_zG](G(z))=\bigl(0,  [\partial_v \sigma(z)](\sigma(z))  \bigr),\]
and
 \[ \tr_{K} \bigl[ G^\prime(z)G(z)\bigr] = \tr_{K} \bigl[ [d_zG]G(z) \bigr]= \bigl(0, \tr_{K} \bigl[ \partial_v \sigma(z)\,\sigma(z) \bigr]  \bigr).
 \]
In view of formula \eqref{eqn-tr K-G'G} we have
\begin{equation}\label{eqn-tr K-sigma'sigma}
\tr_K  \Bigl[ \partial_v \sigma(z)\,\sigma(z) \Bigr] = \sum_{i=1}^{\infty} \partial_v \sigma(z) \bigl( \sigma(z)e_j\bigr)e_j.
\end{equation}
Thus, we have the following formula for the Stratonovich integral in equation \eqref{eqn-strong solution}
\[
\begin{split}
\int_0^t &(0,\sigma(z(s)) \bigr)\,\circ\,dW(s)=\int_0^t (0,\sigma(z(s)) \bigr)\,dW(s)
+\frac12  \int_0^t  \Bigl(0,\tr_{K} \bigl[ \partial_v\sigma(z(s))\,\sigma(z(s)) \bigr] \Bigr)\,ds.
\end{split}
\]

The above results explain why we need to make the following additional assumption.
\begin{Hypothesis}
\label{assumption-sigma_0}
The function 	$\sigma_0:\mathscr{H} \rightarrow \newspaceH$ is of $C^1$-class in the sense that the directional   Fr\'echet derivative $\partial_v \sigma_0(z) \in  \mathscr{L}   \bigl(H,\mathscr{L}(E,H)\bigr) $   exists  for every $z \in \mathscr{H}$ and the map
\[
\partial_v \sigma_0: \mathscr{H} \ni z \mapsto   \partial_v \sigma_0(z) \in  \mathscr{L}   \bigl(H,\mathscr{L}(E,H)\bigr),
\]
is continuous.
Moreover, the function
\begin{equation}\label{eqn-Stratonovich correction term}
\mathscr{H} \ni z \mapsto \tr_{K} \bigl[\partial_v\sigma_0(z)\,\sigma(z) \bigr] \in H,
\end{equation}
is Lipschitz-continuous on balls and has linear growth.
\end{Hypothesis}
\begin{Remark}\label{rem-assumption-sigma_0}
{\em In the framework of  Remark \ref{remark 2.3}-2,  we have
\begin{equation}\label{eqn-g-Stratonovich correction term}
\tr_{K} \bigl[\partial_v\sigma(z)\sigma(z) \bigr]=\tr_{K} \bigl[\partial_vg(z)\tilde{g}(z) \bigr]=0, \ \ \ \ z \in \mathscr{M}.
\end{equation}
Thus, in this case we do not need Hypothesis \ref{assumption-sigma_0} and,  instead of \eqref{eqn-g_0},  we can assume that
$g_0:H \rightarrow \mathscr{T}_2(K,H)$.
Moreover, in Theorem \ref{theorem 2.5} we will need to assume that $g_0$ is Lipschitz-continuous on balls and of linear growth, while
in Theorem \ref{thm-2.7-regular data} we will need to assume that the following {\em restriction} map
\[
	g_0:D({A}_0) \rightarrow \mathscr{T}_2(K,D({A}_0)),
	\]
	is  Lipschitz-continuous on balls and  of linear growth. \hfill\(\Box\)}
\end{Remark}

	\begin{Proposition} \label{prop-assumtopn-sigma}
Assume that Hypothesis  \ref{assumption-sigma_0} holds.
	Then the  map $\newsigma:\mathscr{H}
\to \newspaceH$ defined in \eqref{eqn-sigma-def} is of $C^1$-class in the sense of  Hypothesis \ref{assumption-sigma_0}
 and
\[
 \partial_v {\sigma }(z)(y)= \partial_v \sigma_0 (z)(y)- \scalar{H}{\partial_v \sigma_0 (z)(y) \cdot}{u}\, u,
	\;\;\; \ z, y\in \mathscr{H}.
\]
In particular,
\begin{equation}\label{eqn-partial_v tilde sigma}
\begin{split}
 \bigl[\partial_v {\sigma }(z)(y)\bigr]e&= \bigl[\partial_v \sigma_0 (z)(y)\bigl]e
 - \scalarb{H}{ \bigl[\partial_v \sigma_0 (z)(y)\bigl]e,u}\, u,
	\;\;\; \ z,y\in \mathscr{H},\ \ \ \  e\in E.
\end{split}\end{equation}
Moreover, the function
\[
\mathscr{H} \ni z \mapsto \tr_{K} \bigl[\partial_v\newsigma(z)\,\newsigma(z) \bigr] \in \adda{D(A_0)}\dela{H},
\]
 is Lipschitz-continuous on balls and of polynomial growth.	\end{Proposition}
\begin{proof}
The first part of the result follows directly by applying classical results from Cartan's treatise \cite{Cartan_1971_DC}.
\end{proof}

	We shall prove the following stochastic generalisation of Theorem \ref{theorem 2.1}.

	\begin{Theorem} \label{theorem 2.5}
Assume that Hypotheses  \ref{assumption-sigma_0-1} and \ref{assumption-sigma_0} hold.
		Then for every $z_0=(u_0,v_0)\in \mathscr{M}$
there exists a unique solution to the stochastic constrained wave  equation \eqref{eq-2.26}. Namely, there exists a unique
$ \mathscr{M}$-valued continuous   and adapted process $z(t)=(u(t),v(t))$ such that
\begin{enumerate}
\item[1.]
the process $u$ has $M$-valued $C^1$-class trajectories and
\[v(t)=u_t(t),\;\; t \geq 0,\]
\item[2.] the process $z$ is a mild solution of equation \eqref{eq-2.26} with  initial condition $z_0$, i.e.
for every $t\geq 0$, $\mathbb{P}$-almost surely,
\begin{equation}\label{eqn-mild solution}
  \begin{aligned}
  z(t)&=\mathscr{S}(t)z_0+ \int_0^t \mathscr{S}(t-s) \bigl(0,-\norma{H}{ v(s)}{H}^2u(s)+\norma{H}{ A_0u(s)}{H}^2u(s)  -\gamma v(s)
   \bigr) \, ds
  \\[10pt]
  &\hsllp
  + \int_0^t \mathscr{S}(t-s) \bigl(0, \sigma(z(s)) \bigr)\,dW(s)
  +\frac12  \int_0^t  \mathscr{S}(t-s) \bigl(0,\tr_{K} \bigl[\partial_v\sigma(z(s))\,\sigma(z(s)) \bigr]\,ds,
  \end{aligned}
\end{equation}
where   $\mathscr{S}=(\mathscr{S}(t))_{t\in  \mathbb{R}}$  is the
$C_0$ group  in $\mathscr{H}$ generated by $\mathscr{A}$ defined in \eqref{eqn-mathscr A}.
\end{enumerate}
\medskip
Moreover, if
$\Psi$ is the energy function defined in \eqref{eqn-Psi-energy functional},
then the following energy equality holds, for $t\geq 0$, $\mathbb{P}$-almost surely,
\begin{align}
\label{eqn-energy-Psi}  
  \Psi(z(t))&=\Psi(z_0)-\gamma  \int_0^t \norma{H}{ v(s)}{H}^2\, ds + \int_0^t \scalar{H}{v(s)}{\sigma(z(s))\,dW(s)}
 \\[10pt]
 &+\frac{1}{2} \int_0^t \scalar{H}{v(s)}{\tr_K\bigl[\partial_v\sigma_0(z(s))\,\sigma(z(s))]}\, ds +\frac{1}{2} \int_0^t  \Vert \sigma(z(s))\circ i\Vert^2_{\mathscr{T}_2(K,H)}\,ds.
\nonumber
\end{align}
 		\end{Theorem}

We will also prove the following strengthening  of Theorem \ref{theorem 2.5}.

	\begin{Theorem} \label{thm-2.7-regular data}
Suppose that $\sigma_0$, as well as the corresponding  restriction map (denoted by the same symbol $\sigma_0$), map   $\mathscr{H}_1$ into  $\mathscr{L}(E,D({A}_0))$
and are of linear growth and Lipschitz-continuous on balls.
Then, if the initial  condition $z_0$ belongs to $\mathscr{H}_1\cap \mathscr{M}$  the unique solution
		 to problem  \eqref{eqn-mild solution} guaranteed by Theorem
		 \ref{theorem 2.5} satisfies, $\mathbb{P}$-almost surely,
\[
		(u,v)\in C(\bigl[0,\infty\bigr);\mathscr{H}_1).
		\]
\end{Theorem}

Theorems \ref{theorem 2.5} and \ref{thm-2.7-regular data}
will be proven in the next section.
The   strategy of our proof is as follows.
	\begin{itemize}
	\item[-] Instead of equation \eqref{eq-2.26} we will study its mild form \eqref{eqn-mild solution}, which admits a unique local maximal solution.
	\item[-] We will prove that this solution stays on $M$. In other words we will prove that $z$ stays on the tangent bundle $\mathscr{M}$ of $M$.
	\item[-] We will prove that the solution is global by exploiting the energy functional $\Psi$ and using the Hasminski criterion.
\end{itemize}


\begin{Remark}\label{rem-sigma only on M}
{\em It is now quite obvious that our basic object should  not be a function $\sigma_0$ defined on the whole space $\mathscr{H}$ and taking values in  $\mathscr{L}(K,H)$, but a function
\[
\widetilde{\sigma}: \mathscr{M}\to  \mathscr{L}(K,H)
\]
such that for every $(u,v) \in  \mathscr{M}$, the range $R(\widetilde{\sigma}(u,v))$ is a subset of the tangent plane $T_uM$, i.e.
\[
\widetilde{\sigma}(u,v)\,k \subset T_uM,\ \ \ \ (u,v) \in  \mathscr{M},\ \ \ \ k \in\,K.
\]
 Obviously, the last condition is equivalent to the following one
\[
\scalar{H}{\widetilde{\sigma}(u,v)k}{u}= 0,\ \ \ \ (u,v) \in  \mathscr{H},\ \ \ \ k \in\,K.
\]
 We would also need to assume that $\sigma$ satisfies a natural modification of Hypothesis \ref{assumption-sigma_0-1} and  \ref{assumption-sigma_0}.
 In particular, we would need to assume that there exist $L\geq 0$ and a sequence $(L_n)_{n=1}^\infty $ of nonnegative real numbers  such that for all $(u,v)\in \mathscr{M}$	\begin{align*}
	\Vert \widetilde{\sigma}(u,v)\Vert_{\newspaceH}^2 &\leq L(1+\norma{H}{ v}{H}^2),\end{align*}
and for all $(u_1,v_1),(u_2,v_2)\in \mathscr{M}$, with
 $v_1,v_2 \in B_{H}(0,n)$, it holds
 \begin{align*}	\Vert \widetilde{\sigma}(u_2,v_2)-\widetilde{\sigma}(u_1,v_1) \Vert_{\newspaceH}^2 &\leq L_n|(u_2,v_2)-(u_1,v_1)|_{\mathscr{H}}^2.
	\end{align*}

Had we had  decided to follow this path we would only need, purely for the purposes of the proof, to construct an extension
$\sigma$ of $\widetilde{\sigma}$ to the whole  $\mathscr{H}$, with values in $\mathscr{L}(K,H)$,  satisfying
Hypothesis \ref{assumption-sigma_0-1} and \ref{assumption-sigma_0} and such function $\sigma$ would automatically satisfy assertion \eqref{zs1} from Lemma \ref{lemma 2.2}.
In the same vein, in the case the diffusion coefficient does not depend on the second variable,  we could have started with a Lipschitz-continuous function
\[
g : M \to  \mathscr{L}(K,H)
\]
such that
\[
g(u)\,k \subset T_uM,\;\;\;  u \in M,\ \ \ k \in\,K.
\]
 Notice that the last condition is equivalent to requiring that
$
\scalar{H}{g(u)k}{u}= 0$, for all $u \in\,  M$ and $k \in\,K.$
\hfill\(\Box\)}
\end{Remark}

We finish this section by noticing that the above stated results, and in particular Theorem \ref{theorem 2.5}, are also true for the following equation
	\begin{equation}\label{eqn-SCWE-mass}
	\mu u_{tt}(t)+\mu  \norma{H}{ u_t(t)}{H}^2 u(t)=-A_0^2u(t)+\norma{H}{ A_0u(t)}{H}^2 u(t)-\gamma u_t(t) +\sigma(u(t),u_t(t))\,\circ dW(t),
	\end{equation}
where $\mu>0$ is an arbitrary positive constant representing the mass of the object under consideration.
As always, we consider the above with 	 initial conditions   $(u_0,v_0)\in \mathscr{H}$ that satisfy the constraints conditions \eqref{eq-2.5} and \eqref{eq-2.7}.
Obviously, equation \eqref{eqn-SCWE-mass} can be written in the following form
	\begin{equation}\label{eqn-SCWE-mass-2}
	 u_{tt}(t)+\norma{H}{ u_t(t)}{H}^2 u(t)=-\frac{1}{\mu} A_0^2u(t)+\frac{1}{\mu} \norma{H}{ A_0u(t)}{H}^2 u(t)-\frac{\gamma}{\mu} u_t(t) +\frac{1}{\mu} \sigma(u(t),v(t))\,\circ dW(t),
	\end{equation}
which is of the form of equation \eqref{eq-2.26}.

An important difference is that we consider equation \eqref{eqn-SCWE-mass-2} in the spaces $\mathscr{H}$ or $\mathscr{H}_1$ which are independent of $\mu$ and
the linear operator  $\mathscr{A}$  and the $C_0$ group $\mathscr{S}$ on $\mathscr{H}$ introduced above are replaced by
$\mathscr{A}_\mu$  and  $\mathscr{S}_\mu$.
Thus, in order to rigorously define the solution 	to equation \eqref{eqn-SCWE-mass-2} we introduce the  linear operator
$\mathscr{A}_\mu$ in the space $\mathscr{H}$ as follows,
\[\begin{split}
\mathscr{A}_\mu z&=(v,-\mu^{-1}A_0^2\,u), \;\;\; \  z=(u,v)\in D(\mathscr{A}_\mu):=D(A_0^2)\times D(A_0).
\end{split}
\]
It is well known that $\mathscr{A}_\mu$ generates a $C_0$ group $\mathscr{S}_\mu=(\mathscr{S}_\mu(t))_{t\in  \mathbb{R}}$ in $\mathscr{H}$, see e.g. \cite{Brz+Masl+Seidl_2005} and references therein.

 \section{Proofs of Theorem \ref{theorem 2.5} and Theorem \ref{thm-2.7-regular data}
}
\label{sec-proof of existence}
\label{sec3}

By using  the $C_0$ group  $\mathscr{S}(t)$ generated by the operator $\mathscr{A}$ in $\mathscr{H}$, we rewrite equation  \eqref{eq-2.26}  in the  mild form   \eqref{eqn-mild solution},
on the whole space $\mathscr{H}$. We recall that, according to  Lemma \ref{lemma 2.2},  the map $\newsigma$    is Lipschitz-continuous on balls and of cubic growth. Moreover,
by  Proposition \ref{prop-assumtopn-sigma}, the  map $\newsigma$ is of class $C^1$ in the sense of  Hypothesis \ref{assumption-sigma_0} and
the function
\[\mathscr{H} \ni z \mapsto \tr_{K} \bigl[\partial_v\newsigma(z)\,\newsigma(z) \bigr] \in H,\]
  is also Lipschitz-continuous on balls and of polynomial  growth. Finally, by
Proposition \ref{prop-nonlinearities},
the function
\[F:\mathscr{H}\ni z=(u,v) \mapsto \norma{H}{ v}{H}^2u +\norma{H}{ A_0u}{H}^2\,u  \in H,\]
  is   Lipschitz-continuous on balls.
Therefore,  by proceeding in a standard way (compare, for example, \cite[Theorem 1.5]{Seidler_1993}  or
\cite[Theorem 4.10]{Brz_1997}), we can find a unique maximal local mild solution  $z(t)=\bigl(u(t),v(t)\bigr)$ for
equation  \eqref{eqn-mild solution}, defined up to a certain stopping time $\tau$.   In what follows, by following  \cite{Brz+Masl+Seidl_2005} with some important modifications as in \cite{Brz+Dhariwal_2021} and \cite{Brz+Huss_2020}, we will prove that
\[
\mathbb{P}\left(\tau=\infty\right)=1.
\]

We first establish the following fundamental result, c.f. \cite[Theorem 4.1]{Brz+Huss_2020}.

\begin{Proposition} \label{proposition 3.1} The manifold $\mathscr{M}$ is invariant for the process $z(t)$, $t<  \tau$. More precisely,  we have		\begin{equation}
		\scalar{H}{u(t)}{v(t)}=0, \ \ \ \ t\in [0,\tau),\;\;\;\mathbb{P}-\text{a.s}. \label{eq-3.18}
		\end{equation}
\end{Proposition}
Before we embark on the proof, we  state a few essential equalities.

\begin{Lemma} \label{lemma 3.2}
	For every $z=(u,v)\in \mathscr{H}$, with $ u\in D(A_0^2)$, we have
	\begin{align}
	\begin{split}
	\label{eqn-a}
	&\scalar{H}{u}{-A_0^2u+\norma{H}{ A_0u}{H}^2u}=\norma{H}{ A_0u}{H}^2\bigl(\norma{H}{ u}{H}^2-1\bigr),
\\[10pt]
	&\scalar{H}{u}{\sigma(z)e}=-\scalar{H}{u}{ \sigma_0(z)e} \bigl(\norma{H}{ u}{H}^2-1\bigr), \;\;\; e\in E,\\[10pt]
	&\scalar{H}{u}{\tr_K\bigl[\partial_v\sigma(z)\sigma(z)]}=-\scalar{H}{u}{\tr_K\bigl[\partial_v \sigma_0(z) \sigma(z)\bigr]}  \bigl(\norma{H}{ u}{H}^2-1\bigr).
	\end{split}
\end{align}
Moreover, for every  $z\in \mathscr{M}$,  we have
	\begin{align}
\label{eq-3.20}	&\scalar{H}{v}{\sigma(z)\,e}=\scalar{H}{v}{\sigma_0(z)\,e},\;\;\; e\in E,
\end{align}
and
\begin{align}
\label{eq-3.21}
	&\scalar{H}{v}{\tr_K\bigl[\partial_v\sigma(z)\sigma(z)]}=\scalar{H}{v}{\tr_K\bigl[\partial_v\sigma_0(z)(\sigma(z))]}.
\end{align}
\end{Lemma}
\begin{proof}
The first identity in \eqref{eqn-a} is obvious, because $A_0$ is a self-adjoint operator.
The second one and \eqref{eq-3.20} are  straightforward consequences of the definition   \eqref{eqn-sigma-def} of  the function $\sigma$.
 Now, in order to prove the third identity in  \eqref{eqn-a}, we fix $z=(u,v)\in \mathscr{H}$, such that  $u\in D(A_0^2)$, and we define
\[
B:=\sigma(z)\in \mathscr{L}(E,H), \]
and
\[T_0:= \partial_v \sigma_0(z)\in  \mathscr{L}   \bigl(H,\mathscr{L}(E,H)\bigr), \;\;\; \ T:= \partial_v\sigma(z) \in  \mathscr{L}   \bigl(H,\mathscr{L}(E,H)\bigr).
 \]
With these notations,  we can rewrite  formula \eqref{eqn-tr K-sigma'sigma} as follows
\[
\begin{split}
\tr_K  \bigl[ \partial_v \sigma(z)\,\sigma(z) \bigr] &= \sum_{j=1}^{\infty}\,\bigl[ T  B e_j\bigr]e_j.
\end{split}
\]
Similarly we have
\[
\begin{split}
\tr_K  \bigl[ \partial_v \sigma_0(z)\,\sigma(z) \bigr] &= \sum_{j=1}^{\infty}\,\bigl[ T_0  \,B e_j\bigr]e_j.
\end{split}\]
Moreover,  by  \eqref{eqn-partial_v tilde sigma}  we have
\begin{align*}
 &\bigl[ T y\bigr]e= \bigl[  T_0 y\bigl]e
 - \scalarb{H}{\bigl[ T_0 y\bigl]e}{u}\, u,\ \ \ \  y\in H,\ \ \ e\in E,
\end{align*}
and this implies that
\[
\begin{split}
 \scalarb{H}{ u}{ \bigl[ T y\bigr]e} &= \scalarb{H}{ u}{ \bigl[  T_0 y\bigl]e}
 - \scalarb{H}{ u}{ \scalarb{H}{\bigl[ T_0 y\bigl]e}{u}\, u }
 =\scalarb{H}{ u}{ \bigl[  T_0 y\bigl]e}\,\bigl(1 - \norma{H}{ u }{H}^2 \bigr).
 \end{split}
\]
In particular,
\begin{align*}
\scalarb{H}{ u}{ \tr_K  \bigl[ \partial_v \sigma(z)\,\sigma(z) \bigr] }  &=\sum_{j=1}^{\infty}
\scalarb{H}{ u}{\bigl[ T  \, B e_j\bigr]e_j }
\\[10pt]
&\hs=  \bigl(1 - \norma{H}{ u }{H}^2 \bigr)  \sum_{j=1}^{\infty}
\scalarb{H}{ u}{ \bigl[  T_0 \,B (e_j)\bigl]e_j}= \bigl(1 - \normb{H}{ u }^2 \bigr)\scalarb{H}{ u}{ \tr_K  \bigl[ \partial_v {\sigma}_0(z)\,\sigma(z) \bigr] },
\end{align*}
and the third identity  in \eqref{eqn-a} follows.

Finally, assume that $z=(u,v)\in \mathscr{M}$. For every $y\in H$ and $e\in E$, we have
\[
\begin{split}
 \scalarb{H}{ v}{ \bigl[ T y\bigr]e} &= \scalarb{H}{ v}{ \bigl[  T_0 y \bigl]e}
 - \scalarb{H}{ v}{ \scalarb{H}{\bigl[ T_0  y \bigl]e}{u}\, u }
 =\scalarb{H}{ u}{ \bigl[  T_0  y\bigl]e}.
\end{split}
\]
Therefore,
\begin{align*}
&\scalarb{H}{ v} {\tr_K  \bigl[ \partial_v \sigma(z)\,\sigma(z)) \bigr] =\sum_{j=1}^{\infty}
\scalarb{H}{ v}{
\bigl[ T  B e_j\bigr]e_j }} \\[10pt] &\hsp=\sum_{j=1}^{\infty}
\scalarb{H}{ v}{
\bigl[ T_0 B e_j\bigr]e_j }
=
\scalarb{H}{ v}{ \tr_K\bigl[\partial_v \sigma_0(z)\, \sigma(z)\bigr] },
\end{align*}
and identity \eqref{eq-3.21} follows.
\end{proof}

The proof of Proposition \ref{proposition 3.1} will also use the following version of the It\^o Lemma which is a special case of Lemma \ref{lem-Ito-final}.

\begin{Lemma}\label{lem-Ito}
Assume that  a local process $z(t)=(x(t),y(t))$ is a solution to
\begin{equation}
\left\{\begin{array}{l}
\ds{dy(t)=\left[-A_0^2\,x(t)+f(t)\right]\,dt+g(t)\,dW(t), }\\[10pt]
\ds{dx(t)=y(t)dt,}
\end{array}\right.
\label{eq-3.10}
\end{equation}
where all processes are progressively measurable,  $f$ is $H$-valued, $g$ is $\mathscr{T}_2(K,H)$-valued, and $x(t)$ is $D(A_0)$-valued such that    for every $t \geq 0$, 
\begin{equation}\label{eqn-g+f}
\mathbb{E} \int_0^t \left[ \Vert g(s)\Vert^2_{\mathscr{T}_2(K,H)}+\Vert f(s)\Vert^2_{H}\right]  \,ds < \infty
\end{equation}

In other words,  we assume the above  and 
\[
  \begin{split}
  z(t)&=\mathscr{S}(t)z_0+ \int_0^t \mathscr{S}(t-s) \bigl(0,  f(s)
   \bigr) \, ds
+ \int_0^t \mathscr{S}(t-s) \bigl(0, g(s) \bigr)\,dW(s), \ \ \ t\geq 0.
  \end{split}
\]
Then, for every $t \geq 0$, $\mathbb{P}$-almost surely,
\begin{equation}
\begin{split}
\lb x(t),y(t)\rb_H&=\lb x(0),y(0)\rb_H -\int_0^t  \norma{H}{ A_0x(s)}{H}^2\,ds+\int_0^t  \scalar{H}{x(s)}{f(s)}\, ds
\\[10pt]
&\hsp+\int_0^t  \norma{H}{y(s)}{H}^2\,ds+\int_0^t  \scalar{H}{x(s)}{g(s)\,dW(s)}.
\end{split}
\label{eqn-3.11}
\end{equation}
Moreover if $\Psi$ is defined as in \eqref{eqn-Psi-energy functional}, then, for every $t \geq 0$, $\mathbb{P}$-almost surely,
\begin{align}\begin{aligned}
\Psi(z(t))&=\Psi(z(0))+ \int_0^t \scalar{H}{ y(s)}{f(s)}\, ds\\[10pt]
&\hslp +  \frac{1}{2} \int_0^t \Vert g(s)\Vert^2_{\mathscr{T}_2(K,H)}\,ds+  \int_0^t \scalar{H}{ y(s)}{g(s)\,dW(s)}. \label{eq-3.12}
\end{aligned}\end{align}
\end{Lemma}

\begin{proof}[\textbf{Proof of Proposition \ref{proposition 3.1}
}] Let us  fix $k \in \mathbb{N}$.
We will show  that the processes $\varphi$ and $\psi$  defined by
\[\varphi(t)=\frac{1}{2}\bigl( \norma{H}{ u(t)}{H}^2-1\bigr), \ \ \ \  t\in [0, \tau),\]
and
\[\psi(t)=\scalar{H}{u(t)}{v(t)}=\scalar{H}{u(t)}{u_t(t)}, \ \ \ \ t\in [0,\tau),\]
satisfy the  following system of linear stochastic differential equation
\begin{equation}
\label{eq-3.13-psi}
\left\{\begin{split}
d\varphi(t)&=\psi(t)\,dt \\[10pt]
d\psi(t)&+\gamma\,d\varphi(t)=\alpha(t)\varphi(t)\,dt-2 \scalar{H}{ u(t)}{ \sigma_0(z(s)) dW(t)}\varphi(t),\;\;\; t\in [0, \tau),
\end{split}\right.
\end{equation}
for an appropriate process $\alpha(t)$ defined by
\[
\alpha(t)= 2\norma{H}{ A_0u(t)}{H}^2 - 2\norma{H}{ v(t)}{H}^2 - \scalar{H}{u(t)}{ \tr_K\bigl[\partial_v \sigma_0(z(t)) \sigma(z(t))\bigr] }.
\]
Since   $\varphi(0)=0$ and   $\psi(0)=\varphi_t(0)=0$, this implies that   $\varphi(t)=\psi(t)= 0$, and Proposition \ref{proposition 3.1} follows.
Thus, it is sufficient to prove  \eqref{eq-3.13-psi}.

Let us observe that the process $(u(t), v(t))$, $t\geq 0$,  is a solution for system \eqref{eq-3.10}, with
$x=u$, $y=v$,  $g=\newsigma(u,v)$ and with
\[f(s)= -\norma{H}{ v(s)}{H}^2u(s) +\norma{H}{ A_0u(s)}{H}^2-\gamma v(s) +\frac12 \tr_{K} \bigl[\partial_v\newsigma(u(s),v(s))\newsigma(u(s),v(s))\bigr], \ \ \ s\geq 0.\]
Thanks to identity \eqref{eqn-3.11} in Lemma \ref{lem-Ito} and  Lemma \ref{lemma 3.2}, we have
\[
 \begin{split}
\psi(t)-\psi(0)&=
 \int_0^t \bigl(  \norma{H}{ v(s)}{H}^2- \norma{H}{ A_0u(s)}{H}^2\bigr)\, ds \\[10pt]
 &\hsl+ \int_0^t \scalar{H}{u(s)}{-\norm{v(s)}{H}^2 u(s)+ \norma{H}{ A_0u(s)}{H}^2u(s)-\gamma v(s) }\, ds\\[10pt]
& \hsl+\frac{1}{2}\int_0^t \scalar{H}{u(s)}{ \tr_K\bigl[ \partial_v \newsigma(z(s))\newsigma(z(s))]}\, ds
+\int_0^t \scalar{H}{u(s)}{\newsigma(z(s))\,dW(s)} \\[10pt]
 &=\int_0^t  \Bigl( \norma{H}{ A_0u(s)}{H}^2 - \norma{H}{ v(s)}{H}^2 \Bigr) \Bigl( \norma{H}{ u(s)}{H}^2-1\Bigr)\, ds   -\gamma \int_0^t \scalar{H}{u(s)}{ v(s) } \, ds
 \\[10pt]
 &\hsl- \int_0^t \bigl(\norma{H}{ u(s)}{H}^2-1\bigr) \scalar{H}{u(s)}{ \sigma_0(z(s)) \,dW(s)}\\[10pt]
 &\hsl- \frac{1}{2}\int_0^t \scalar{H}{u(s)}{ \tr_K\bigl[\partial_v \sigma_0(z(s)) \sigma(z(s))\bigr] } \bigl(\norma{H}{ u(s)}{H}^2-1\bigr)\, ds
\\[10pt]
 &=\int_0^t  \Bigl( 2\norma{H}{ A_0u(s)}{H}^2 - 2\norma{H}{ v(s)}{H}^2 - \scalar{H}{u(s)}{ \tr_K\bigl[\partial_v \sigma_0(z(s)) \,\sigma(z(s))\bigr] }\Bigr) \varphi(s) \, ds\\[10pt]
 &\hsl-\gamma\int_0^t\psi(s)\,ds-2\int_0^t\langle u(s),\sigma_0(z(s))dW(s)\rangle_H\varphi(s).
\end{split}
 \]
Hence  equality \eqref{eq-3.13-psi} follows.
\end{proof}

	Our next task is to show that the local maximal solution is in fact a global one.
	\begin{Proposition} \label{proposition 3.5}
		We have
		\begin{equation}  \label{zs2}\mathbb{P}\left(\tau=\infty\right)=1.\end{equation}	\end{Proposition}
	\begin{proof}
We define the following stopping times
\[
\tau_k := \inf\{t\in[0,\tau) : |z(t)|_{\mathscr{H}} \geq k\},\;\;\; \ k \in\,\mathbb{N}. \]
According to our definition, if $|z(t)|_{\mathscr{H}} < k$, for every $t \in [0,\tau)$, then $\tau_k=\tau$. The sequence of stopping times $\{\tau_k\}_{k \in\,\mathbb{N}}$ is non decreasing and
\[\tau^\star:=\lim_{k\to\infty}\tau_k\leq \tau.\]
Thus, if we show that for every $t\geq 0$
\begin{equation}
\label{zs3}\lim_{k\to\infty}\mathbb{P}\left(\tau_k\leq t\right)=0,\end{equation}
we conclude that
\[
\mathbb{P}\left(\tau<\infty\right)\leq \mathbb{P}\left(\tau^\star<\infty\right)=0,
	\]
and \eqref{zs2} follows.

	We apply  the It\^o Lemma \ref{lem-Ito-final} to the function $\Psi$ defined  in \eqref{eqn-Psi-energy functional},  see   Lemma \ref{lem-Ito}.
	Since  the  process
	$z(t)=(u(t),v(t))$, $t\in[0,\tau)$
	 is a local solution to problem
\eqref{eq-2.26},
 we infer that
	\begin{align*}
	&d\Psi(z(t))= \scalar{H}{v(t)}{\norma{H}{A_0u(t)}{H}^2u(t)-\norma{H}{ v(t)}{H}^2u(t)-\gamma v(t)} \,dt \\[10pt]
&+ \scalar{H}{v(t)}{\newsigma(z(t))\,dW(t)}+\frac{1}{2} \scalar{H}{v(t)}{\tr_K\bigl[\partial_v\newsigma(z(t))\,\newsigma(z(t))]}\,dt +\frac{1}{2}\Vert \newsigma(z(t))\Vert^2_{\mathscr{T}_2(K,H)}\,dt.
	\end{align*}
	Hence, thanks to \eqref{eq-3.18}  we have
	\begin{align*}\begin{split}
	d\Psi(z(t))&+\gamma \norma{H}{ v(t)}{H}^2\,dt=\scalar{H}{v(t)}{\newsigma(z(t))\,dW(t) } \\[10pt]
	&\hslp+\frac{1}{2}\scalar{H}{v(t)}{\tr_K\bigl[\partial_v\newsigma(z(t))(\newsigma(z))]}\, dt+\frac{1}{2}
\Vert \newsigma(z)\Vert_{\mathscr{T}_2(K,H)}^2\,dt. \end{split}
	\end{align*}
Next note that since $z(t) \in \mathscr{M}$, in view of   \eqref{eq-3.20} and \eqref{eq-3.21}, we have
	\[
\scalar{H}{v(t)}{\newsigma(z(t))\,dW(t) }= \scalar{H}{v(t)}{\sigma_0(z(t))\,dW(t) },\]
and
\[\scalar{H}{v(t)}{\tr_K\bigl[\partial_v\newsigma{v}(z(t))\newsigma(z(t))]}=\scalar{H}{v(t)}{\tr_K\bigl[\partial_v\sigma_0(z(t))\,\newsigma(z(t))]},
\]
so that
	\begin{align*}\begin{split}
	d\Psi(z(t))&+\gamma \norma{H}{ v(t)}{H}^2\,dt=\scalar{H}{v(t)}{\sigma_0(z(t))\,dW(t) } \\[10pt]
	&+\frac{1}{2}\scalar{H}{v}{\tr_K\bigl[\partial_v\sigma_0(z(t))\,\newsigma(z(t))]}\, dt+\frac{1}{2}
\Vert \newsigma(z(t))\Vert_{\mathscr{T}_2(K,H)}^2\,dt. \end{split}
	\end{align*}
Let us also observe that since $z(t) \in \mathscr{M}$, from inequality \eqref{eqn-growth tilde sigma-1-H} in Lemma \ref{lemma 2.2} we deduce that
\[
\frac12 \Vert \newsigma(z(t))\Vert_{\mathscr{T}_2(K,H)}^2 \leq  L(1+\norm{z(t)}{\mathscr{H}}^2),
\]
and, by assumption \eqref{eqn-Stratonovich correction term}, there exists a constant $c>0$ such that
\[
\frac{1}{2}\scalar{H}{v}{\tr_K\bigl[\partial_v\sigma_0(z(t))\newsigma(z(t))]}\, \leq c(1+\norm{z(t)}{\mathscr{H}}^2).
\]

Thus, if we put together  all the estimates above,  we deduce that for $ t <\tau$
\begin{align} \label{eq-3.31}
\Psi(z(t))&\leq  \Psi(z(0)) + \int_0^t \scalar{H}{v(s)}{\sigma_0(z(t))\,dW(s) } +c  \int_0^t \bigl( 1+ \Psi(z(s))\bigr)\,ds.
	\end{align}
In particular, by taking the expectation of both sides of the stopped version of \eqref{eq-3.31}, we infer that   for every $ k\in \mathbb{N}$
\[
\mathbb{E}\bigl(1+\Psi(z(t\wedge \tau_k))\bigr)\leq 1+\psi(z(0))+c\, \mathbb{E}\int _0^{t} \bigl(1+ \Psi(z(s\wedge \tau_k)\bigr)\, ds.
\]
As a consequence of Gronwall's Lemma, this gives
\begin{equation}
 \mathbb{E}\,\bigl(1+\Psi(z(t\wedge \tau_k))\bigr)
 \leq \bigl(1+\Psi(z_0)\bigr) e^{c\,t}, \ \ \  t\geq 0.
\end{equation}
Now, since $\norma{H}{ u(t\wedge \tau_k)}{H}=1$,  this implies that
\[\mathbb{E}\,\norm{ z(t\wedge \tau_k)}{\mathscr{H}}^2\leq c(t),\ \ \ \ \ t\geq 0,\]
so that
\[\lim_{k\to\infty} \mathbb{P}\left(\tau_k<t\right)=0.\]
As we have explained above, this yields \eqref{zs2}.
\end{proof}

The proof of Proposition \ref{proposition 3.5} completes the proof of Theorem \ref{theorem 2.5}.
Thus, we only need to prove Theorem \ref{thm-2.7-regular data}. However, its proof  is very similar to the one of Theorem \ref{theorem 2.5} and for this reason we will only sketch it.

\medskip

 We  fix  $(u_0,v_0)\in \mathscr{M} \cap \mathscr{H}_1$.
Since  we are assuming that $\sigma_0(z) \in\,\mathscr{L}(E,D(A_0))$, if $z \in\,\mathscr{H}_1$,
arguing as in the proof of   Lemma \ref{lemma 2.2}  we can  show the map $\sigma:\mathscr{H}_1\rightarrow \newspace$
is Lipschitz-continuous on balls and of cubic growth. Moreover, arguing as in the proof of
Proposition \ref{prop-assumtopn-sigma}, we can prove that the  map $\newsigma$ is of $C^1$-class in the sense of  Hypothesis \ref{assumption-sigma_0} and
the function
\[\mathscr{H}_1 \ni z \mapsto \tr_{K} \bigl(\partial_v\newsigma(z)\newsigma(z) \bigr) \in D(A_0),\]
is also Lipschitz-continuous on balls and of polynomial  growth. Finally, it follows trivially from
Proposition \ref{prop-nonlinearities},  that the function
\[F:\mathscr{H}_1\ni z=(u,v) \mapsto \norma{H}{ v}{H}^2u + \norma{H}{ A_0u}{H}^2u  \in D(A_0),\]
is  Lipschitz-continuous on balls and of polynomial growth.
Therefore,  by  proceeding in a standard way (compare, for example, \cite[Theorem 1.5]{Seidler_1993}  or
\cite[Theorem 4.10]{Brz_1997}), we can find a unique maximal local mild solution  $z(t)=\bigl(u(t),v(t)\bigr)$, defined for $t< \xi$.   In what follows we will prove that
\begin{equation}  \label{zs5}
\mathbb{P}\left(\xi=\infty\right)=1.\end{equation}
For this aim, we define the following sequence of stopping times
\[
\xi_k := \inf\{t\in[0,\tau) : |z(t)|_{\mathscr{H}_1} \geq k\},\ \ \ \ k \in\,\mathbb{N}. \]

To do this we could  follow the proof of Theorem \ref{theorem 2.5}.  But an easier way is available since by the uniqueness of solutions guaranteed  by Theorem \ref{theorem 2.5} we have

\begin{equation}\label{eqn-z=bf z}
z(t)=\mathbf{z}(t), \;\; t< \xi,
\end{equation}
where $\mathbf{z}$ is the unique global solution from Theorem \ref{theorem 2.5}.

Thus we only need to prove a counterpart of  Proposition \ref{proposition 3.5}, i.e.
	the local maximal solution $z(t)=(u(t),v(t))$, $t<  \xi$, is a global one, i.e. $\xi=\infty$, $\mathbb{P}$-almost surely.
		The proof of this fact follows once we first  apply the  It\^o Lemma \ref{lem-Ito-final} to the following modification of  the function $\Psi$ defined  in \eqref{eqn-Psi-energy functional},
\[	\Phi(z):= \frac{1}{2} \Bigl[\norma{H}{ A_0^2u}{H}^2+\norma{H}{ A_0v}{H}^2\Bigr], \;\;\;z=(u,v)\in \mathscr{H}_1,
	\]
and then apply the Gronwall Lemma. This allows to show that for every fixed $t\geq 0$
\[\lim_{k\to\infty}\mathbb{P}\left(\xi_k\leq t\right)=0,\] and this implies \eqref{zs5}.

\section{The small mass limit: notations, assumptions and main results }
\label{sec4}

Let $D$ be a bounded and smooth domain in $ \mathbb{R}^d$, with $d\geq 1$, and let $H$ denote the Hilbert space $L^2(D)$, endowed with the usual scalar product $\scalar{H}{\cdot}{\cdot}$ and the corresponding norm $\normb{H}{\cdot}$. It is well known that, if $\Delta$ is the Laplace operator on the domain $D$, endowed with the Dirichlet boundary conditions, then there exists a complete orthonormal system $\{e_j\}_{j \in\,\mathbb{N}}\subset H$ and a non-decreasing divergent sequence of positive real numbers $\{\alpha_j\}_{j \in\,\mathbb{N}}$, such that
\[\Delta e_j=-\alpha_j e_j,\ \ \ \ k \in\,\mathbb{N}.\]
For every $\beta \in\,\mathbb{R}$, we denote by $H^\beta$ the  space $D((-\Delta)^\beta)$, endowed with the norm
\[\norm{x}{H^\beta}^2:=\vert (-\Delta)^\beta x\vert_{H}^2=\sum_{j=1}^\infty \alpha_j^\beta|\scalar{H}{ x}{e_j}|^2,
\]and we set  $\mathscr{H}_\beta:=H^{\beta+1}\times H^{\beta}$. When $\beta=0$, we simply denote  $\mathscr{H}_0$ by  $\mathscr{H}$. Moreover, we denote by $M$ is the unit sphere in $H$
\[M=\left\{u \in\,H\ :\,\norma{H}{ u}{H}=1\right\}.\]
Notice that by using interpolation for every $0\leq \vartheta <\varrho$ and $u \in\,H^\varrho \cap M$ we have
\begin{equation}
\label{jan2}
\norm{u}{H^\vartheta}\leq \norm{u}{	H^\varrho}^{\vartheta/\varrho}\, \norma{H}{ u}{H}^{1-\vartheta/\varrho}=\norm{u}{	H^\varrho}^{\vartheta/\varrho}.
\end{equation}

Throughout the rest of this paper, we will consider  the following class of stochastic damped wave equations on $D$

\begin{equation}\label{eq1-infin}\left\{
\begin{array}{l}
\displaystyle{\mu\, \partial_t^2u_\mu(t,\xi)+\mu\norma{H}{ \partial_t u_\mu(t)}{H}^2u_\mu(t,\xi)}\\[10pt]
\displaystyle{\hspace{+1.5truecm}=\Delta u_\mu(t,\xi)+\normb{H}{ \nabla u_\mu(t)}^2u_\mu(t)-\gamma\, \partial_t u_\mu(t,\xi)+\sigma(u_\mu(t))\,\partial_t \WP(t,\xi)}\\[10pt]
\displaystyle{u_\mu(0,\xi)=u_0(\xi),\;\;\;\partial_t u_\mu(0,\xi)=v_0(\xi),\;\;\; \ u_\mu(t,\xi)=0,\ \ \xi \in\,\partial D,}\end{array}
\right.\end{equation}
depending on a positive parameter $\mu$. Here, $\gamma$ is a positive constant, $\WP$ is a  Wiener process on $H$ and the mapping $\sigma:H^1\to\mathscr{L}(H)$ is such that  $\sigma(u)$ projects $H$ onto $T_uM$, for every $u \in\,H^1$. Namely, as in  \eqref{eqn-sigma-def}
\[\sigma(u)=\sigma_0(u)-\langle \sigma_0(u),u\rangle_H u,\]
for some mapping $\sigma_0:H^1\to \mathcal{L}(H)$.

In this section, as well as in all following sections, we assume that $\sigma_0$ depends only on the first component, i.e. the domain of $\sigma_0$  is  $H^1$ and not $H^1\times H$ as in the previous sections. This stronger framework is precisely the one described in Remark \ref{remark 2.3} part (ii) (here we have decided to use the symbol $\sigma_0$ and not $g_0$).
Moreover, this framework has the following consequence. The It\^o-Stratonovich correction term $\tr_{K} \bigl[ \partial_v \sigma(z)\,\sigma(z) \bigr]$, where $\sigma$ is defined in \eqref{eqn-sigma-def}, see also
\eqref{eqn-sigma-def-2}, is equal to $0$. Hence, there is no need of introducing a Banach space $E$ in which the Wiener process takes values. We may simply consider a cylindrical Wiener process $w^Q$ on some separable Hilbert space $K$, called the reproducing kernel Hilbert space.
If this Wiener process takes values in $H$, then its  covariance operator $Q$ belongs to $\mathscr{L}^+(H)$, the space of non-negative and symmetric  operators of  trace class. Note that in this case $K=Q(H)$, so  that $\WP(t,\xi)$ can be formally written as  the sum
\[\WP(t,\xi)=\sum_{j=1}^\infty \tilde{e}_j(\xi)\beta_j(t),\ \ \ \ \ \  t\geq 0,\ \ \ \ \ \xi \in\,D,\]
where  $\{\tilde{e}_j\}_{j \in\,\mathbb{N}}$  is an orthonormal  basis   of $K$
and $\{\beta_k\}_{j \in\,\mathbb{N}}$ is a sequence of mutually independent Brownian motions, all defined on the same stochastic basis $(\Omega, \mathscr{F}, \{\mathscr{F}_t\}_{t\geq 0},\mathbb{P})$. Since $K=Q(H)$ and
\[\scalar{K}{ Q u}{Q^{-1} v}=\scalar{H}{ u}{v},\;\;\; u,v \in\,H, \]
 we can assume that
\[
\tilde{e}_j=Q e_j, \;\; j \in \mathbb{N},
\]
where  $\{e_j\}_{j \in\,\mathbb{N}}$  is an orthonormal basis of $H$. We may assume, although this is not necessary,  that $\{e_j\}_{j \in\,\mathbb{N}}$ diagonalizes the Laplacian $\Delta$.

\dela{  Notice that the reproducing kernel $K$ of the cylindrical Wiener process $\WP$ is a Hilbert space, endowed with the scalar product
Thus, if $\{e_j\}_{j \in\,\mathbb{N}}$ is a complete orthonormal system in $H$, then $\{Qe_j\}_{j \in\,\mathbb{N}}$ is a complete orthonormal system in $K$.
}

\medskip

In what follows we assume a modified version of Hypothesis \ref{assumption-sigma_0-1}, namely we assume that the function $\sigma_0$ depends only on the first variable, see Remark \ref{remark 2.3}(2), where we used an auxiliary notation $g_0$. Since $\sigma_0$ depends only on the first variable, we can relax the assumption by replacing the space $\mathscr{L}(E,H)$
 by the space $\mathscr{T}_2(K,H)$.

\begin{Hypothesis}
\label{hyp-H3} The function $\sigma_0:H^1\to \mathscr{T}_2(K,H)$ is  Lipschitz on balls and
\begin{equation}\label{eqn-c_0}
  \sup_{u \in\,H^1}\Normb{\mathscr{T}_2(K,H)}{\sigma_0(u)} <\infty.
\end{equation}
If $u \in H^2$, then $\sigma_0(u) \in \mathscr{T}_2(K,H^1)$ and the corresponding function
$\sigma_0:H^2\to \mathscr{T}_2(K,H^1)$ is  Lipschitz on balls and
\begin{equation}\label{eqn-c_1}
  \sup_{u \in\,H^2}\Normb{\mathscr{T}_2(K,H^1)}{\sigma_0(u)} <\infty.
\end{equation}
 \end{Hypothesis}
We also assume  the following strengthening of Hypothesis  \ref{hyp-H3}.
		\begin{Hypothesis}\label{hyp-H4}
The function $\sigma_0:H^1\to \mathscr{T}_2(K,H)$ satisfies the following condition
	\begin{equation}  \label{sz102}
	\sup_{u \in\,H^1}\Vert \sigma_0(u)\Vert_{\mathscr{T}_2(K,H)}\,\left(1+\norm{u}{H^1}^2\right)<\infty.\end{equation}
	\end{Hypothesis}

\begin{Remark}\label{not-sigma_0^ast}
{\em In what follows,  by $\sigma_0^\ast(u) \in \mathscr{T}_2(H,K) \subset \mathscr{L}(H,K) $ we will understand the Hilbert adjoint of the operator  $\sigma_0(u) \in \mathscr{T}_2(K,H) \subset \mathscr{L}(K,H)$. \\
Note that in view of  Hypothesis \ref{hyp-H3} the map
\begin{align}
\sigma_0^\ast: H^1 \to  \mathscr{T}_2(H,K) \subset \mathscr{L}(H,K)
\end{align}
is bounded.
Moreover,  for every $ u \in H^1$,
\begin{align}\label{eqn-c_0**}
  &\sup_{u \in\,H^1}\Normb{\mathscr{L}(K,H)}{\sigma_0(u)} <\infty, \ \ \ \ \ \
  \sup_{u \in\,H^1}\Normb{\mathscr{L}(H,K)}{\sigma_0^\ast(u)}<\infty.
\end{align}}
\end{Remark}

We have already seen that
the diffusion coefficient $\sigma$ is given by
\begin{equation}\label{eqn-sigma-def-2}
\sigma(u)h=\sigma_0(u)h-\scalar{H}{ \sigma_0(u)h}{u}u,\ \ \ \ u \in\,H^1,\ \ \ \ h \in\,K.
\end{equation}
In what follows we will use the following useful notation
\begin{equation}\label{eqn-sigma1-def-2}
\sigma_1(u)h:=\scalar{H}{ \sigma_0(u)h}{u}u,\ \ \ \ u \in\,H^1,\ \ \ \ h \in\,K,
\end{equation}
so that \[\sigma(u)=\sigma_0(u)-\sigma_1(u),\ \ \ \ \ u \in\,H^1.\]

We will also assume the following  additional hypothesis.

		\begin{Hypothesis}\label{hyp-H5}
If $u \in H^3$, then $\sigma_0(u) \in \mathscr{T}_2(K,H^2)$ and there exists $c>0$ such that
	\begin{equation}
	\label{sz191}
	\Vert \sigma_0(u)\Vert_{\mathscr{T}_2(K,H^2)}\leq c\left(1+\norm{u}{H^2}\right),\ \ \ \ u \in\,H^3.
	\end{equation}
\end{Hypothesis}

\dela{Something is wrong? }
\dela{	
	Notice, that thanks to Hypothesis \ref{hyp-H5} we have that $\sigma_0(u)Q \in\,\mathscr{L}(K,H^2)$, for every $u \in\,H^2$, and
	\begin{equation}   \label{sz220-bis}\Vert \sigma_0(u)Q\Vert_{ \mathscr{L}(K,H^2)}\leq \Vert \sigma_0(u)Q\Vert_{ \mathscr{T}_2(K,H^2)}\leq c\,\left(1+\norm{u}{H^2}\right),
	\end{equation}
so that for every $h \in\,H$ and $u \in\,H^2\cap M$ and $v \in\,H^2$

\coma{
\begin{equation}
\label{sz180-bis}
\left|\scalar{H^2}{ \sigma(u)Qh}{v}\right|\leq c\left(1+\norm{u}{H^2}\right)\vert v\vert_{H^2}\,\norma{H}{ h}{H}.
\end{equation}
}}

\begin{Remark}
{\em Assume that for every $u \in\,H^1$
\[\sigma_0(u)\tilde{e}_k=\lambda_k(\norm{u}{H^1})e_k,\ \ \ \ k \in\,\mathbb{N},\]
for some mappings $\lambda_k:[0,\infty)\to\mathbb{R}$ such that
for every $R>0$
\[r_1, r_2 \in\,[0,R]\Longrightarrow 	\vert \lambda_k(r_1)-\lambda_k(r_2)\vert\leq c_{R,k}\,\vert r_1-r_2\vert.
\]
For every $\delta\geq 0$ and $u \in\,H^1$ we have
\[\begin{array}{l}
\ds{\Vert \sigma_0(u)\Vert^2_{\mathscr{T}_2(K,H^\delta)}=\sum_{k=1}^\infty \vert \sigma_0(u)\tilde{e}_k\vert_{H^\delta}^2=\sum_{k=1}^\infty\lambda_k^2(\norm{u}{H^1})\alpha_k^\delta,}
\end{array}\]
and
for every $u_1, u_2 \in\,B_R(H^1)$, we have
\[\begin{array}{l}
\ds{\Vert \sigma_0(u_1)-\sigma_0(u_2)\Vert^2_{\mathscr{T}_2(K,H^1)}=\sum_{k=1}^\infty \vert \left[\sigma_0(u_1)-\sigma_0(u_2)\right]Q e_k\vert_{H^1}^2}\\[14pt]
\ds{=\sum_{k=1}^\infty \left[\lambda_k(\vert u_1\vert_{H^1})-\lambda_k(\vert u_2\vert_{H^1})\right]^2\alpha_k\leq \sum_{k=1}^\infty c_{R,k}^2\alpha_k.}	\end{array}\]
In particular, if
\[\Lambda_1:=\sup_{r\geq 0}\sum_{k=1}^\infty \lambda_k(r)\alpha_k<\infty,\ \ \ \ c_R:=\sum_{k=1}^\infty c_{R, k}^2\,\alpha_k<\infty,\]
we have that
\[\sup_{u \in\,H^1}\Vert \sigma_0(u)\Vert^2_{\mathscr{T}_2(K,H^1)}\leq \Lambda_1,\]
and
\[u_1, u_2 \in\,B_R(H^1)\Longrightarrow
\Vert \sigma_0(u_1)-\sigma_0(u_2)\Vert^2_{\mathscr{T}_2(K,H^1)}\leq c_R,\]
so that
Hypothesis \ref{hyp-H3} follows.
Next, if we assume
\[\Lambda_2:=\sup_{r\geq 0}\frac 1{1+r^4}\sum_{k=1}^\infty \lambda^2_k(r)\alpha_k^2<\infty,\]
due to \eqref{jan2} we have
\[\Vert \sigma_0(u)\Vert^2_{\mathscr{T}_2(K,H^2)}=\sum_{k=1}^\infty\lambda_k^2(\norm{u}{H^1})\alpha_k^2\leq \Lambda_2\left(1+\norm{u}{H^1}^4\right)\leq \Lambda_2\left(1+\norm{u}{H^2}^2\right).\]
Moreover, if we assume
\[\Lambda_3:=\sup_{r\geq 0}\sum_{k=1}^\infty \lambda^2_k(r)(1+r^4)<\infty,\]
we get
\[\sup_{u \in\,H^1}\Vert \sigma_0(u)\Vert^2_{\mathscr{T}_2(K,H)}\left(1+\norm{u}{H^1}^4\right)\leq \Lambda_3.\]
All this implies that Hypotheses \ref{hyp-H4} and \ref{hyp-H5} hold. \hfill\(\Box\)}	\end{Remark}

The following result (as well as its proof) is similar to Lemma \ref{lemma 2.2}.

\begin{Lemma}\label{lem-basic estimates on sigma}
Assume that $K$  is a separable Hilbert space.
Assume that the functions $\sigma$ and $\sigma_1$ are defined by formulae  \eqref{eqn-sigma-def-2} and \eqref{eqn-sigma1-def-2} respectively, where
$\sigma_0:H^1 \to \mathscr{T}_2(K,H)$.
\begin{itemize}
  \item[1.]
  If $u \in H^1$, then
\begin{equation*}
\Vert \sigma(u)\Vert_{\mathscr{T}_2(K,H)}^2 \leq \Vert \sigma_0(u)\Vert^2_{\mathscr{T}_2(K,H)},
\end{equation*}
and, if the map $\sigma_0:H^1 \to \mathscr{T}_2(K,H)$  satisfies \eqref{eqn-c_0},  then
\begin{equation}
	\label{sz5}
\sup_{u \in\,M\cap H^1}\Vert \sigma(u)\Vert_{\mathscr{T}_2(K,H)}<\infty.		\end{equation}
Moreover, if Hypothesis \ref{hyp-H4} holds, then
\begin{equation}
	\label{sz5-bis}
\sup_{u \in\,M\cap H^1}\Vert \sigma(u)\Vert_{\mathscr{T}_2(K,H)}\left(1+\vert h\vert_{H^1}^2\right)<\infty,		\end{equation}
\item[2.] If $X \subset H^1$ is a Hilbert space, then $\sigma_1$ maps $X$ into $\mathscr{T}_2(K,X)$ and for  every  $u\in X$,
\begin{align}
\Normb{\mathscr{T}_2(K,X)}{\sigma (u)}  & \leq
 \Normb{\mathscr{T}_2(K,X)}{\sigma_0 (u)} +
 \normb{X}{u}\normb{K}{ \sigma_0^\ast(u) u}.
\end{align}
In particular, for  every  $u\in X \cap M$,
\begin{align}\label{eqn-sigma_1-d}
\Normb{\mathscr{T}_2(K,X)}{\sigma (u)}  & \leq
 \Normb{\mathscr{T}_2(K,X)}{\sigma_0 (u)} +
 c\, \normb{X}{u}.
\end{align}

  \item[3.]
Under Hypothesis \ref{hyp-H3} there exists some $c>0$ such that
\begin{equation}
	\label{sz120}
	\Vert \sigma(u)\Vert_{\mathscr{T}_2(K,H^1)}\leq c\left(1+\norm{u}{H^1}\right),\ \ \ \ u \in\,H^2\cap M.
\end{equation}
\item[4.] Under Hypothesis \ref{hyp-H5}, there exists some $c>0$ such that
\begin{equation}
\label{sz190}	
\Vert \sigma(u)\Vert_{\mathscr{T}_2(K,H^2)}\leq c\left(1+\norm{u}{H^2}\right),\ \ \ \ \ \ u \in\,H^3\cap M.
\end{equation}
 \end{itemize}
\end{Lemma}

\begin{proof}
For every $u \in H$, we have
\begin{equation}\label{eqn-4.14}
\normb{H}{ \sigma(u)\tilde{e}_k}^2 \leq \normb{H}{ \sigma_0(u)\tilde{e}_k}^2, \ \ \ \ \ k \in\,\mathbb{N}.
\end{equation}
Hence,  by summing the expression above over $k \in\,\mathbb{N}$,  we obtain
\[\Vert \sigma(u)\Vert_{\mathscr{T}_2(K,H)}^2 \leq \Vert \sigma_0(u)\Vert^2_{\mathscr{T}_2(K,H)}
\]
and this implies \eqref{sz5} and, in case Hypothesis \ref{hyp-H4} holds, \eqref{sz5-bis}.

Next  if $u \in H$, by the Parseval identity in $K$ we get
\begin{align}
\Normb{\mathscr{T}_2(K,X)}{\sigma_1 (u)}^2 &=  \sum_{k=1}^\infty  \normb{X}{ \scalar{H}{\sigma_0(u)\tilde{e}_k}{u} u}^2
\\
 &=
\normb{X}{u}^2 \sum_{k=1}^\infty   \scalar{K}{\tilde{e}_k}{\sigma_0^\ast(u)  u}^2=  \normb{X}{u}^2 \normb{K}{ \sigma_0^\ast(u) u}^2,
 \end{align}
and, recalling that $\sigma(u)=\sigma_0(u)-\sigma_1(u)$, this proves \eqref{eqn-sigma_1-d}.

Now, if we assume  Hypothesis \ref{hyp-H3}, for any $u \in H^1$ \eqref{eqn-sigma_1-d} gives
\begin{align}
\Normb{\mathscr{T}_2(K,H^1)}{\sigma (u)} & \leq  c\left(1+\normb{H}{u} \normb{H^1}{u}\right),
\end{align}
and \eqref{sz120} follows.

Finally, under Hypothesis \ref{hyp-H5}, we have
\begin{align}
\Normb{\mathscr{T}_2(K,H^2)}{\sigma (u)}^2 & \leq   c\left(1+ \normb{H^2}{u}^2\right)+c\,\vert u\vert_H^2\normb{H^2}{u}^2,
\end{align}
and thus  \eqref{sz190} follows.

\dela{
    Next, if $u \in\,H^\delta\cap M$ (with $\delta=1,2$) and $k \in\,\mathbb{N}$, we have
\begin{equation} \label{sz221}
\begin{array}{ll}
\displaystyle{\normb{H^\delta}{\sigma(u)Q e_j}^2=}  &  \ds{\normb{H^\delta}{\sigma_0(u)Q e_j}^2-2
\scalar{H}{e_j}{ [\sigma_0(u)Q]^\star u}  \scalar{H}{ e_j}{ [\sigma_0(u)Q]^\star (-\Delta)^\delta u}}\\[14pt]
&\ds{+|\scalar{H}{ e_j}{ [\sigma_0(u) Q]^\star u}|^2\,\norm{u}{H^\delta}^2.}
\end{array}	
\end{equation}
We have
\begin{align} \begin{split}  \label{sz230}
\sum_{j=1}^\infty &\scalar{H}{  e_j}{ [\sigma_0(u)Q]^\star u} \scalar{H}{ e_j}{ [\sigma_0(u)Q]^\star (-\Delta)^\delta u}=\scalar{H^\delta}{ [\sigma_0(u)Q]	[\sigma_0(u)Q]^\star u}{u}\\[10pt]
&\leq \Vert \sigma_0(u)\Vert_{\mathscr{L}(K,H^\delta)}\norma{H}{[\sigma_0(u)Q]^\star u}{H}\norm{u}{H^\delta}\leq \Vert \sigma_0(u)\Vert_{\mathscr{L}(K,H^\delta)}\Vert [\sigma_0(u)Q]^\star\Vert_{\mathscr{L}(H)}\norm{u}{H^\delta}\\[10pt]
&\hsp\leq \Vert \sigma_0(u)\Vert_{\mathscr{L}(K,H^\delta)}\Vert \sigma_0(u)\Vert_{\mathscr{L}(K,H)}\norm{u}{H^\delta}.
\end{split}
\end{align}
Thus,  if we sum both sides in \eqref{sz221} with respect to $k \in\,\mathbb{N}$,  thanks again to \eqref{sz5} we obtain
\begin{align*}
\Vert \sigma(u)&\Vert_{\mathscr{T}_2(K,H^\delta)}^2  \\[10pt]
&\leq \Vert \sigma_0(u)\Vert_{\mathscr{T}_2(K,H^\delta)}^2+c\,\Vert \sigma_0(u)\Vert_{\mathscr{L}(K,H^\delta)}\Vert \sigma_0(u)\Vert_{\mathscr{L}(K,H)}\norma{H}{ u\vert_{H^\delta}+|[\sigma_0(u)Q]^\star u}{H}^2 \vert u\vert^2_{H^\delta}\\[10pt]
& \leq \Vert \sigma_0(u)\Vert_{\mathscr{T}_2(K,H^\delta)}^2+c\,\Vert \sigma_0(u)\Vert_{\mathscr{L}(K,H^\delta)}\Vert \sigma_0(u)\Vert_{\mathscr{L}(K,H)}\norm{u}{H^\delta}+c\, \Vert \sigma_0(u)\Vert_{\mathscr{L}(K,H)}^2\vert u\vert^2_{H^\delta}.
\end{align*}
	
	Now, if we take $\delta=1$ and only assume  the boundedness of the mapping $\sigma_0:H^1\to\mathscr{T}_2(K,H^1)$, we obtain \eqref{sz120}. Moreover, if we assume also
Hypothesis \ref{hyp-H4}
, we obtain \eqref{sz5-bis}.
	Finally, if we take $\delta=2$ and, in addition to Hypothesis \ref{hyp-H4}  assume also \eqref{sz191}, we 	get \eqref{sz190}.
}
\end{proof}

\medskip

Equation  \eqref{eq1-infin} can be rewritten as the system
\begin{equation}
\label{system}
\left\{
\begin{array}{l}
\ds{du_\mu(t)=v_\mu(t)\,dt}\\[10pt]
\ds{\mu dv_\mu(t)=\left[\Delta u_\mu(t)+\vert u_\mu(t)\vert_{H^1}^2u_\mu(t)	-\mu\norma{H}{ v_\mu(t)}{H}^2 u_\mu(t)-\gamma v_\mu(t)\right]\,dt}\\[10pt]
\ds{\hspace{+2truecm}+\sigma(u_\mu(t))\,d\WP(t),}\\[10pt]
\ds{u_\mu(0)=u_0,\ \ \ \ v_\mu(0)=v_0,\ \ \ \ u_\mu(t)_{|_{\partial D}}=0.}
\end{array}
\right.	
\end{equation}
Thus, if we take

\begin{align}
\label{eqn-6.02}
A_0^2&=-\Delta, \; H=L^2(D),\; D(A_0^2)= H^2(D),
\end{align}
we have
\begin{equation}
\normb{H}{A_0 u}= \normb{H}{\nabla u},
\end{equation}
and  problem \eqref{eq1-infin} is precisely problem \eqref{eqn-strong solution}. 
Moreover,  if $\sigma_0:H^1\to \mathscr{L}_2(K,H)$ has linear growth and is Lipschitz continuous, then Hypothesis \ref{assumption-sigma_0-1} is satisfied, and, since $\sigma_0$ is independent of $v$, Hypothesis \ref{assumption-sigma_0}  is  satisfied as well.
In particular, thanks to Theorem \ref{theorem 2.5} we have the following result.

\begin{Theorem}\label{thm-existence-example}
Assume that the function $\sigma_0:H^1 \rightarrow \mathscr{T}_2(K,H)$
has linear growth and is Lipschitz-continuous on balls. Then, for every  $z_0=(u_0,v_0)\in \mathscr{M}$,
there exists a unique solution to the stochastic constrained wave  equation \eqref{system}, i.e.
an   $ \mathscr{M}$-valued continuous   and an adapted process $z(t)=(u(t),v(t))$ such that
\begin{enumerate}
\item[1.]
the process $u$ has $M$-valued $C^1$ trajectories and
\[v(t)=\partial_tu(t),\ \ \ \ \ \ t \geq 0;\]
\item[2.] the process $z$ is a mild solution of equation \eqref{eq-2.26} with  initial conditions $(u_0,v_0)$, i.e.
for every $t\geq 0$, $\mathbb{P}$-almost surely,
\[
  \begin{split}
  z(t)&=\mathscr{S}_\mu(t)z_0+ \frac 1\mu\int_0^t \mathscr{S}_\mu(t-s) \bigl(0,-\mu\,\vert v(s)\vert^2u(s)+\vert \nabla u(s)\vert^2u(s)  -\gamma v(s)
   \bigr) \, ds
  \\[10pt]
  &\hsp+ \frac 1\mu\int_0^t \mathscr{S}_\mu(t-s) \bigl(0, \sigma(u(s)) \bigr)\,d\WP(s),
  \end{split}
\]
where   $\mathscr{S}_\mu=(\mathscr{S}_\mu(t))_{t\in  \mathbb{R}}$  is the
$C_0$ group  in $\mathscr{H}$ generated by $\mathscr{A}_\mu$.
\end{enumerate}
Moreover, the process $z$ satisfies the following energy equality, for $t\geq 0$, $\mathbb{P}$-almost surely,
\begin{equation}\label{eqn-energy equality}
\begin{split}
  \frac 12\,\norm{u(t)}{H^1}^2+\frac {\mu}2 \vert v(t)\vert^2_H&=\frac 12\,\vert u_0\vert_{H^1}^2+\frac {\mu}2 \vert v_0\vert^2_H-\gamma  \int_0^t \norma{H}{ v(s)}{H}^2\, ds
 \\[10pt]
 &\hs+ \int_0^t \scalar{H}{v(s)}{\sigma(u(s))\,d\WP(s)}+\frac{1}{2\mu } \int_0^t  \Vert \sigma(u(s))\Vert^2_{\mathscr{T}_2(K,H)}\,ds.
\end{split}
\end{equation}
Finally, if Hypothesis \ref{hyp-H3} is satisfied, and
if $z_0=(u_0,v_0)\in\,\mathscr{H}_1$,
  the above unique solution $z$ belongs to $C(\bigl[0,\infty\bigr);\mathscr{H}_1)$,
		  $\mathbb{P}$-almost surely.
		  \end{Theorem}

In what follows, we will study the asymptotic behavior of $u_\mu$, when the parameter $\mu$ goes to zero and we will prove that the following diffusion approximation result holds.

\begin{Theorem}
\label{teo8.3}
Assume Hypotheses \ref{hyp-H3}, \ref{hyp-H4} and \ref{hyp-H5} and fix $(u_0,v_0) \in\,\mathscr{H}_2\cap \mathscr{M}$.  Then, for every  $\alpha \in\,[0,2)$ and $q<4/\alpha$ and every $T>0$, we have
\begin{equation}   \label{sz-fine}\lim_{\mu\to 0} \mathbb{P}\left (\vert u_\mu-u\vert_{L^q(0,T;H^\alpha)}>\eta\right)=0,\ \ \ \ \ \eta>0,\end{equation}
where $u \in\,L^2(\Omega;L^4(0,T;H^1\cap M)\cap L^2(0,T;H^2))$ is the unique solution of the equation
\begin{equation}
\label{limit-eq}
\left\{\begin{array}{l}
\ds{\gamma\,\partial_t u(t,\xi)=\Delta u(t,\xi)+\norm{u(t)}{H^1}^2 u(t,\xi)-\frac 12 \Vert \sigma(u(t))\Vert_{\mathscr{T}_2(K,H)}^2 u(t) +\sigma(u(t))\partial_t \WP(t,\xi),}\\[18pt]
\ds{u(0,\xi)=u_0(\xi),\ \ \ \ u(t,\xi)=0,\ \ \ \xi \in\,\partial D.	}
\end{array}\right.
	\end{equation}
	\end{Theorem}
	
	 \begin{Remark}  {\em 	Although the limit of \( u_\mu \) to \( u \) lies in \( L^q(0,T,H^\alpha) \) for \( \alpha < 2 \), in Theorem \ref{teo8.3}, we require that \( (u_0, v_0) \) belongs to \( \mathcal{H}_2 = H^3 \times H^2 \). This requirement arises because, to establish the tightness of \( \{u_\mu\}_{\mu \in (0,1)} \) and validate the limit \eqref{sz-fine}, we need a priori bounds in the \( H^2 \)-norm. Given the nature of the equation satisfied by \( u_\mu \), such bounds follow from additional estimates on \( (u_\mu, \sqrt{\mu} \,\partial_t u_\mu) \) in \( \mathcal{H}_2 \).

However, in this context we note that \eqref{limit-eq} remains valid even if the initial conditions \( (u_0^\mu, v_0^\mu) \in \mathcal{H}_2 \cap \mathcal{M} \) depend on \( \mu \), provided that the following conditions hold
\[
\lim_{\mu\to 0} \vert (u^\mu_0,v^\mu_0)-(u_0,v_0) \vert_{\mathcal{H}_1} = 0,
\]
for some \( (u_0, v_0) \in \mathcal{H}_1 \), and
\[
\lim_{\mu\to 0} \sqrt{\mu}\, \vert (u^\mu_0, v^\mu_0) \vert_{\mathcal{H}_2} = 0.
\]
}
	\end{Remark}

	\section{A few comments about the limiting equation \eqref{limit-eq}}
	\label{sec5}
	
	In \cite{Brz+Huss_2020}  it is proven that for every $T>0$ there exists a unique mild solution
	\[u \in\,L^2(\Omega;C([0,T];H^1\cap M))\] for the {\em constrained} parabolic equation
	\begin{equation}
\label{limit-eq-bis}
\left\{\begin{array}{l}
\ds{\gamma\,\partial_t u(t,\xi)=\Delta u(t,\xi)+\norm{u(t)}{H^1}^2 u(t,\xi) +\sigma(u(t))\circ\partial_t \WP(t,\xi),}\\[18pt]
\ds{u(0,\xi)=u_0(\xi),\ \ \ \ u(t,\xi)=0,\ \ \ \xi \in\,\partial D.	}
\end{array}\right.
	\end{equation}
As we have seen above, equation \eqref{limit-eq-bis} can be rewritten in terms of It\^{o}'s integral as
\begin{equation}
	\label{sz-2-bis}
	\left\{\begin{array}{l}
\ds{\gamma\,\partial_t u(t,\xi)=\Delta u(t,\xi)+\norm{u(t)}{H^1}^2 u(t,\xi)+\frac 12\, \tr_K[\sigma^\prime(u(t))\sigma(u(t))] +\sigma(u(t))\partial_t \WP(t,\xi),}\\[18pt]
\ds{u(0,\xi)=u_0(\xi),\ \ \ \ u(t,\xi)=0,\ \ \ \xi \in\,\partial D.	}
\end{array}\right.
\end{equation}

The same arguments used in \cite{Brz+Huss_2020} for equation \eqref{limit-eq-bis} (or, equivalently, equation \eqref{sz-2-bis}) can be adapted to prove the well-posedness of the limiting equation \eqref{limit-eq} from Theorem \ref{teo8.3}. 
 Namely, for every $u_0 \in\,H^1\cap M$ and $T>0$ there exists a unique adapted process $u \in\,L^2(\Omega;C([0,T];H^1\cap M))$ such that 
for every   $\psi\in C^\infty_0(D)$ and $t \in\,[0,T]$
\begin{align}
\begin{split}
\gamma\scalar{H} { &u(t)}{\psi} =\gamma\scalar{H}{ u_0}{\psi} -\int_0^t \scalar{H}{ \nabla u(s)) }{\nabla \psi}\,ds\\
&\quad \quad \quad \quad \quad \quad\quad \quad +\int_0^t  \vert \nabla u(s)\vert_{H}^2\scalar{H}{u(s)}{\psi}\, ds +\int_0^t\scalar{H}{ \sigma(u(s))dw^Q (s)}{\psi}.
\end{split}
\end{align}

However, as we will show in the example we are providing below,  equation \eqref{sz-2-bis} and equation \eqref{limit-eq} are different,
 as well as their respective solutions.  This fact is somehow unexpected and shows how, as a consequence of the Smoluchowski-Kramers approximation of a damped stochastic wave equation, a new stochastic parabolic equation  satisfying the same constraints as equation \eqref{sz-2-bis} is obtained. In particular, all this  poses the intriguing question whether different stochastic parabolic equations can still describe a  motion confined to the unitary sphere of $L^2$ (to this purpose see also \cite{hr}).

 \medskip

Let $K=\mathbb{R}$ and let
\[\sigma_0(u):=g(\norm{u}{H^1}^2)h,\;\;\; u \in\,H^1,\]
where $g(t)=(1+t)^{-1}$, and $h \in\,H^2$, with $\norma{H}{ h}{H}=1$. It is immediate to check that $\sigma_0$ satisfies  Hypotheses \ref{hyp-H3}, \ref{hyp-H4} and \ref{hyp-H5}.
If we define
\[\sigma(u)=\sigma_0(u)-\scalar{H}{ \sigma_0(u)}{u} u,\ \ \ \ u \in\,H^1,\]
we have the following identity.
\begin{Lemma}
\label{lemma9.1}
For every $u \in\,H^1\cap M$, we have
\begin{equation}
\label{szfine1}
\sigma^\prime(u)\sigma(u)+\norma{H}{ \sigma(u)}{H}^2 u=\Lambda(u) \left(1+\vert u\vert_{H^1}^2\right)^{-3},	\end{equation}
where
\begin{align}\begin{split}
\label{szfine2}
\Lambda(u)&=\left(\scalar{H}{ u}{h}\,\norm{u}{H^1}^2-2\scalar{H^1}{ u}{h}-\scalar{H}{ u}{h}\right)\left(h-\langle u,h\rangle_H u\right).
\end{split}  \end{align}
\begin{proof}

The mapping $\sigma_0:H^1\to H$ is differentiable, so that also the mapping $\sigma:H^1\to H$ is differentiable and for every $u \in\,H^1\cap M$ it holds
\begin{align*}
	\sigma^\prime(u)&\sigma(u)=\sigma_0^\prime(u)\sigma_0(u)-\scalar{H}{ \sigma_0(u)}{u}\sigma_0^\prime(u)u-\scalar{H}{ \sigma_0^\prime(u)\sigma_0(u)}{u} u\\[10pt]
	&\hslp +\scalar{H}{ \sigma_0(u)}{u} \scalar{H}{ \sigma_0^\prime(u)u}{u} u-\norma{H}{ \sigma_0(u)}{H}^2u+2\left|\scalar{H}{ \sigma_0(u)}{u}\right|^2u -\scalar{H}{ \sigma_0(u)}{u}\sigma_0(u).\end{align*}
Since for every $u \in\,H^1\cap M$ and $v \in\,H^1$ we have
\begin{align*}
\sigma_0^\prime(u)v=2\,g^\prime(\norm{u}{H^1}^2)	\scalar{H^1}{ v}{u}h=-2\,g^2(\norm{u}{H^1}^2)	\scalar{H^1}{ v}{u}h,
\end{align*}
this gives
\begin{align*}
\sigma^\prime&(u)\sigma(u)=-2\,g^3(\norm{u}{H^1}^2)\,\scalar{H^1}{ h}{u}h+2\,g^3(\norm{u}{H^1}^2)\,\scalar{H}{ h}{u} \norm{u}{H^1}^2 h\\[10pt]
&+2\,g^3(\norm{u}{H^1}^2)\,\scalar{H^1}{ h}{u}\scalar{H}{ h}{u} u-2\,g^3(\norm{u}{H^1}^2)\,\left|\scalar{H}{ h}{u}\right|^2\norm{u}{H^1}^2u\\[10pt]
&-g^2(\norm{u}{H^1}^2)u+g^2(\norm{u}{H^1}^2)\left|\scalar{H}{ h}{u}\right|^2 u-g^2(\norm{u}{H^1}^2)\scalar{H}{ h}{u}h+g^2(\norm{u}{H^1}^2)\vert \scalar{H}{ h}{u}\vert^2 u,
\end{align*}
so that
\begin{align}
	\begin{split}
	\label{szfine4}
\sigma^\prime&(u)\sigma(u)=	g^3(\norm{u}{H^1}^2)\left(\scalar{H}{ u}{h}\,\norm{u}{H^1}^2-2\scalar{H^1}{ u}{h}-\scalar{H}{ u}{h}\right)h\\[10pt]
&\hsp +\left(2\vert\scalar{H}{ u}{h}\vert^2+2\scalar{H}{ u}{h}\scalar{H^1}{ u}{h}-1-\norm{u}{H^1}^2\right)u	
	\end{split}
\end{align}

Now, since we are assuming that $\norma{H}{ u}{H}=1$, we have
\begin{align*}
\norma{H}{ \sigma(u)}{H}^2u	&=\norma{H}{ \sigma_0(u)}{H}^2-\left|
\scalar{H}{ \sigma_0(u)}{h}\right|^2u=g^2(\norm{u}{H^1}^2)\left(1-\left|\scalar{H}{ u}{h}\right|^2\right)u\\[10pt]
&=g^3(\norm{u}{H^1}^2)\left(1+\norm{u}{H^1}^2-\left|\scalar{H}{ u}{h}\right|^2-\norm{u}{H^1}^2\left|\scalar{H}{ u}{h}\right|^2\right)u.
\end{align*}
Therefore, if we sum this expression with \eqref{szfine4}, we obtain \eqref{szfine1}, with $\Lambda$ defined as in \eqref{szfine2}.
	
\end{proof}

\end{Lemma}

The mapping $\Lambda:H^1\to H^1$ we have introduced in Lemma \ref{lemma9.1} is continuous and the set
\[Z:=\left\{ u\in\,H^1\cap M\ :\ \Lambda(u)=0\right\},\]
is a closed subset of $H^1\cap M$. It is immediate to check that, if $\bar{u} \in\,H^1\cap M$ is such that
\[\scalar{H}{ \bar{u}}{h}=0,\;\;\; \scalar{H^1}{ \bar{u}}{h}\neq 0,\]
we have that $\Lambda(\bar{u})=-2\langle \bar{u},h\rangle_{H^1}(1+\vert \bar{u}\vert_{H^1}^2)^3 h\neq 0$, and this means  that $Z^c:=H^1\cap M\setminus Z$ is a non-empty open set.

Now, we fix $u_0 \in\,Z^c$ and we denote by $u$ the solution of the equation
\begin{equation}
	\label{sz-1}
	\left\{\begin{array}{l}
\ds{\gamma\,\partial_t u(t,\xi)=\Delta u(t,\xi)+\norm{u(t)}{H^1}^2 u(t,\xi)-\frac 12 \vert \sigma(u(t))\vert_{H}^2 u(t) +\sigma(u(t))d\beta_t,}\\[18pt]
\ds{u(0,\xi)=u_0(\xi),\ \ \ \ u(t,\xi)=0,\ \ \ \xi \in\,\partial D,	}
\end{array}\right.
\end{equation}
where $\sigma$ is the mapping introduced above and $\beta_t$ is a standard Brownian motion. Moreover, we denote by $\tilde{u}$ the solution of the equation
\begin{equation}
	\label{sz-2}
	\left\{\begin{array}{l}
\ds{\gamma\,\partial_t \tilde{u}(t,\xi)=\Delta \tilde{u}(t,\xi)+\vert \tilde{u}(t)\vert_{H^1}^2 \tilde{u}(t,\xi)+\frac 12 \sigma^\prime(\tilde{u}(t))\sigma(\tilde{u}(t)) +\sigma(\tilde{u}(t))d\beta_t,}\\[18pt]
\ds{\tilde{u}(0,\xi)=u_0(\xi),\;\;\; \tilde{u}(t,\xi)=0,\ \ \ \xi \in\,\partial D,	}
\end{array}\right.
\end{equation}
for the same mapping $\sigma $ and the same Brownian motion $\beta_t$. Both equations admit a unique solution in $L^2(\Omega;C([0,T];H^1\cap M))$
\begin{Theorem}
The two solutions $u$ and $\tilde{u}$ are different.	
\end{Theorem}
\begin{proof}
We introduce the stopping time
\[\tau:=\inf\{t \in\,[0,T]\ :\ u(t)	 \in\,Z\},\]
with the usual convention that $\inf \emptyset=T$.
Since $u(0)=u_0 \in\,Z^c$ and $Z$ is closed, we have that $\mathbb{P}(\tau>0)=1$. Now, if we assume that there exists some stopping time $\tau^\prime$ such that $\mathbb{P}(\tau^\prime>0)=1$ and
\[u(s)=\tilde{u}(s),\;\;\; s \in\,[0,\tau^\prime),\;\;\;\mathbb{P}-\text{a.s.}\]
we have
\[-\int_0^t \norma{H}{\sigma(u(s))}{H}^2 u(s)\,ds=\int_0^t \sigma^\prime(u(s))\sigma(u(s))\,ds,\ \ \ \ t<\tau^\prime,\;\;\;\mathbb{P}-\text{a.s.}\]
In particular
\[\int_0^t \Lambda(u(s))\,ds=0,\ \ \ \ t<\tau^\prime,\;\;\;\mathbb{P}-\text{a.s.}\]
so that
\[\Lambda(u(t))=0,\;\;\;\text{a.e.}\ t<\tau^\prime\wedge \tau, \;\;\;\mathbb{P}-\text{a.s.}\]
However, this is not possible, as $\mathbb{P}\left(u(t) \in\,Z^c,\ t<\tau\right)=1$.
\end{proof}

\section{A-priori bounds. Part I}
\label{sec6}
In what follows we prove a series of a priori-bounds for the solution of system \eqref{system}.

\medskip

\begin{Lemma} \label{Lemma6.1}
	Assume Hypothesis \ref{hyp-H3}  and fix $(u_0,v_0) \in\,\mathscr{H}\cap \mathscr{M}$. Then, for every  integer $p\geq 1$ and every $T>0$ there exists a constant $c_{T,p}>0$ such that for every $\mu \in\,(0,1)$
	\begin{align}
\begin{split}
\label{sz159}
\mu^{p}\,\mathbb{E}\sup_{s \in\,[0,t]}&\norma{H}{ \partial_t u_\mu(s)}{H}^{2p}+ \mathbb{E}\sup_{s \in\,[0,t]}\vert u_\mu(s)\vert_{H^1}^{2p}	+\mu^{p-1}\,\mathbb{E} \,\int_0^t\norma{H}{ \partial_t u_\mu(s)}{H}^{2p}\,ds
\\[10pt]
&\leq c_{T,p}+\frac{c}{\mu}\,\mathbb{E} \,\int_0^t\vert u_\mu(s)\vert_{H^1}^{2(p-1)}\,\Vert \sigma_0(u_\mu(s))\Vert_{\mathscr{T}_2(K,H)}^2\,ds.\end{split}
\end{align}
	\end{Lemma}
	
	\begin{proof}
	Let us consider the function $K_p:H \ni u \mapsto \vert v\vert_{H}^{2p} \in    \mathbb{R}$. Then  the 2nd order Frech\'et derivative of $K_p$ satisfies
\begin{equation}
\label{eqn-d^2K_2p}
D^2K_p(v)=
4p(p-1)\,\normb{H}{ v}^{2(p-2)}\scalar{H}{v}{\cdot}\scalar{H}{v}{\cdot} +2p\,\normb{H}{v}^{2(p-1)} \scalar{H}{\cdot}{\cdot}.
\end{equation}
	Thus, if we set
\begin{equation}\label{eqn-I_p}
I_p(u,v):= \tr \bigl[ D^2K_p(\sigma(u)\cdot,\sigma(u)\cdot) \bigr]= \sum_{k=1}^\infty \dela{\scalar{H}{ D^2 K_p(v)\sigma(u)\tilde{e}_k}{\sigma(u)\tilde{e}_k}}
 D^2 K_p (\sigma(u)\tilde{e}_k,\sigma(u)\tilde{e}_k),
\end{equation}
	by \eqref{eqn-4.14} we have
\begin{align}
	 \begin{split}
	 \label{sz160}
	 I_p(u,v)=& \sum_{k=1}^\infty \left(
4p(p-1)\,\normb{H}{ v}^{2(p-2)}\scalar{H}{v}{\sigma(u)\tilde{e_k}}^2 +2p\,\norma{H}{ v}{H}^{2(p-1)} \normb{H}{\sigma(u)\tilde{e_k}}^2 \right)
\\[10pt]
&
\hsp \leq c_p \norma{H}{ v}{H}^{2(p-1)}\,\Vert \sigma_0(u)\Vert_{\mathscr{T}_2(K,H)}^2.	
	 \end{split}
	\end{align}
\dela{
	4p(p-1)\,\norma{H}{ v}{H}^{2(p-2)}\,\norma{H}{ [\sigma_0(u)Q]^\star v}{H}^2
+	
2p\,\norma{H}{ v}{H}^{2(p-1)}\left(\Vert \sigma_0(u)\Vert_{\mathscr{T}_2(K,H)}^2-\vert [\sigma_0(u)Q]^\star u\vert_{H}^2\right)
}

In particular, by  the It\^o Lemma \ref{lem-Ito-final} applied to the function $K_p$ and the process $v_\mu$, we get
	\begin{align}
\label{sz165}
d\,&\vert v_\mu(t)\vert_{H}^{2p}\leq -\frac p\mu \vert v_\mu(t)\vert_{H}^{2(p-1)}\,d\vert u_\mu(t)\vert_{H^1}^2-\frac{2\,p\,\gamma}\mu \vert v_\mu(t)\vert_{H}^{2p}\,dt\\[10pt]
&+\frac{c_p}{\mu^2}\,\norma{H}{ v_\mu(t)}{H}^{2(p-1)}\,\Vert \sigma_0(u_\mu(t))\Vert_{\mathscr{T}_2(K,H)}^2\,dt+\frac{2p}\mu \norma{H}{ v_\mu(t)}{H}^{2(p-1)}\scalar{H}{ \sigma_0(u_\mu(t))\,d\WP(t)}{v_\mu(t)},
	\end{align}
so that
\begin{align}
\label{sz166}
d\,\vert v_\mu(t)&\vert_{H}^{2p}+\frac p\mu d\left(\vert v_\mu(t)\vert_{H}^{2(p-1)}\,\vert u_\mu(t)\vert_{H^1}^2\right)+\frac{2\,p\,\gamma}\mu \vert v_\mu(t)\vert_{H}^{2p}\,dt\\[10pt]
&\leq\frac{c_p}{\mu^2}\,\norma{H}{ v_\mu(t)}{H}^{2(p-1)}\,\Vert \sigma_0(u_\mu(t))\Vert_{\mathscr{T}_2(K,H)}^2\,dt+
\\[10pt]
	&\hslp\frac{2p}\mu \norma{H}{ v_\mu(t)}{H}^{2(p-1)}\scalar{H}{ \sigma_0(u_\mu(t))\,d\WP(t)}{v_\mu(t)}+\frac p{\mu}\,\vert u_\mu(t)\vert_{H^1}^2\,d\,\vert v_\mu(t)\vert_{H}^{2(p-1)}.
\end{align}

	Next, if we use  inequality \eqref{sz165} in \eqref{sz166}, with $p$ replaced by  $p-1$,  we get
	\begin{align*}
&d\,\vert v_\mu(t)\vert _{H}^{2p}	+\frac p\mu d\left(\vert v_\mu(t)\vert_{H}^{2(p-1)}\,\vert u_\mu(t)\vert_{H^1}^2\right)+\frac{2\,p\,\gamma}\mu \vert v_\mu(t)\vert_{H}^{2p}\,dt\\[10pt]
&\leq \frac{c_p}{\mu^2}\,\norma{H}{ v_\mu(t)}{H}^{2(p-1)}\,\Vert \sigma_0(u_\mu(t))\Vert_{\mathscr{T}_2(K,H)}^2\,dt+\frac{2p}\mu \norma{H}{ v_\mu(t)}{H}^{2(p-1)}\scalar{H}{ \sigma_0(u_\mu(t))\,d\WP(t)}{v_\mu(t)}\\[10pt]
&\hsllp-\frac{p(p-1)}{\mu^2}\vert u_\mu(t)\vert_{H^1}^2\vert v_\mu(t)\vert_{H}^{2(p-2)}\,d\vert u_\mu(t)\vert_{H^1}^2-\frac{2\,p(p-1)\,\gamma}{\mu^2}\vert u_\mu(t)\vert_{H^1}^2 \vert v_\mu(t)\vert_{H}^{2(p-1)}\,dt\\[10pt]
&\hsllp+\frac{p\,c_p}{\mu^3}\,\vert u_\mu(t)\vert_{H^1}^2\norma{H}{ v_\mu(t)}{H}^{2(p-2)}\,\Vert \sigma_0(u_\mu(t))\Vert_{\mathscr{T}_2(K,H)}^2\,dt\\[10pt]
&\hsllp+\frac{2p(p-1)}{\mu^2}\,\vert u_\mu(t)\vert_{H^1}^2 \norma{H}{ v_\mu(t)}{H}^{2(p-2)}\scalar{H}{ \sigma_0(u_\mu(t))\,d\WP(t)}{v_\mu(t)},	\end{align*}
and this implies
\begin{align*}
&d\,\vert v_\mu(t)\vert _{H}^{2p}	+\frac p\mu d\left(\vert v_\mu(t)\vert_{H}^{2(p-1)}\,\vert u_\mu(t)\vert_{H^1}^2\right)+\frac{p(p-1)}{2\mu^2}\,d\left(\vert v_\mu(t)\vert_{H}^{2(p-2)}\,\vert u_\mu(t)\vert_{H^1}^4\right)\\[10pt]
&+\frac{2\,p\,\gamma}\mu \vert v_\mu(t)\vert_{H}^{2p}\,dt+\frac{2\,p(p-1)\,\gamma}{\mu^2}\vert u_\mu(t)\vert_{H^1}^2 \vert v_\mu(t)\vert_{H}^{2(p-1)}\,dt\\[10pt]
&\leq c_p\left(\frac{1}{\mu^2}\,\norma{H}{ v_\mu(t)}{H}^{2(p-1)}+\frac{1}{\mu^3}\,\vert u_\mu(t)\vert_{H^1}^2\norma{H}{ v_\mu(t)}{H}^{2(p-2)}\right)\Vert \sigma_0(u_\mu(t))\Vert_{\mathscr{T}_2(K,H)}^2\,dt\\[10pt]
&+\frac{2p}\mu \norma{H}{ v_\mu(t)}{H}^{2(p-1)}\scalar{H}{ \sigma_0(u_\mu(t))\,d\WP(t)}{v_\mu(t)}+\frac{p(p-1)}{2\mu^2}\,\vert u_\mu(t)\vert_{H^1}^4\,d\vert v_\mu(t)\vert_{H}^{2(p-2)}\\[10pt]
&+\frac{2p(p-1)}{\mu^2}\,\vert u_\mu(t)\vert_{H^1}^2 \norma{H}{ v_\mu(t)}{H}^{2(p-2)}\scalar{H}{ \sigma_0(u_\mu(t))\,d\WP(t)}{v_\mu(t)}\\[10pt]
&.	\end{align*}
By proceeding in this way recursively, we obtain
\begin{align*}
\sum_{i=0}^p\frac{a_{i,p}}{\mu^i}\,d&\left(\norma{H}{ v_\mu(t)}{H}^{2(p-i)}\,\vert u_\mu(t)\vert_{H^1}^{2i}\right)+\sum_{i=1}^p\frac{b_{i,p}}{\mu^i}\,\norma{H}{ v_\mu(t)}{H}^{2(p-i+1)}\,\vert u_\mu(t)\vert_{H^1}^{2(i-1)}\,dt
\\[10pt]
&\leq \sum_{i=1}^p\frac{c_{i,p}}{\mu^{i+1}}\norma{H}{ v_\mu(t)}{H}^{2(p-i)}\,\vert u_\mu(t)\vert_{H^1}^{2(i-1)}\,\Vert \sigma_0(u_\mu(t))\Vert_{\mathscr{T}_2(K,H)}^2\,dt\\[10pt]
&\hsp+\sum_{i=1}^p\frac{d_{i,p}}{\mu^{i}}\norma{H}{ v_\mu(t)}{H}^{2(p-i)}\,\vert u_\mu(t)\vert_{H^1}^{2(i-1)}\scalar{H}{ \sigma_0(u_\mu(t))\,d\WP(t)}{v_\mu(t)}.
\end{align*}
Thus, if we integrate both sides with respect to time and then take the supremum, we get
\begin{align}
\begin{split}
\label{sz177}
\sum_{i=0}^p&\frac{1}{\mu^i} \sup_{s \in\,[0,t]}\norma{H}{ v_\mu(s)}{H}^{2(p-i)}\,\vert u_\mu(s)\vert_{H^1}^{2i}	+\sum_{i=0}^{p-1}\frac{1}{\mu^{i+1}}\,\int_0^t\norma{H}{ v_\mu(s)}{H}^{2(p-i)}\,\vert u_\mu(s)\vert_{H^1}^{2i}\,ds
\\[10pt]
&\leq \frac{c_{p}}{\mu^p}+c\sum_{i=1}^p \frac 1{\mu^{i+1}}\int_0^t\norma{H}{ v_\mu(s)}{H}^{2(p-i)}\,\vert u_\mu(s)\vert_{H^1}^{2(i-1)}\,\Vert \sigma_0(u_\mu(s))\Vert_{\mathscr{T}_2(K,H)}^2\,ds\\[10pt]
&+c\sum_{i=1}^p \frac 1{\mu^{i}}\sup_{s \in\,[0,t]}\left|\int_0^s\norma{H}{ v_\mu(r)}{H}^{2(p-i)}\,\vert u_\mu(r)\vert_{H^1}^{2(i-1)}\scalar{H}{ \sigma_0(u_\mu(r))\,d\WP(r)}{v_\mu(r)}\right|\\[10pt]
&\hsp=:\frac{c_{p}}{\mu^p}+\sum_{i=1}^p I_{i,p}(t)+\sum_{i=1}^p J_{i,p}(t).
\end{split}
\end{align}
Due  to the boundedness of $\sigma_0:H^1\to\mathscr{T}_2(K,H)$,  we have
\begin{align}
	\begin{split}
	\label{sz175}
\sum_{i=1}^p &\,\mathbb{E}\,I_{i,p}(t)	=	\frac 1{\mu^{p+1}}\mathbb{E} \,\int_0^t\,\vert u_\mu(s)\vert_{H^1}^{2(p-1)}\,\Vert \sigma_0(u_\mu(s))\Vert_{\mathscr{T}_2(K,H)}^2\,ds\\[10pt]
&\hslp +\frac 14 \sum_{i=2}^p \frac 1{\mu^{i+1}}\mathbb{E} \,\int_0^t\,\norma{H}{ v_\mu(s)}{H}^{2(p-i)}\,\vert u_\mu(s)\vert_{H^1}^{2i}\,ds+c \sum_{i=2}^p \frac 1{\mu^{i+1}}\mathbb{E} \,\int_0^t\,\norma{H}{ v_\mu(s)}{H}^{2(p-i)}\,ds.
	\end{split}
\end{align}
Moreover,
\begin{align*}
\begin{split}
\mathbb{E}\,&\sum_{i=1}^p J_{i,p}(t)\leq \frac c{\mu}\,\mathbb{E}\left(\int_0^t \norma{H}{ v_\mu(s)}{H}^{4(p-1)+2}\,\Vert \sigma_0(u_\mu(s))\Vert_{\mathscr{T}_2(K,H)}^2\,ds\right)^{\frac 12}\\[10pt]
&\hsp+c\sum_{i=2}^p\frac 1{\mu^i}\,\mathbb{E}\left(\int_0^t \norma{H}{ v_\mu(s)}{H}^{4(p-i)+2}\vert u_\mu(s)\vert_{H^1}^{4(i-1)}\,\Vert \sigma_0(u_\mu(s))\Vert_{\mathscr{T}_2(K,H)}^2\,ds\right)^{\frac 12}\\[10pt]
&\leq \frac 1{2\mu}\,\mathbb{E}\sup_{s \in\,[0,t]}\norma{H}{ v_\mu(s)}{H}^{2p}+\frac c{\mu^2}\,\mathbb{E} \,\int_0^t\norma{H}{ v_\mu(s)}{H}^{2(p-1)}\,ds\\[10pt]
&\hsllp+\sum_{i=2}^p\frac 1{2\mu^i}\,\mathbb{E}
\sup_{s \in\,[0,t]}\norma{H}{ v_\mu(s)}{H}^{2(p-i)}\,\vert u_\mu(s)\vert_{H^1}^{2i}+c\sum_{i=2}^p\frac 1{\mu^i}\,\mathbb{E} \,\int_0^t
\norma{H}{ v_\mu(s)}{H}^{2(p-i+1)}\,\vert u_\mu(s)\vert_{H^1}^{2(i-2)}\,ds.
\end{split}
	\end{align*}
Now, we have
\begin{align*}
c\sum_{i=2}^p&\frac 1{\mu^i}\,\mathbb{E} \,\int_0^t
\norma{H}{ v_\mu(s)}{H}^{2(p-i+1)}\,\vert u_\mu(s)\vert_{H^1}^{2(i-2)}\,ds	=c\sum_{i=1}^{p-1}\frac 1{\mu^{i+1}}\,\mathbb{E} \,\int_0^t
\norma{H}{ v_\mu(s)}{H}^{2(p-i)}\,\vert u_\mu(s)\vert_{H^1}^{2(i-1)}\,ds\\[10pt]
&\leq \sum_{i=1}^{p-1}\frac 1{4\mu^{i+1}}\,\mathbb{E} \,\int_0^t
\norma{H}{ v_\mu(s)}{H}^{2(p-i)}\,\vert u_\mu(s)\vert_{H^1}^{2i}\,ds+c\sum_{i=1}^{p-1}\frac 1{\mu^{i+1}}\,\mathbb{E} \,\int_0^t
\norma{H}{ v_\mu(s)}{H}^{2(p-i)}\,ds,
\end{align*}
so that
\begin{align}
\label{sz176}
\begin{split}
\sum_{i=1}^p \mathbb{E}\,&J_{i,p}(t) \leq \frac 1{2\mu}\,\mathbb{E}\sup_{s \in\,[0,t]}\norma{H}{ v_\mu(s)}{H}^{2p}+\sum_{i=2}^p\frac 1{2\mu^i}\,\mathbb{E}
\sup_{s \in\,[0,t]}\norma{H}{ v_\mu(s)}{H}^{2(p-i)}\,\vert u_\mu(s)\vert_{H^1}^{2i}\\[10pt]
&\hslp +\sum_{i=1}^{p-1}\frac 1{4\mu^{i+1}}\,\mathbb{E} \,\int_0^t
\norma{H}{ v_\mu(s)}{H}^{2(p-i)}\,\vert u_\mu(s)\vert_{H^1}^{2i}\,ds+c\,\sum_{i=1}^{p-1}\frac 1{\mu^{i+1}}\,\mathbb{E} \,\int_0^t
\norma{H}{ v_\mu(s)}{H}^{2(p-i)}\,ds.
\end{split}
	\end{align}
Therefore, if we take the expectation of both sides in \eqref{sz177} and replace  \eqref{sz175} and \eqref{sz176} in it, we obtain
\begin{align}
\begin{split}
\mu^p\,&\mathbb{E}\sup_{s \in\,[0,t]}\norma{H}{ v_\mu(s)}{H}^{2p}+ \mathbb{E}\sup_{s \in\,[0,t]}\vert u_\mu(s)\vert_{H^1}^{2p}	+\mu^{p-1}\,\mathbb{E} \,\int_0^t\norma{H}{ v_\mu(s)}{H}^{2p}\,ds
\\[10pt]
&\leq c_{T,p}+\frac 1{\mu}\mathbb{E} \,\int_0^t\vert u_\mu(s)\vert_{H^1}^{2(p-1)}\,\Vert \sigma_0(u_\mu(s))\Vert_{\mathscr{T}_2(K,H)}^2\,ds+c\, \sum_{i=1}^p \mu^{(p-i)-1}\mathbb{E} \,\int_0^t
\norma{H}{ v_\mu(s)}{H}^{2(p-i)}\,ds.
\end{split}
\end{align}
In particular, by a recursive argument,  this implies \eqref{sz159}.

\end{proof}

\begin{Remark}
{\em 	In the case $p=2$, inequality \eqref{sz159} implies \begin{align} \begin{split}
\label{sz173}	
\mu^3\,\mathbb{E}\sup_{s \in\,[0,t]}\vert v_\mu(s)&\vert _{H}^{4}+ \mu\,\mathbb{E}\sup_{s \in\,[0,t]}\vert u_\mu(s)\vert_{H^1}^{4}+\mu^2\,\mathbb{E} \,\int_0^t\norma{H}{ v_\mu(s)}{H}^{4}\,ds\\[10pt]&\leq c_T
+ \mathbb{E} \,\int_0^t\vert u_\mu(s)\vert_{H^1}^{2}\,ds.
\end{split}
\end{align}

	\hfill\(\Box\)}
\end{Remark}

\begin{Lemma}
\label{lemma20}
Under Hypothesis \ref{hyp-H3}, for every $(u_0,v_0) \in\,\mathscr{H}_1\cap \mathscr{M}$ and $T>0$ there exists a constant $c_T>0$ such that for every $\mu \in\,(0,1)$ and $t \in\,[0,T]$
\begin{align}
\begin{split}
\mathbb{E}\sup_{t \in\,[0,T]}\vert u_\mu(t)\vert_{H^2}^2&+\mu\,\mathbb{E}\,\sup_{t \in\,[0,T]}\vert \partial_t u_\mu(t)\vert_{H^1}^2+ \mathbb{E} \int_0^t\vert \partial_tu_\mu(s)\vert_{H^1}^2\,ds\\[10pt]
&\leq \frac c\mu\mathbb{E} \,\int_0^t\vert u_\mu(s)\vert_{H^1}^2\,ds+  \frac{c_T}{\mu}.\label{sz16}
\end{split}
 \end{align}
\end{Lemma}

\begin{proof}
The It\^o Lemma \ref{lem-Ito-final} gives\begin{align}
\begin{split}
\hs\frac 12 (\vert u_\mu(t)\vert_{H^2}^2	&+\mu\vert v_\mu(t)\vert_{H^1}^2)	=\scalar{H^2}{ u_\mu(t)}{v_\mu(t)}\,dt\\[10pt]
&+\scalar{H^1}{ v_\mu(t)}{\Delta u_\mu(t)+\vert u_\mu(t)\vert_{H^1}^2u_\mu(t)-\mu\vert v_\mu(t)\vert^2_H u_\mu(t)-\gamma v_\mu(t)}\,dt\\[10pt]
\label{sz17}&+\frac 1{2\mu}\Vert \sigma(u_\mu(t))\Vert_{\mathscr{T}_2(K,H^1)}^2\,dt+\scalar{H^1}{ \sigma(u_\mu(t))d\WP(t)}{v_\mu(t)}\\[10pt]
&\hs =\frac 12\vert u_\mu(t)\vert_{H^1}^2d\vert u_\mu(t)\vert_{H^1}^2-\frac{\mu}2\,\vert v_\mu(t)\vert_{H}^2\,d\vert u_\mu(t)\vert_{H^1}^2-\gamma\,\vert v_\mu(t)\vert_{H^1}^2\,dt\\[10pt]
&+\frac 1{2\mu}\Vert \sigma(u_\mu(t))\Vert_{\mathscr{T}_2(K,H^1)}^2\,dt+\scalar{H}{ \sigma(u_\mu(t))d\WP(t)}{v_\mu(t)}.
\end{split}
\end{align}
Now,
\begin{align}
\begin{split}
\mu\,\vert v_\mu(t)\vert_{H}^2\,d\vert u_\mu(t)\vert_{H^1}^2&=\mu\, d\left(\vert v_\mu(t)\vert_{H}^2\,\vert u_\mu(t)\vert_{H^1}^2\right)-2\vert u_\mu(t)\vert_{H^1}^2\scalar{H}{ v_\mu(t)}{\mu\,dv_\mu(t)}\\[10pt]
&-\frac 1{\mu}\vert u_\mu(t)\vert_{H^1}^2\Vert \sigma(u_\mu(t))\Vert^2_{\mathscr{T}_2(K,H)}\,dt\\[10pt]
\label{sz18}&\hs =\mu d\left(\vert v_\mu(t)\vert_{H}^2\,\vert u_\mu(t)\vert_{H^1}^2\right)	+\vert u_\mu(t)\vert_{H^1}^2\,d\vert u_\mu(t)\vert_{H^1}^2+2\gamma\,\vert u_\mu(t)\vert_{H^1}^2\vert v_\mu(t)\vert_{H}^2\,dt\\[10pt]
&-\frac 1{\mu}\vert u_\mu(t)\vert_{H^1}^2\Vert \sigma(u_\mu(t))\Vert^2_{\mathscr{T}_2(K,H)}dt.	
\end{split}
\end{align}
Thus, if we replace \eqref{sz18} into \eqref{sz17}, we get
\begin{align*}
\frac 12 d\left(\vert u_\mu(t)\vert_{H^2}^2	\right.&+\left.\mu\vert v_\mu(t)\vert_{H^1}^2+ \mu  \vert v_\mu(t)\vert_{H}^2\,\vert u_\mu(t)\vert_{H^1}^2\right)+\gamma\left(\vert u_\mu(t)\vert_{H^1}^2\norma{H}{ v_\mu(t)}{H}^2+\vert v_\mu(t)\vert_{H^1}^2\right)\,dt\\[10pt]
&\hs=\frac 1{2\mu}\left(\vert u_\mu(t)\vert^2_{H^1}\Vert \sigma(u_\mu(t))\Vert^2_{\mathscr{T}_2(K,H)}+\Vert \sigma(u_\mu(t))\Vert^2_{\mathscr{T}_2(K,H^1)}\right)\,dt\\[10pt]
&\hsp+\scalar{H^1}{ \sigma(u_\mu(t))d\WP(t)}{v_\mu(t)}.
\end{align*}
In view of  inequalities \eqref{sz5} and \eqref{sz120} in Lemma \ref{lem-basic estimates on sigma}, this implies
\begin{align*}
\frac 12 d\left(\vert u_\mu(t)\vert_{H^2}^2	\right.&+\left.\mu\vert v_\mu(t)\vert_{H^1}^2+ \mu  \vert v_\mu(t)\vert_{H}^2\,\vert u_\mu(t)\vert_{H^1}^2\right)+\gamma\left(\vert u_\mu(t)\vert_{H^1}^2\norma{H}{ v_\mu(t)}{H}^2+\vert v_\mu(t)\vert_{H^1}^2\right)\,dt\\[10pt]
&\leq \frac c{2\mu}\left(\vert u_\mu(t)\vert^2_{H^1}+1\right)\,dt+\scalar{H^1}{ \sigma(u_\mu(t))d\WP(t)}{v_\mu(t)}.
\end{align*}
Therefore, there exists a constant $c_T>0$ such that after we integrate with respect to time and take the supremum, for every $\mu \in\,(0,1)$ we obtain
\begin{align}
\begin{split}
\sup_{s\in [0, t]}\vert u_\mu(s)\vert_{H^2}^2&	+\mu\sup_{s\in [0, t]}\vert v_\mu(s)\vert_{H^1}^2 +\int_0^t\vert v_\mu(s)\vert_{H^1}^2\,ds\\[10pt]
\label{sz20}&\hsllp\leq \frac{c_T}{\mu}+\frac c{\mu}\int_0^t\vert u_\mu(s)\vert^2_{H^1}\,ds+c\,\sup_{s\in [0, t]}\left|\int_0^s \scalar{H^1}{ \sigma(u_\mu(r))d\WP(r)}{v_\mu(r)}\right|.
\end{split}
\end{align}
By the Davis inequality, see \cite{Pardoux_1976},   and Hypothesis \ref{hyp-H3}, we have
\begin{align}\begin{split}
	&\mathbb{E}\sup_{s\in [0, t]}\left|\int_0^s \scalar{H^1}{ \sigma(u_\mu(r))d\WP(r)}{v_\mu(r)}\right|\leq c\,
\mathbb{E} \left(\int_0^t  \Normb{\mathscr{T}_2(K,H^1)}{\sigma(u_\mu(r))}^2\normb{H^1} {v_\mu(r)}^2 \, dr \right)^{\frac12}
\\[10pt]
&\leq  c\, \mathbb{E} \left(\int_0^t \normb{H^1} {v_\mu(r)}^2 \left(1+\normb{H^1} {u_\mu(r)}^2   \right)\, dr \right)^{\frac12}
\\[10pt]
 & \leq  c\, \mathbb{E} \left(\int_0^t \normb{H^1} {v_\mu(r)}^2 \, dr \right)^{\frac12}
 +c\, \mathbb{E} \left(\int_0^t \normb{H^1} {v_\mu(r)}^2 \normb{H^1} {u_\mu(r)}^2   \, dr \right)^{\frac12}
 \\[10pt]
 & \leq  c+ \frac12\, \mathbb{E} \int_0^t \normb{H^1} {v_\mu(r)}^2 \, dr  + c\, \mathbb{E}  \left(  \sup_{s \in\,[0,t]}\normb{H^1} {v_\mu(s)}   \int_0^t \normb{H^1} {u_\mu(r)}^2   \, dr  \right)^{\frac12}
  \\[10pt]
 & \leq  c+ \frac12\, \mathbb{E} \int_0^t \normb{H^1} {v_\mu(r)}^2 \, dr  + \frac{\mu}{2}\,  \mathbb{E}  \sup_{s \in\,[0,t]}  \normb{H^1} {v_\mu(s)}^2
+\frac{c}{\mu}\, \mathbb{E}    \int_0^t \normb{H^1} {u_\mu(r)}^2   \, dr.
\label{eqn-new-Davis}	
\end{split}\end{align}

Combining \eqref{eqn-new-Davis} and \eqref{sz20} we get

\begin{align}\label{sz20-1}
\mathbb{E} \sup_{s\in [0, t]}\vert u_\mu(s)\vert_{H^2}^2	&+\mu\, \mathbb{E} \sup_{s\in [0, t]}\vert v_\mu(s)\vert_{H^1}^2 + \mathbb{E} \int_0^t\vert v_\mu(s)\vert_{H^1}^2\,ds\\[10pt]
&\leq \frac c{\mu} \mathbb{E} \int_0^t\vert u_\mu(s)\vert^2_{H^1}\,ds+ \frac{c_T}{\mu},
\end{align}
and this  implies \eqref{sz16}.
\end{proof}

\begin{Lemma}
\label{lemma1}
Under Hypothesis \ref{hyp-H3}, for every $(u_0,v_0) \in\,\mathscr{H}\cap \mathscr{M}$, such that $v_0 \in\,H^1$, and for every $T>0$ there exists a constant $c_T>0$ such that for every $\mu \in\,(0,1)$ and $t \in\,[0,T]$
\begin{align}
\begin{split}
\label{sz8}
\mathbb{E}\vert u_\mu(t)\vert_{H^1}^2\leq c_T+c\,\mu\,\mathbb{E} \,\int_0^t\vert \partial_t u_\mu(s)\vert_{H^1}^2\,ds.
\end{split}
\end{align}

\end{Lemma}
\begin{proof}
The It\^o Lemma \ref{lem-Ito-final} gives
\begin{align*}
\mu\, d\scalar{H^1}{ u_\mu(t)}{v_\mu(t)}&=\mu\, \scalar{H^1}{ du_\mu(t)}{v_\mu(t)}+ \scalar{H^1}{ u_\mu(t)}{ \mu\, dv_\mu(t)}\\[10pt]
&\hspace{-3truecm}=\left(\mu\,\vert v_\mu(t)\vert_{H^{1}}^2+\scalar{H^1}{ u_\mu(t)}{\Delta u_\mu(t)}+\vert u_\mu(t)\vert_{H^1}^2\vert u_\mu(t)\vert_{H^{1}}^2-\mu\,\vert v_\mu(t)\vert_{H}^2\vert u_\mu(t)\vert_{H^1}^2\right.\\[10pt]
&\hsl\left.-\gamma\scalar{H^1}{ v_\mu(t)}{u_\mu(t)} \right)\,dt+\scalar{H^1}{ \sigma(u_\mu(t))d\WP(t)}{u_\mu(t)}\\[10pt]
&\hspace{-3truecm}=\left(\mu\,\vert v_\mu(t)\vert_{H^{1}}^2-\vert u_\mu(t)\vert_{H^{2}}^2+\vert u_\mu(t)\vert_{H^1}^4-\mu\,\vert v_\mu(t)\vert_{H}^2\vert u_\mu(t)\vert_{H^{1}}^2\right)\,dt\\[10pt]
&\hsl-\frac\gamma 2 d\vert u_\mu(t)\vert_{H^{1}}^2 +\scalar{H^1}{ \sigma(u_\mu(t))d\WP(t)}{u_\mu(t)}.
\end{align*}
Hence, if we integrate both sides above with respect to time, we get
\begin{align*}
\begin{split}
\frac \mu 2\,\frac{d}{dt} \vert u_\mu(t)&\vert_{H^{1}}^2+\frac \gamma 2\vert u_\mu(t)\vert_{H^{1}}^2+\int_0^t \left(\vert u_\mu(s)\vert_{H^{2}}^2-\vert u_\mu(t)\vert_{H^1}^4\right)\,ds\\[10pt]
&\hsl\leq \mu\,\int_0^t \vert v_\mu(s)\vert_{H^{1}}^2\,ds+\int_0^t \scalar{H^1}{ \sigma(u_\mu(s))d\WP(s)}{u_\mu(s)}+\frac \mu 2\scalar{H^1}{ u_0}{v_0}+\frac \gamma 2\vert u_0\vert_{H^{1}}^2.\end{split}
\end{align*}
Now, thanks to \eqref{jan2} for every $u \in\,H^{2}\cap M$  we have $
\norm{u}{H^1}^4\leq \norm{u}{H^{2}}^2$. Therefore we can find  $c>0$ independent of  $\mu \in\,(0,1)$ such that
\begin{align}
\label{sz225}
\frac{d}{dt} \vert u_\mu(t)\vert_{H^{1}}^2&+\frac \gamma \mu\vert u_\mu(t)\vert_{H^{1}}^2\leq \frac{c}\mu + c\int_0^t \vert v_\mu(s)\vert_{H^{1}}^2\,ds
\\&+\frac{2}\mu\int_0^t \scalar{H^1}{ \sigma(u_\mu(s))d\WP(s)}{u_\mu(s)}.
\end{align}
In particular, if we take the expectation of both sides in \eqref{sz225}
we get
\[\frac{d}{dt} \mathbb{E}\,\vert u_\mu(t)\vert_{H^{1}}^2+\frac 1 \mu\,\mathbb{E}\,\vert u_\mu(t)\vert_{H^{1}}^2\leq \frac{c_T}\mu +c\int_0^t \mathbb{E}\,\vert v_\mu(s)\vert_{H^{1}}^2\,ds.\]
Finally, by a comparison argument this implies \eqref{sz8}.

\end{proof}

\begin{Remark}
{\em By combining together \eqref{sz16} and \eqref{sz8}, we have	 that for every $(u_0,v_0) \in\,\mathscr{H}_1\cap \mathscr{M}$ and $T>0$ there exists  $c_T(u_0,v_0)>0$ such that for every $\mu \in\,(0,1)$
\[\mathbb{E}\vert u_\mu(t)\vert_{H^1}^2\leq c_T(u_0,v_0)+c\,\mathbb{E} \,\int_0^t\vert u_\mu(s)\vert_{H^1}^2\,ds,\;\;\; t \in\,[0,T].\]
Hence, from the Gronwall Lemma we conclude
\begin{equation}
\label{sz70}
\mathbb{E}\vert u_\mu(t)\vert_{H^1}^2\leq c_T(u_0,v_0),\;\;\; t \in\,[0,T].
\end{equation}
In view of \eqref{sz16}, this also implies that
\begin{equation}
\label{sz154}
\int_0^T \mathbb{E}\vert\partial_t u_\mu(s)\vert_{H^1}^2\,ds\leq \frac 1\mu\,c_T(u_0,v_0).	
\end{equation}
Moreover, thanks to  \eqref{sz173} and \eqref{sz70}, we get
 \begin{equation}
\label{sz24}
\sup_{\mu \in\,(0,1)}\,\mu^{3/2}\,\mathbb{E}\sup_{t \in\,[0,T]}\vert \partial_t u_\mu(t)\vert_{H}^2<\infty.	
\end{equation}
 \hfill\(\Box\)}
\end{Remark}

\section{A priori bounds. Part II}
\label{sec7}

Now, we want to show that in fact, if the initial condition $(u_0,v_0)$ belongs to $\mathscr{H}_2\cap \mathscr{M}$, then  the solution $(u_\mu, \partial u_\mu)$ belongs to $L^2(\Omega;C([0,T];\mathscr{H}_2\cap \mathscr{M}))$ and suitable uniform bounds with respect to $\mu \in\,(0,1)$ are satisfied. In what follows, it will be fundamental to assume that Hypotheses \ref{hyp-H4} and \ref{hyp-H5} hold.

\begin{Lemma} \label{Lemma7.1}
Under Hypotheses \ref{hyp-H3} and \ref{hyp-H4}, for every $(u_0,v_0) \in\,\mathscr{H}_1\cap \mathscr{M}$ and $T>0$ there exists a constant $c_T>0$ such that for every $\mu \in\,(0,1)$	we have
\begin{align}
\begin{split}
\mu^3\,\mathbb{E}&\sup_{t \in\,[0,T]}\vert \partial_tu_\mu(t)\vert_{H^1}^4+\mu\,\mathbb{E}\sup_{t \in\,[0,T]}\vert u_\mu(t)\vert_{H^2}^4+\mu^2\int_0^T \mathbb{E}\vert \partial_t u_\mu(s)\vert_{H^1}^4\,ds\\[10pt]
\label{sz29}&\hsp+\mu\,\int_0^T\mathbb{E}\vert u_\mu(s)\vert_{H^2}^2\vert \partial_t u_\mu(s)\vert_{H^1}^2\,ds \leq c_T+c_T\int_0^T\mathbb{E}\,\vert u_\mu(t)\vert_{H^2}^2\,dt.
\end{split}	
\end{align}

\end{Lemma}
\begin{proof}
In order to prove \eqref{sz29} we  apply the It\^o Lemma \ref{lem-Ito-final} to the function
\[K: H^1 \ni v \mapsto \vert v\vert_{H^1}^4 \in \mathbb{R}\]
and the $H^1$-valued   process $v_\mu(t)$. Since $DK(v)=4\,\vert v\vert^2_H v$, we get
\begin{align*}
\begin{split}
d\,\vert v_\mu(t)\vert_{H^1}^4=&	\frac{4}\mu\,\vert v_\mu(t)\vert_{H^1}^2\scalar{H^1}{ v_\mu(t)}{\Delta u_\mu(t)+\vert u_\mu(t)\vert_{H^1}^2u_\mu(t)-\mu \norma{H}{ v_\mu(t)}{H}^2 u_\mu(t)-\gamma v_\mu(t)}\,dt\\[10pt]
&\hsl+\frac 1{2\mu^2}\,\sum_{j=1}^\infty D^2K(v_\mu(t))\,(\sigma(u_\mu(t))\tilde{e_j},\sigma(u_\mu(t))\tilde{e_j})\,dt\\[10pt]
&\hsl+\frac 4\mu \vert v_\mu(t)\vert_{H^1}^2\scalar{H^1}{ \sigma(u_\mu(t))d\WP(t)}{v_\mu(t)}.	
\end{split}
\end{align*}
Thus, if we denote
\[I(u,v):=\sum_{j=1}^\infty D^2K(v)\,(\sigma(u)\tilde{e_j},\sigma(u)\tilde{e_j}),\]
we have
\begin{align}
\begin{split}
\label{sz132}
d\,\vert v_\mu(t)\vert_{H^1}^4&+\frac {4\gamma}\mu \vert v_\mu(t)\vert_{H^1}^4=-\frac 2\mu\vert v_\mu(t)\vert_{H^1}^2\,d\vert u_\mu(t)\vert _{H^2}^2+\frac 2\mu \vert v_\mu(t)\vert_{H^1}^2 \vert u_\mu(t)\vert_{H^1}^2\,d\vert u_\mu(t)\vert_{H^1}^2\\[10pt]
&\hsl-2\,\vert v_\mu(t)\vert_{H^1}^2 \vert v_\mu(t)\vert_{H}^2\,d\vert u_\mu(t)\vert_{H^1}^2+\frac 1{2\mu^2}\,I(v_\mu(t),u_\mu(t))\,dt\\[10pt]
&+\frac 4\mu \vert v_\mu(t)\vert_{H^1}^2\scalar{H^1}{ \sigma(u_\mu(t))d\WP(t)}{v_\mu(t)}.
\end{split}
\end{align}

The It\^o \ref{lem-Ito-final} Lemma gives
\begin{align*}
d\vert v_\mu(t)\vert_{H^1}^2&=\frac 2\mu \scalar{H}{ v_\mu(t)}{\Delta u_\mu(t)+\vert u_\mu(t)\vert_{H^1}^2u_\mu(t)-\mu \norma{H}{ v_\mu(t)}{H}^2 u_\mu(t)-\gamma v_\mu(t)}\,dt\\[10pt]
&\hslp +\frac 1{\mu^2} \Vert \sigma(u_\mu(t))\Vert^2_{\mathscr{T}_2(K,H^1)}\,dt+\frac 2\mu\,\scalar{H^1}{ \sigma(u_\mu(t))d\WP(t)}{v_\mu(t)}	\\[10pt]
&= -\frac 1{\mu}d\,\vert u_\mu(t)\vert_{H^2}^2+\frac 1{2\mu}d\,\vert u_\mu(t)\vert_{H^1}^4-\norma{H}{ v_\mu(t)}{H}^2d\vert u_\mu(t)\vert_{H^1}^2-\frac {2\gamma}\mu\vert v_\mu(t)\vert_{H^1}^2\,dt\\[10pt]
&\hslp +\frac 1{\mu^2} \Vert \sigma(u_\mu(t))\Vert^2_{\mathscr{T}_2(K,H^1)}\,dt+\frac 2\mu\,\scalar{H^1}{ \sigma(u_\mu(t))d\WP(t)}{v_\mu(t)}.
\end{align*}
Hence, if we define
\[\Phi(u):=\norm{u}{H^2}^2-\frac 12\norm{u}{H^1}^4,\]
it is not difficult to check that
\begin{align} \begin{split}
-\frac 2\mu&\vert v_\mu(t)\vert_{H^1}^2\,d\vert u_\mu(t)\vert _{H^2}^2+\frac 2\mu \vert v_\mu(t)\vert_{H^1}^2 \vert u_\mu(t)\vert_{H^1}^2\,d\vert u_\mu(t)\vert_{H^1}^2\\[10pt]
&=-\frac 2\mu d\left(\Phi(u_\mu(t))\vert v_\mu(t)\vert_{H^1}^2\right)-\frac 2{\mu^2}\,	\Phi( u_\mu(t))\,d\Phi(u_\mu(t))-\frac 2\mu\Phi(u_\mu(t))\vert v_\mu(t)\vert_{H}^2\,d\vert u_\mu(t)\vert_{H^1}^2 \\[10pt]
\label{sz137}& -\frac {4\gamma}{\mu^2}\Phi(u_\mu(t))\vert v_\mu(t)\vert_{H^1}^2dt+\frac 2{\mu^3}\,\Phi(u_\mu(t))\,\Vert \sigma(u_\mu(t))\Vert^2_{\mathscr{T}_2(K,H^1)}dt\\[10pt]
&+\frac {4}{\mu^2}\,\Phi(u_\mu(t))\scalar{H^1}{ \sigma(u_\mu(t))d\WP(t)}{v_\mu(t)}.
\end{split}  \end{align}

In particular, if we replace \eqref{sz137} into \eqref{sz132}, we get
\begin{align}
\begin{split}  \label{sz138}
d\,\Big(\vert &v_\mu(t)\vert_{H^1}^4+\frac 2\mu \Phi(u_\mu(t))\vert v_\mu(t)\vert_{H^1}^2+\frac 1{\mu^2}\,	\Phi^2( u_\mu(t))\Big)+\Big(\frac {4\gamma}\mu \vert v_\mu(t)\vert_{H^1}^4+\frac {4\gamma}{\mu^2}\Phi(u_\mu(t))\vert v_\mu(t)\vert_{H^1}^2\Big)\,dt\\[10pt]
&=-\frac 2\mu\left(\Phi(u_\mu(t))+\mu\,\vert v_\mu(t)\vert_{H^1}^2\right) \vert v_\mu(t)\vert_{H}^2\,d\vert u_\mu(t)\vert_{H^1}^2  \\[10pt]
&\hsllp +\frac 2{\mu^3}\,\Phi(u_\mu(t))\,\Vert \sigma(u_\mu(t))\Vert^2_{\mathscr{T}_2(K,H^1)}dt+\frac 1{2\mu^2}\,I(v_\mu(t),u_\mu(t))\,dt\\[10pt]
&\hsllp+\frac 4{\mu^2}\,\Phi(u_\mu(t))\scalar{H^1}{ \sigma(u_\mu(t))d\WP(t)}{v_\mu(t)}+\frac 4\mu \vert v_\mu(t)\vert_{H^1}^2\scalar{H^1}{ \sigma(u_\mu(t))d\WP(t)}{v_\mu(t)}.
\end{split}
\end{align}

Now, since similarly to \eqref{eqn-d^2K_2p} we have
\begin{equation}
\label{eqn-d^2K_4}
D^2K(v)=
8 \scalar{H^1}{v}{\cdot}\scalar{H^1}{v}{\cdot} +4\,\normb{H^1}{ v}^{2} \scalar{H^1}{\cdot}{\cdot}.
\end{equation}
remembering that $\normb{H}{u}=1$ from \eqref{eqn-singma-sigma_0} we infer that
\begin{align} \label{sz124}
D^2 K\,(\sigma(u)\tilde{e}_j,\sigma(u)\tilde{e}_j)&= 8 \scalar{H^1}{v}{\sigma(u)\tilde{e}_j}^2
+4\,\normb{H^1}{ v}^{2} \normb{H^1}{\sigma(u)\tilde{e}_j }^{2}
\\[10pt]
& \hs\leq  c\, \normb{H^1}{ v}^{2} \normb{H^1}{\sigma(u)\tilde{e}_j }^{2}
\leq c\, \normb{H^1}{ v}^{2} \left( \normb{H^1}{\sigma_0(u)\tilde{e}_j }^{2}  +
\normb{H}{\sigma_0(u)\tilde{e}_j }^2 \dela{\normb{H}{u}^2} \normb{H^1}{u}^2 \right).
 \end{align}
Therefore, if we sum both sides in \eqref{sz124} with respect to $j \in\,\mathbb{N}$,  we get
\begin{align}
\label{sz127}
I(v,u)&=\sum_{j=1}^\infty D^2K\,(\sigma(u)\tilde{e}_j,\sigma(u)\tilde{e}_j)
\\
&\leq c\, \normb{H^1}{ v}^{2} \left( \Normb{\mathscr{T}_2(K,H^1)}{\sigma_0(u)}^{2}  +
\Normb{\mathscr{T}_2(K,H)}{\sigma_0(u)}^2 \dela{\normb{H}{u}^2} \normb{H^1}{u}^2 \right).
\end{align}
Thus, in view of Hypotheses \ref{hyp-H3} and  \ref{hyp-H4}
we infer that for some positive $c>0$,
 \begin{equation}
	\label{sz131}
\sup_{u \in\,H^1\cap M}\,I(v,u)\leq c\,\vert v\vert_{H^1}^2,\;\;\; v \in\,H^1.
\end{equation}

\dela{
\[D^2K(v)=8 (-\Delta)v\otimes (-\Delta)v+4\,\vert v\vert_{H^1}^2(-\Delta),\]
for every $k \in\,\mathbb{N}$ we have
\begin{align}
\begin{split}
D^2K(v)&\,(\sigma(u)Qe_j,\sigma(u)Qe_j)\\[10pt]
&\hsl=8\,\scalar{H}{ [\sigma_0(u) Q]^\star (-\Delta)v}{e_j}^2-8\scalar{H}{ [\sigma_0(u) Q]^\star u}{e_j} \scalar{H}{ [\sigma_0(u) Q]^\star (-\Delta) v}{e_j}\scalar{H^1}{ u}{\Delta u_\mu(t)}\\[10pt]
& +4\,\vert v\vert_{H^1}^2\vert \sigma_0(u)Q e_j\vert^2_{H^1}-4\,\vert v\vert_{H^1}^2\scalar{H}{ [\sigma_0(u) Q]^\star (-\Delta) u}{e_j} \scalar{H}{ [\sigma_0(u) Q]^\star  u}{e_j}\\[10pt]
& -\scalar{H}{ e_j}{[\sigma_0(u) Q]^\star u}\Big(8\coma{\langle [\sigma_0(u) Q]^\star (-\Delta)v,_k\rangle_H}   \scalar{H^1}{ u}{\Delta u_\mu(t)}-
8\scalar{H}{ e_j}{[\sigma_0(u) Q]^\star u} \scalar{H^1}{ u}{\Delta u_\mu(t)}^2\\[10pt]
& -4\,\vert v\vert_{H^1}^2\scalar{H}{ e_j}{[\sigma_0(u) Q]^\star (-\Delta)u}+4\,\vert v\vert_{H^1}^2\scalar{H}{ e_j}{[\sigma_0(u) Q]^\star u}^2 |u|^2_{H^1}\Big).
	\end{split}	
 \end{align}
For arbitrary $v, z \in\,H$, we have
\[\sum_{j=1}^\infty \scalar{H}{ [\sigma_0(u)Q]^\star v}{e_j} \scalar{H}{ [\sigma_0(u)Q]^\star (-\Delta) z}{e_j}=\scalar{H^1}{ [\sigma_0(u)Q]^\star [\sigma_0(u)Q] v}{z}.\]
Therefore, if we sum both sides in \eqref{sz124} with respect to $k \in\,\mathbb{N}$,  we get
\begin{align}
\begin{split}	
&I(v,u)=\,8\,\norma{H}{ [\sigma_0(u)Q]^\star(-\Delta)v}{H}^2	-16\,\scalar{H^1}{ u}{\Delta u_\mu(t)}\scalar{H^1}{ [\sigma_0(u)Q]^\star [\sigma_0(u)Q] u}{\Delta u_\mu(t)}\\[10pt]
&\hsp+4\,\vert v\vert_{H^1}^2\Vert \sigma_0(u)\Vert^2_{\mathscr{T}_2(K,H^1)}-8\,\vert v\vert_{H^1}^2 \scalar{H^1}{ [\sigma_0(u)Q]^\star [\sigma_0(u)Q] u}{u}\\[10pt]
 &\hsp+8\, \scalar{H^1}{ u}{\Delta u_\mu(t)}^2\norma{H}{ [\sigma_0(u)Q]^\star u}{H}^2+4\,\vert v\vert_{H^1}^2 \vert u\vert^2_{H^1}\norma{H}{ [\sigma_0(u)Q]^\star u}{H}^2.
\end{split}
\end{align}
For every $h \in\,H$, we have
\begin{align*}
	\left|\scalar{H}{ [\sigma_0(u)Q]^\star(-\Delta)v}{h}\right|=\left|\scalar{H^1}{ v}{[\sigma_0(u)Q]^\star h}\right|\leq \Vert \sigma_0(u)\Vert_{\mathscr{L}(K,H^1)}\norma{H}{ h}{H}\vert v\vert_{H^1},
\end{align*}
so that
\coma{\begin{equation}
\label{sz231}
\norma{H}{ [\sigma_0(u)Q]^\star(-\Delta)v}{H}\leq c\,	\Vert \sigma_0(u)\Vert_{\mathscr{T}_2(K,H^1)}\vert v\vert_{H^1}.
\end{equation}}
Thus, if we use estimates \eqref{sz230} and \eqref{sz231}  in \eqref{sz127}, we easily obtain
\begin{align*}
&I(v,u)\leq c\,\vert v\vert^2_{H^1}\Vert \sigma_0(u)\Vert_{\mathscr{T}_2(K,H^1)}^2\\[10pt]
&\hsp+c\,\vert v\vert_{H^1}^2\Vert \sigma_0(u)\Vert_{\mathscr{T}_2(K,H^1)}\Vert \sigma_0(u)\Vert_{\mathscr{T}_2(K,H)}\norm{u}{H^1}+c\,\vert v\vert_{H^1}^2\Vert \sigma_0(u)\Vert_{\mathscr{T}_2(K,H)}^2 \norm{u}{H^1}^2,	
\end{align*}
and thanks to Hypothesis \ref{hyp-H4} and Hypothesis \ref{hyp-H3}, i.e.  the boundedness of $\sigma_0:H^1\to \mathscr{T}_2(K,H^1)$, this implies \begin{equation}
\sup_{u \in\,H^1\cap M}\,I(v,u)\leq c\,\vert v\vert_{H^1}^2,\;\;\; v \in\,H^1.
\end{equation}
}

According to \eqref{jan2}, we have
\begin{equation}\label{inequality-gradient flow}
\frac 12 \norm{u}{H^2}^2\leq \Phi(u)\leq \norm{u}{H^2}^2.
\end{equation}
Then, if we integrate  both sides in \eqref{sz138} with respect to time, thanks to \eqref{sz131} we obtain
\begin{align}
\begin{split}
\label{sz145}
\vert v_\mu(t)\vert_{H^1}^4	+&\frac 1{\mu^2}\vert u_\mu(t)\vert^4_{H^2}+\frac 1\mu \int_0^t\vert v_\mu(s)\vert_{H^1}^4\,ds+\frac 1{\mu^2}\int_0^t\vert u_\mu(s)\vert^2_{H^2}\vert v_\mu(s)\vert_{H^1}^2\,ds\\[10pt]
&\hsl\leq \frac{c}{\mu^2}+\frac c\mu\int_0^t\vert u_\mu(s)\vert_{H^2}^2 \vert v_\mu(s)\vert_{H}^2\vert\scalar{H^1}{ u_\mu(s)}{v_\mu(s)}\vert\,ds\\[10pt]
&+\int_0^t\vert v_\mu(s)\vert_{H^1}^2 \vert v_\mu(s)\vert_{H}^2\left| \scalar{H^1}{ u_\mu(s)}{v_\mu(s)}\right|\,ds \\[10pt]
&+\frac c{\mu^3}\,\int_0^t\vert u_\mu(s)\vert_{H^2}^2\,\Vert \sigma(u_\mu(s))\Vert^2_{\mathscr{T}_2(K,H^1)}\,ds+\frac c{\mu^2}\,\int_0^t \vert v_\mu(s)\vert_{H^1}^2\,ds\\[10pt]
&+\frac 4{\mu^2}\,\int_0^t\Phi(u_\mu(s))\scalar{H^1}{ \sigma(u_\mu(s))\,d\WP(s)}{v_\mu(s)}\\[10pt]
&+\frac 4\mu \int_0^t\vert v_\mu(s)\vert_{H^1}^2\scalar{H^1}{ \sigma(u_\mu(s))\,d\WP(s)}{v_\mu(s)}=:\frac{c}{\mu^2}+\sum_{i=1}^6 J^\mu_i(t).
\end{split}
\end{align}

In what follows we will estimate each term  $J^\mu_i(t)$.
Since
\[\vert\langle u,v\rangle_{H^1}\vert\leq \vert u\vert_{H^1}\vert v\vert_{H^1}\leq \vert u\vert_{H^2}^{1/2}\vert v\vert_{H^1},\]
for $J^\mu_1(t)$ we have
\begin{align}
\begin{split}
\label{sz140}
J^\mu_1(t)&\leq \frac c\mu\int_0^t\vert u_\mu(s)\vert_{H^2}^{5/2} \vert v_\mu(s)\vert_{H}^2\vert v_\mu(s)\vert_{H^1}\,ds\\[10pt]	
&\leq \frac 1{2\mu^2}\int_0^t \vert u_\mu(s)\vert_{H^2}^2\vert v_\mu(s)\vert_{H^1}^2\,ds+\frac c{\mu^2}\int_0^t \vert u_\mu(s)\vert_{H^2}^4\,ds+c\,\mu^6\int_0^t \norma{H}{ v_\mu(s)}{H}^{16}\,ds.
\end{split}
\end{align}
For $J^\mu_2(t)$, due to \eqref{jan2}, we have
\begin{align}
\begin{split}
\label{sz141}
J^\mu_2(t)&\leq c\int_0^t\vert v_\mu(s)\vert_{H^1}^2 \vert v_\mu(s)\vert_{H}^2\vert u_\mu(s)\vert_{H^1}\vert v_\mu(s)\vert_{H^1}\,ds\leq c\int_0^t\vert v_\mu(s)\vert_{H^1}^3 \vert v_\mu(s)\vert_{H}^2\vert u_\mu(s)\vert_{H^2}^{1/2}\,ds\\[10pt]
&\leq \frac {1}{4\mu}\int_0^t \vert v_\mu(s)\vert_{H^1}^4\,ds+\frac {c}{\mu^2}\int_0^t \vert u_\mu(s)\vert_{H^2}^4\,ds+c\,\mu^8\int_0^t \norma{H}{ v_\mu(s)}{H}^{16}\,ds.
\end{split}
\end{align}
As for $J^\mu_3(t)$, due to \eqref{sz5-bis} we have
\begin{align}
\begin{split}
\label{sz142}
J^\mu_3(t)&\leq \frac{c}{\mu^3}\int_0^t\vert u_\mu(s)\vert_{H^2}^2\,ds.
\end{split}
\end{align}
For $J^\mu_4(t)$, we have
\begin{align}
\begin{split}
\label{sz143}
J^\mu_4(t)&\leq \frac{1}{4\mu}\int_0^t\vert v_\mu(s)\vert_{H^1}^4\,ds+\frac{c_T }{\mu^3}.
\end{split}
\end{align}
Therefore, if we replace \eqref{sz140}, \eqref{sz141}, \eqref{sz142} and \eqref{sz143} in \eqref{sz145}, we get
\begin{align*}
\vert v_\mu(t)\vert_{H^1}^4	+&\frac 1{\mu^2}\vert u_\mu(t)\vert^4_{H^2}+\frac 1{2\mu} \int_0^t\vert v_\mu(s)\vert_{H^1}^4\,ds+\frac 1{2\mu^2}\int_0^t\vert u_\mu(s)\vert^2_{H^2}\vert v_\mu(s)\vert_{H^1}^2\,ds\\[10pt]
&\hsl \leq \frac{c}{\mu^2}+ \frac c{\mu^2}\int_0^t \vert u_\mu(s)\vert_{H^2}^4\,ds+\frac{c}{\mu^3}\int_0^t\vert u_\mu(s)\vert_{H^2}^2\,ds\\[10pt]
&\hslp +c\,\mu^6\int_0^t \norma{H}{ v_\mu(s)}{H}^{16}\,ds+\frac{c_T}{\mu^3}+J^\mu_5(t)+J^\mu_6(t).
\end{align*}
In particular, thanks to  the Gronwall Lemma, for every $t \in\,[0,T]$ and $\mu \in\,(0,1)$ we get
\begin{align}
\begin{split}
\label{sz150}
\sup_{s\in [0, t]}\vert v_\mu(s)\vert_{H^1}^4&+	\frac 1{\mu^2}\sup_{s\in [0, t]}\vert u_\mu(s)\vert^4_{H^2}+\frac 1{2\mu} \int_0^t\vert v_\mu(s)\vert_{H^1}^4\,ds+\frac 1{2\mu^2}\int_0^t\vert u_\mu(s)\vert^2_{H^2}\vert v_\mu(s)\vert_{H^1}^2\,ds\\[10pt]
&\hsl\leq \frac{c_T}{\mu^2}+\frac{c_T}{\mu^3}\int_0^t\vert u_\mu(s)\vert_{H^2}^2\,ds+\frac{c_T}{\mu^2}\int_0^t\vert u_\mu(s)\vert_{H^2}^4\,ds+c_T\,\mu^6\int_0^t \norma{H}{ v_\mu(s)}{H}^{16}\,ds\\[10pt]
&+\frac {c_T}{\mu^3}+c_T\left(\sup_{s\in [0, t]}\left|J^\mu_5(s)\right|+\sup_{s\in [0, t]}\left|J^\mu_6(s)\right|\right).
	\end{split}
	\end{align}

Thanks to inequality  \eqref{sz5-bis}\dela{{sz231}} in Lemma \ref{lem-basic estimates on sigma}  and inequality \eqref{inequality-gradient flow}, we have
\begin{align}
\begin{split}
\label{sz152}
&\mathbb{E}	\sup_{s\in [0, t]}\left|J^\mu_5(s)\right|\leq \frac c{\mu^2}\mathbb{E}\left(\int_0^t|\Phi(u_\mu(s))|^2 \normb{H^1} {v_\mu(r)}^2 \,ds\right)^{\frac 12}\\[10pt]
&\hsp\leq \frac 1{2\mu^2}\,\mathbb{E}\sup_{s\in [0, t]}\vert u_\mu(s)\vert_{H^2}^4 + \frac c{\mu^2} \mathbb{E} \int_0^t\vert v_\mu(s)\vert_{H^1}^2\,ds.
\end{split}
\end{align}
Similarly, we get
\begin{align}\begin{split}
\label{sz151}	
\mathbb{E}	\sup_{s\in [0, t]}\left|J^\mu_6(s)\right|&\leq \frac c{\mu}\,\mathbb{E}\left(\int_0^t\vert v_\mu(s))\vert_{H^1}^4  \normb{H^1} {v_\mu(r)}^2 \,ds\right)^{\frac 12}
\\[10pt]
&\leq \frac 1{2}\,\mathbb{E}\sup_{s\in [0, t]}\vert v_\mu(s)\vert_{H^1}^4
+\frac c{\mu^2}\mathbb{E} \,\int_0^t\vert v_\mu(s)\vert_{H^1}^2\,ds.
\end{split}
\end{align}

Therefore, if we take the expectation of both sides in \eqref{sz150}, in view of \eqref{sz152} and \eqref{sz151}, we obtain
\begin{align}
\begin{split}
\label{sz153}
\mathbb{E}\,\sup_{s\in [0, t]}\vert v_\mu(s)&\vert_{H^1}^4+	\frac 1{\mu^2}\mathbb{E}\,\sup_{s\in [0, t]}\vert u_\mu(s)\vert^4_{H^2}\\[10pt]
&+\frac 1{\mu} \mathbb{E} \,\int_0^t\,\vert v_\mu(s)\vert_{H^1}^4\,ds+\frac 1{\mu^2}\mathbb{E} \,\int_0^t\vert u_\mu(s)\vert^2_{H^2}\vert v_\mu(s)\vert_{H^1}^2\,ds\\[10pt]
&\hs\leq \frac{c_T}{\mu^2}+\frac{c_T}{\mu^3}\mathbb{E} \,\int_0^t\,\vert u_\mu(s)\vert_{H^2}^2\,ds+\frac{c_T}{\mu^2}\,\mathbb{E} \,\int_0^t\,\vert u_\mu(s)\vert_{H^2}^4\,ds+c_T\,\mu^6 \int_0^t \mathbb{E}\,\norma{H}{ v_\mu(s)}{H}^{16}\,ds\\[10pt]
&+\frac c{\mu^2}\mathbb{E} \,\int_0^t\vert v_\mu(s)\vert_{H^1}^2\,ds+\frac {c_T}{\mu^3}.
	\end{split}
	\end{align}
As a consequence of the Gronwall lemma, thanks to \eqref{sz154}, after we multiply both sides by $\mu^3$ we get
\begin{align*}
\begin{split}
\mu^3\,\mathbb{E}&\sup_{t \in\,[0,T]}\vert \partial_tu_\mu(t)\vert_{H^1}^4+\mu\,\mathbb{E}\sup_{t \in\,[0,T]}\vert u_\mu(t)\vert_{H^2}^4\\[10pt]
&\hsp+\mu^2\int_0^t \mathbb{E}\vert \partial_t u_\mu(s)\vert_{H^1}^4\,ds+\mu\,\mathbb{E} \,\int_0^t\vert u_\mu(s)\vert^2_{H^2}\vert v_\mu(s)\vert_{H^1}^2\,ds\\[10pt]
&\leq c_T+c_T\int_0^T\mathbb{E}\,\vert u_\mu(t)\vert_{H^2}^2\,dt+c_T\,\mu^{9}\int_0^T \mathbb{E}\,\vert \partial_t u_\mu(s)\vert_{H}^{16}\,ds.
\end{split}	
\end{align*}
Now, according to \eqref{sz159} and Hypothesis \ref{hyp-H4}, we have
\begin{align*}  \begin{split}
\mu^{9}\,\mathbb{E}&\int_0^T\vert \partial_t u_\mu(s)\vert_{H}^{16}\,ds\leq c_T+c_{T}\,\mu\,
\mathbb{E}\int_0^t \vert u_\mu(s)\vert_{H^1}^{14}\Vert \sigma_0(u_\mu(s))\Vert_{\mathscr{T}_2(K,H)}^2\,ds\\[10pt]
&\hslp\leq c_T+c_{T}\,\mu\,
\mathbb{E}\int_0^t \vert u_\mu(s)\vert_{H^1}^{10}\,ds,
\end{split}
\end{align*}
and if we use again \eqref{sz159} and Hypothesis \ref{hyp-H4} we have
\begin{align*}  \begin{split}
\mu^{9}\,\mathbb{E}&\int_0^T\vert \partial_t u_\mu(s)\vert_{H}^{16}\,ds\leq c_T+c_{T}\,\mu\,
\left(\frac{c_T}{\mu}+\frac{c}{\mu}\mathbb{E}\int_0^t
\vert u_\mu(s)\vert_{H^1}^{8}\Vert \sigma_0(u_\mu(s))\Vert_{\mathscr{T}_2(K,H)}^2\,ds\right)\\[10pt]
&\hsp\leq c_T+c_{T}\,
\mathbb{E}\int_0^t \vert u_\mu(s)\vert_{H^1}^{4}\,ds.
\end{split}
\end{align*}
This allows to conclude that
\eqref{sz29} holds.
\end{proof}

\begin{Lemma}
\label{lemma1+delta}	
Assume Hypotheses \ref{hyp-H3}, \ref{hyp-H4} and \ref{hyp-H5}, and fix $(u_0,v_0) \in\,\mathscr{H}_2\cap \mathscr{M}$.   Then, for every $T>0$  there exists $c_T$ such that for every $\mu \in\,(0,1)$
\begin{align}
\begin{split}
\label{sz210}
\mu\,\mathbb{E}&\sup_{s \in\,[0,t]}\vert v_\mu(t)\vert_{H^2}^2+\mathbb{E} \,\int_0^t \vert v_\mu(s)\vert_{H^2}^2\,ds	\leq \frac {c_T}{\mu}\mathbb{E} \,\int_0^t\vert u_\mu(s)\vert^2_{H^2}\,ds+\frac{c_T}{\mu}.
\end{split}
	\end{align}

\end{Lemma}
\begin{proof}
The It\^o Lemma \ref{lem-Ito-final} gives
\begin{align}
\begin{split}
\frac 12 d\left(\vert u_\mu(t)\vert_{H^{3}}^2+\right.&\left.\mu\vert v_\mu(t)\vert_{H^2}^2\right)\\[10pt]
\label{sz17-bis}&\hs=\scalar{H^3}{ u_\mu(t)}{du_\mu(t)}+\scalar{H^2}{ v_\mu(t)}{\mu\,dv_\mu(t)}+\frac 1{2\mu} \Vert \sigma(u_\mu(t))\Vert_{\mathscr{T}_2(K,H^2)}^2\,dt\\[10pt]
&\hs=\frac 12 \vert u_\mu(t)\vert_{H^1}^2\,d\vert u_\mu(t)\vert_{H^2}^2-\frac \mu 2\norma{H}{ v_\mu(t)}{H}^2\, d\vert u_\mu(t)\vert_{H^2}^2-\gamma\,\vert v_\mu(t)\vert_{H^2}^2\,dt\\[10pt]
&+\frac 1{2\mu} \Vert \sigma(u_\mu(t))\Vert_{\mathscr{T}_2(K,H^2)}^2\,dt+\scalar{H^2}{ \sigma(u_\mu(t))d\WP(t)}{v_\mu(t)}.	
\end{split}
	\end{align}
	
	Now, we have
	\begin{align}
	\begin{split}
		\label{sz201}
	\vert u_\mu(t)\vert_{H^1}^2\,d\vert u_\mu(t)\vert_{H^2}^2= d\left(\vert u_\mu(t)\vert_{H^1}^2\vert u_\mu(t)\vert_{H^2}^2\right)-d\vert u_\mu(t)\vert_{H^1}^2\vert u_\mu(t)\vert_{H^2}^2, 	
	\end{split}	
	\end{align}
and \begin{align}
\begin{split}
\label{sz200-bis}
\mu\,\vert v_\mu(t)\vert_{H}^2\,d\vert u_\mu(t)\vert_{H^2}^2&\leq 2\mu\,\vert v_\mu(t)\vert_{H}^2\,\vert u_\mu(t)\vert_{H^2}\vert v_\mu(t)\vert_{H^2}\\[10pt]
&\leq \gamma \,\vert v_\mu(t)\vert_{H^2}^2+c\,\vert u_\mu(t)\vert_{H^2}^4+c\,\mu^2\,\vert v_\mu(t)\vert_{H}^4.
\end{split}
\end{align}
Thus, if we plug \eqref{sz201}, and \eqref{sz200-bis} into \eqref{sz17-bis},
we get
\begin{align*}
\begin{split}	
\frac 12 d&\left(\vert u_\mu(t)
\vert_{H^{3}}^2-	\vert u_\mu(t)\vert_{H^2}^2\,\vert u_\mu(t)\vert_{H^1}^2\right)+\frac{\mu}2\,d\vert v_\mu(t)\vert_{H^2}^2+\frac \gamma 2\,\vert v_\mu(t)\vert_{H^2}^2\,dt\\[10pt]
&\hslp\leq c\,\vert u_\mu(t)\vert^4_{H^2}\,dt+c\,\mu^2\,\vert v_\mu(t)\vert_{H}^4-\frac 12\,\vert u_\mu(t)\vert_{H^2}^2\,d\vert u_\mu(t)\vert_{H^1}^2\,dt\\[10pt]
&\hsp +\frac 1{2\mu} \Vert \sigma(u_\mu(t))\Vert_{\mathscr{T}_2(K,H^2)}^2\,dt+\scalar{H^2}{ \sigma(u_\mu(t))d\WP(t)}{v_\mu(t)}.
\end{split}
\end{align*}
Due to \eqref{jan2}, we have
\begin{align*}
\vert u_\mu(t)\vert_{H^2}^2\,d\vert u_\mu(t)\vert_{H^1}^2&\leq 	2\vert u_\mu(t)\vert_{H^2}^2\vert u_\mu(t)\vert_{H^1}\vert v_\mu(t)\vert_{H^1}\leq 	2\vert u_\mu(t)\vert_{H^2}^{5/2}\vert v_\mu(t)\vert_{H^2}^{1/2}\vert v_\mu(t)\vert_{H}^{1/2}\\[10pt]
&\leq c\, \vert u_\mu(t)\vert_{H^2}^2 \vert v_\mu(t)\vert_{H}^2+\frac \gamma 4\, \vert v_\mu(t)\vert_{H^2}^2+c\,\vert u_\mu(t)\vert_{H^2}^4,
\end{align*}
and thanks to inequality \eqref{sz190} from Lemma \ref{lem-basic estimates on sigma} this implies
\begin{align*}
\begin{split}	
\frac 12 d\left(\vert u_\mu(t)\right.&\left.
\vert_{H^{3}}^2-	\vert u_\mu(t)\vert_{H^2}^2\,\vert u_\mu(t)\vert_{H^1}^2\right)+\frac \mu 2\,d\vert v_\mu(t)\vert_{H^2}^2+\frac\gamma 4 \vert v_\mu(t)\vert_{H^2}^2\,dt\\[10pt]
&\leq \frac c{\mu}\,dt+\frac c{\mu}\vert u_\mu(t)\vert^2_{H^2}\,dt+c\,\vert u_\mu(t)\vert_{H^2}^4\,dt+c\, \vert u_\mu(t)\vert_{H^2}^2 \vert v_\mu(t)\vert_{H}^2\,dt+c\,\mu^2\,\vert v_\mu(t)\vert_{H}^4\,dt\\[10pt]
&\hsp+\scalar{H^2}{ \sigma(u_\mu(t))d\WP(t)}{v_\mu(t)}.
\end{split}
\end{align*}
Since
\[\vert u_\mu(t)
\vert_{H^{3}}^2-	\vert u_\mu(t)\vert_{H^2}^2\,\vert u_\mu(t)\vert_{H^1}^2\geq 0, \]
after we integrate with respect to $t$, and  take first the supremum in $t$ and then the expectation,  we get
\begin{align}
\label{sz210-bis}
\mu\,\mathbb{E}&\sup_{s \in\,[0,t]}\vert v_\mu(t)\vert_{H^2}^2+\frac \gamma 4\mathbb{E} \,\int_0^t \vert v_\mu(s)\vert_{H^2}^2\,ds	\leq \frac {c_T}{\mu}+\frac c{\mu}\mathbb{E} \,\int_0^t\vert u_\mu(s)\vert^2_{H^2}\,ds\\[10pt]
&\hslp +c\int_0^t \mathbb{E}\vert u_\mu(s)\vert_{H^2}^4\,ds+c\int_0^t \vert u_\mu(s)\vert_{H^2}^2 \vert v_\mu(s)\vert_{H}^2\,ds+c\,\mu^2\mathbb{E} \,\int_0^t\vert v_\mu(s)\vert_{H}^4\,ds\\[10pt]
&\hsp+\mathbb{E}\sup_{s \in\,[0,t]}\left|\int_0^s\scalar{H^2}{ \sigma(u_\mu(r))d\WP(r)}{v_\mu(r)}\right|.
	\end{align}
According to inequality \eqref{sz190}\dela{{sz180-bis}} from Lemma \ref{lem-basic estimates on sigma}, and inequalities \eqref{sz16} and \eqref{sz70}, we have
\begin{align*}
\mathbb{E}\sup_{s \in\,[0,t]}&\left|\int_0^s\scalar{H^2}{ \sigma(u_\mu(r))d\WP(r)}{v_\mu(r)}\right|\leq c\,\mathbb{E}\left(\int_0^t	\left(1+\vert u_\mu(s)\vert_{H^2}^2\right)\vert v_\mu(s)\vert_{H^2}^2\,ds\right)^{1/2}\\[10pt]
&\leq \frac \gamma 8\mathbb{E} \,\int_0^t\vert v_\mu(s)\vert_{H^2}^2\,ds+c\,\mathbb{E}\sup_{s \in\,[0,t]}\vert u_\mu(s)\vert_{H^2}^2+c_T\leq \frac \gamma 8\mathbb{E} \,\int_0^t\vert v_\mu(s)\vert_{H^2}^2\,ds+\frac {c_T}\mu.
\end{align*}
and if we replace this into \eqref{sz210-bis} we obtain
\begin{align*}
\begin{split}
\mu\,\mathbb{E}&\sup_{s \in\,[0,t]}\vert v_\mu(t)\vert_{H^2}^2+\frac \gamma 8\mathbb{E} \,\int_0^t \vert v_\mu(s)\vert_{H^2}^2\,ds	\leq \frac c{\mu}\mathbb{E} \,\int_0^t\vert u_\mu(s)\vert^2_{H^2}\,ds+c\int_0^t \mathbb{E}\vert u_\mu(s)\vert_{H^2}^4\,ds\\[10pt]
&\hslp+c\int_0^t \vert u_\mu(s)\vert_{H^2}^2 \vert v_\mu(s)\vert_{H}^2\,ds+c\,\mu^2\mathbb{E} \,\int_0^t\vert v_\mu(s)\vert_{H}^4\,ds+\frac {c_T}\mu.
\end{split}
	\end{align*}
Hence, \eqref{sz29} allows to conclude that
\eqref{sz210} holds.
\end{proof}

\begin{Lemma}
\label{lemma1-bis}
Under Hypotheses \ref{hyp-H3}, \ref{hyp-H4} and \ref{hyp-H5}, for every $(u_0,v_0) \in\,\mathscr{H}_2\cap \mathscr{M}$ and $T>0$ there exists a constant $c_T>0$ such that for every $\mu \in\,(0,1)$ and $t \in\,[0,T]$
\begin{align}
\begin{split}
\label{sz8-bis}
\mathbb{E}\vert u_\mu(t)\vert_{H^2}^2\leq c_T+c\,\mu\,\mathbb{E} \,\int_0^t\vert \partial_t u_\mu(s)\vert_{H^2}^2\,ds.
\end{split}
\end{align}
\end{Lemma}

\begin{proof}
We apply the It\^o Lemma \ref{lem-Ito-final} to the function
\[K_\mu:H^2\times H^2 \ni (u,v) \mapsto \mu\,\scalar{H^2}{ u}{v} \in \mathbb{R},\]
and by proceeding as in the the proof of Lemma \ref{lemma1}, we get 	
\begin{align}
\begin{split}
\label{sz213}
\frac \mu 2\,\frac{d}{dt} \vert u_\mu(t)&\vert_{H^{2}}^2+\frac \gamma 2\vert u_\mu(t)\vert_{H^{2}}^2+\int_0^t \left(\vert u_\mu(s)\vert_{H^{3}}^2-\vert u_\mu(t)\vert_{H^1}^2\vert u_\mu(t)\vert_{H^2}^2\right)\,ds\\[10pt]
&\hsl\leq \mu\,\int_0^t \vert v_\mu(s)\vert_{H^{2}}^2\,ds+\int_0^t \scalar{H^2}{ \sigma(u_\mu(s))d\WP(s)}{u_\mu(s)}+\frac \mu 2\scalar{H^2}{ u_0}{v_0}+c.
\end{split}
\end{align}
According to \eqref{jan2} we have
\[\vert u_\mu(s)\vert_{H^{3}}^2-\vert u_\mu(t)\vert_{H^1}^2\vert u_\mu(t)\vert_{H^2}^2\geq 0.\]
Moreover, by combining together \eqref{sz29}, with \eqref{sz16} and \eqref{sz70},
we have that for every $\mu \in\,(0,1)$
\[\mathbb{E} \,\int_0^t\,\vert u_\mu(s)\vert_{H^2}^4\,ds<\infty.\]
Due to inequality  \eqref{sz190}\dela{sz180-bis} from Lemma \ref{lem-basic estimates on sigma}, this implies
\[\mathbb{E}\left|\int_0^t \scalar{H^2}{ \sigma(u_\mu(s))d\WP(s)}{u_\mu(s)}\right|^2\leq  c\,\mathbb{E} \,\int_0^t\,\vert u_\mu(s)\vert_{H^2}^4\,ds+c_T<\infty,\]
so that we can take the expectation of both sides in \eqref{sz213} and we get
\begin{align*}
\begin{split}
\frac{d}{dt} \mathbb{E}\vert u_\mu(t)\vert_{H^{2}}^2+\frac 1\mu\,\mathbb{E}\vert u_\mu(t)\vert_{H^{2}}^2\leq c\int_0^t \mathbb{E}\vert v_\mu(s)\vert_{H^{2}}^2\,ds+\frac{c_T}\mu.
\end{split}
\end{align*}
By a comparison argument this gives \eqref{sz8-bis}.

\end{proof}

\begin{Remark}
{\em By combining together \eqref{sz210} and \eqref{sz8-bis}, we have	 that for every $(u_0,v_0) \in\,\mathscr{H}_2\cap \mathscr{M}$ and $T>0$ there exists  a constant $c_T(u_0,v_0)>0$ such that for every $\mu \in\,(0,1)$
\[\mathbb{E}\vert u_\mu(t)\vert_{H^2}^2\leq c_T(u_0,v_0)+c\,\mathbb{E} \,\int_0^t\vert u_\mu(s)\vert_{H^2}^2\,ds,\;\;\; t \in\,[0,T].\]
Hence, the Gronwall Lemma allows to conclude
\begin{equation}
\label{sz70-bis}
\mathbb{E}\vert u_\mu(t)\vert_{H^2}^2\leq c_T(u_0,v_0),\;\;\; t \in\,[0,T].
\end{equation}
In particular, thanks again to \eqref{sz210},
\begin{equation}
\label{sz154-bis}
\int_0^T \mathbb{E}\vert\partial_t u_\mu(s)\vert_{H^2}^2\,ds\leq \frac 1\mu\,c_T(u_0,v_0).	\end{equation}
\hfill\(\Box\) }
\end{Remark}

\begin{Remark} {\em Let us present an alternative proof of inequalities \eqref{inequality-gradient flow}
\[\frac 12\,\vert u\vert_{H^2}^2\leq \Phi(u):=\vert u\vert_{H^2}^2-\frac 12\,\vert u\vert_{H^1}^4\leq \vert u\vert_{H^2}^2.\]

. The inequality on the right is obvious. The inequality on the left is a consequence of the following argument. Assume that $u \in M \cap H^2$. Then
\begin{align}\label{inequality-gradient flow-3}
\Phi(u)&= \normb{H}{\Delta u}^2 -\frac12\normb{H}{\nabla u}^4=
\frac12 \normb{H}{\Delta u}^2  + \frac12 \left( \normb{H}{\Delta u}^2 -\normb{H}{\nabla u}^4   \right)
\\
&=\frac 12\,\normb{H}{\Delta u+ \normb{H}{\nabla u}^2u}^2 + \frac12 \normb{H}{\Delta u}^2 \geq \frac12 \normb{H}{\Delta u}^2.
\end{align}
Notice that here the crucial identity is
\[
\normb{H}{\Delta u}^2 - \normb{H}{\nabla u}^4= \normb{H}{\Delta u+ \normb{H}{\nabla u}^2u}^2, \;\; u \in M \cap H^2.
\]

}

\end{Remark}

\medskip

\section{Tightness}
\label{sec8}

We first need to introduce some notations and preliminary results. If $E$ is a Banach space and  $T>0$, for every   $0<\sigma<1$ and $1\leq p\leq \infty$ we define
\[W^{\sigma, p}(0,T;E):=\left\{f \in\,L^p(0,T;E)\ :\ [f]_{\dot{W}^{\sigma, p}(0,T;E)}<\infty \right\},\]
where
\[[ f]_{\dot{W}^{\sigma, p}(0,T;E)}:=\left(\int_0^T\int_0^T\frac{\vert f(t)-f(s)\vert_E^p}{\vert t-s\vert^{1+\sigma p}}\,dt\,ds\right)^{\frac 1p}.\]
The space $W^{\sigma, p}(0,T;E)$, endowed with the norm
\[\vert\cdot\vert_{W^{\sigma, p}(0,T;E)}:=\vert\cdot\vert_{L^p(0,T;E)}+[\cdot]_{\dot{W}^{\sigma, p}(0,T;E)},\] is a Banach space. Moreover, for every $f \in\,L^p(0,T;E)$ and $h \in\,[0,T)$ we denote
\[\tau_hf(t)=f(t+h),\;\;\; t \in\,[-h,T-h].\]
In \cite[Lemma 5]{simon1986} it is proven that if $f \in\,W^{\sigma, r}(0,T;E)$, with $0<\sigma<1$ and $1\leq r\leq \infty$, and if $p$ is such that
\[p\leq \infty,\ \text{if}\ \sigma>\frac 1r,
\ \ \ \    p<\infty, \ \text{if}\ \sigma=\frac 1 r,\ \ \ \ p\leq r_\star:=\frac r{1-\sigma r},\ \text{if}\ \sigma<\frac 1r,\]
then $f \in\,L^p(0,T;E)$ and there exists a constant $c$ independent of $f$ such that for every $h\geq 0$
\begin{equation}
\label{sz50}	
\vert \tau_hf-f\vert_{L^p(0,T-h;E)}\leq c\begin{cases}
 h^{\sigma+\frac 1p-\frac 1r}[f]_{\dot{W}^{\sigma, r}(0,T;E)}, & \text{if}\ \ r\leq p\leq \infty,\\[10pt]

 h^{\sigma}T^{\frac 1p-\frac 1r}[f]_{\dot{W}^{\sigma, r}(0,T;E)}, & \text{if}\ \ 1\leq r\leq p.	
 \end{cases}
\end{equation}

According to \eqref{system}, we have
\begin{align*}
\gamma\, du_\mu(t)+\mu\,d v_\mu(t)=&\left(\Delta u_\mu(t)+\vert u_\mu(t)\vert_{H^1}^2u_\mu(t)	-\mu\norma{H}{ v_\mu(t)}{H}^2 u_\mu(t)\right)\,dt\\[10pt]
&+\sigma(u_\mu(t))\,d\WP(t),	
\end{align*}
with $v_\mu(t)=\partial_t u_\mu(t)$.  Then, if we integrate with respect to time, we get
\begin{align}
\begin{split}
\Phi_\mu(t):=\gamma u_\mu(t)+\mu\,v_\mu(t)=&\,\gamma u_0+\mu\,v_0	+\int_0^t\Delta u_\mu(s)\,ds+\int_0^t\vert u_\mu(s)\vert_{H^1}^2u_\mu(s)\,ds	\\[10pt]
\label{sz51}&\hs\hs-\mu\int_0^t\norma{H}{ v_\mu(s)}{H}^2 u_\mu(s)\,ds+\int_0^t
\sigma(u_\mu(s))\,d\WP(s)=:I_\mu+\sum_{j=1}^4 J_{\mu,k}(t).
\end{split}
\end{align}
`
\begin{Lemma}
\label{lemma4}
Under Hypotheses \ref{hyp-H3}, \ref{hyp-H4} and \ref{hyp-H5},  for every $(u_0,v_0) \in\,\mathscr{H}_2\cap \mathscr{M}$, and for every $T>0$ and $\theta <1/2$ we have
\begin{equation}
\label{sz52}
\sup_{\mu \in\,(0,1)}\,\mathbb{E}\left[\Phi_\mu\right]_{	\dot{W}^{\theta,2}(0,T;H)}<\infty.
\end{equation}
	\end{Lemma}
	\begin{proof}
We are going to estimate every  term $J_{\mu, k}(t)$  in \eqref{sz51}. 	
Thanks to \eqref{sz70-bis}, we have
\[\mathbb{E}\,\vert J_{\mu,1}(t)-J_{\mu,1}(s)\vert_{H}^2\leq \int_s^t\mathbb{E}\,\vert u_\mu(r)\vert_{H^2}^2\,dr\,|t-s|\leq c_T\,|t-s|,	
\]
and,  due to \eqref{jan2},
\[
\mathbb{E}\,\vert J_{\mu,2}(t)-J_{\mu,2}(s)\vert_{H}^2\leq \int_s^t\mathbb{E}\,\vert u_\mu(r)\vert_{H^1}^4\,dr|t-s|\leq \int_s^t\mathbb{E}\,\vert u_\mu(r)\vert_{H^2}^2\,dr\,|t-s|\leq c_T\,|t-s|.\]
Next, due  to \eqref{sz173} and \eqref{sz70}, we have
\[
\mathbb{E}\,\norma{H}{ J_{\mu,3}(t)-J_{\mu,3}(s)}{H}^2\leq \mu^2\int_s^t\mathbb{E}\,\vert v_\mu(r)\vert_{H}^4\,dr\,|t-s|\leq c_T\,|t-s|.	
\]
Finally, thanks to \eqref{sz5} we have
\[
\mathbb{E}\vert J_{\mu,4}(t)-J_{\mu,4}(s)\vert^2_H	\leq \mathbb{E}\int_s^t\Vert \sigma(u_\mu(r))\Vert^2_{\mathscr{T}_2(K,H)}\,dr\leq c_T\,|t-s|.
\]
Therefore, by combining together all these bounds, we conclude  that \eqref{sz52} holds, for every $\theta<1/2$.

	\end{proof}
\begin{Lemma}
\label{lemma7.2}
Under Hypotheses \ref{hyp-H3}, \ref{hyp-H4} and \ref{hyp-H5},  for every $(u_0,v_0) \in\,\mathscr{H}_2\cap \mathscr{M}$ and $T>0$ the family $\{\mathscr{L}(u_\mu)\}_{\mu \in\,(0,1)}$ is tight in $L^q(0,T;H^\alpha)$, for every $\alpha \in\,[0,2)$ and $q<4/\alpha$.
\end{Lemma}

\begin{proof}
Due to \eqref{sz50},  for every $p<\infty$ there exists $\theta<1/2$ such that the set
\[K_L:=\left\{ f \in\,L^	p(0,T;H)\,:\,[f]_{\dot{W}^{\theta,2}(0,T;H)}\leq L,\ \int_0^T\vert f(s)\vert_{H^2}^2\,ds\leq L\right\}\]
is relatively compact in $L^p(0,T;H)$, for every  $L>0$ (for a proof see e.g. \cite[Theorem 3]{simon1986}).

Now, due to \eqref{sz70-bis} and \eqref{sz154-bis}, we have that $u_\mu+\mu \partial_t u_\mu$ is bounded in $L^2(\Omega;L^2(0,T;H^2))$. Then, according to \eqref{sz52},  for every  $\e>0$ there exists $L_\e>0$ such that
\[\mathbb{P}\left(u_\mu+\mu\,\partial_t u_\mu \in\,K_{L_\epsilon}\right)\geq 1-\epsilon,\;\;\;\ \ \ \ \ \ \mu \in\,(0,1).\]
This means that the family $\{u_\mu+\mu \partial_t u_\mu\}_{\mu \in\,(0,1)}$ is tight in $L^p(0,T;H)$, for every $p<\infty$.
Moreover, as a consequence of \eqref{sz24}, the family $\{\mu \partial_t u_\mu\}_{\mu \in\,(0,1)}$ is tight in $L^\infty(0,T;H)$,  and this allows to conclude that  the family $\{u_\mu\}_{\mu \in\,(0,1)}$ is tight in $L^p(0,T;H)$, for every $p<\infty$.

In particular, for every $p<\infty$ and $\e>0$ there exists a relatively compact set $K_{1,\e, p}\subset L^p(0,T;H)$ such that
\[\mathbb{P}\left( u_\mu \in\,K_{1,\e, p}\right)\geq 1-\frac \e 2,\;\;\; \mu \in\,(0,1).\]
Moreover, according to \eqref{sz70-bis}, for every $\e>0$ there exists $M_\e>0$ such that
\[\mathbb{P}\left( u_\mu \in\,K_{2,\e}\right)\geq 1-\frac \e 2,\;\;\; \mu \in\,(0,1),\]
where
\[K_{2,\e}:=\left\{ f \in\,L^2(0,T;H^2)\ :\ \vert f\vert_{L^2(0,T;H^2)}\leq M_\e\right\}.\]

Now, for every $\alpha \in\,[0,2)$ we have
\[ \vert h\vert_{H^\alpha}\leq \norm{u}{H}^{1-\alpha/2}\,\norm{u}{H^{2}}^{\alpha/2}.\]
Therefore, if we set
\begin{equation}
\label{final2}	
q:=\frac{4p}{\alpha p+4-2\alpha},\end{equation}
we get
\begin{align}\label{final1}
\begin{split}
\vert \tau_h f-f\vert_{L^q(0,T-h;H^\alpha)}\leq \vert \tau_h f-f\vert_{L^p(0,T-h;H)}^{1-\alpha/2}\vert \tau_h f-f\vert_{L^2(0,T-h;H^2)} ^{\alpha/2}.	
\end{split}
	\end{align}
In view of the characterization of compact sets in $L^p(0,T;H)$ given in \cite[Theorem 1]{simon1986}, we have that
\[\lim_{h\to 0}\,\sup_{f \in\,K_{1,\e, p}}\vert \tau_h f-f\vert_{L^p(0,T-h;H)}=0.\]
Therefore, thanks to \eqref{final1}, we get
\[\lim_{h\to 0}\,\sup_{f \in\,K_{1,\e, p}}\vert \tau_h f-f\vert_{L^q(0,T-h;H^\alpha)}=0.\]
By applying again \cite[Theorem 3]{simon1986}, we conclude that
the set $K_{1, \e, p}\cap K_{2, \e}$ is relatively compact in $L^{q}(0,T;H^\alpha)$ and this allows to
conclude that the family $\{u_\mu\}_{\mu \in\,(0,1)}$ is tight in $L^{q}(0,T;H^\alpha)$, just by noticing that
\[\mathbb{P}\left( u_\mu \in\,K_{1,\e}\cap K_{2,\e, p}\right)\geq 1-\e,\ \ \ \ \ \ \mu \in\,(0,1).\]
Finally, since we can take any arbitrary $p<\infty$, due to \eqref{final2} we have that we can take any $q<4/\alpha$ and our proof is concluded.
\end{proof}

\section{Proof of Theorem \ref{teo8.3}}
\label{sec9}

We start with the following fundamental identity.

\begin{Lemma}
\label{lemma2}
Assume Hypotheses \ref{hyp-H3} and \ref{hyp-H4}. Then, for every $(u_0,v_0) \in\,\mathscr{H}_1\cap \mathscr{M}$ and every  $\mu>0$ and $t \in\,[0,T]$, we have
\[\mu\int_0^t	\norma{H}{ \partial_t u_\mu(s)}{H}^2 u_\mu(s)\,ds=\frac 1{2\gamma}\int_0^t \Vert \sigma(u_\mu(s))\Vert_{\mathscr{T}_2(K,H)}^2 u_\mu(s)\,ds+R_\mu(t),\]
where
\begin{align*}
R_\mu(t):=&\,\frac{\mu^2}{2\gamma}\norma{H}{ v_0}{H}^2u_0 -\frac{\mu^2}{2\gamma}\norma{H}{ \partial_t u_\mu(t)}{H}^2u_\mu(t)+	\frac{\mu}\gamma\int_0^t u_\mu(s)\scalar{H}{ \sigma_0(u_\mu(s))d\WP(s)}{\partial_t u_\mu(s)}\\[10pt]
&\hsl+\frac{\mu^2}{2\gamma}\int_0^t\norma{H}{ \partial_t u_\mu(s)}{H}^2\partial_t u_\mu(s)\,ds+\frac\mu\gamma\int_0^tu_\mu(s)\scalar{H}{ \partial_t u_\mu(s)}{\Delta u_\mu(s)}\,ds\\[10pt]
&=:\frac{\mu^2}{2\gamma}\norma{H}{ v_0}{H}^2u_0 +\sum_{j=1}^4 I_{\mu,k}(t).
\end{align*}

\end{Lemma}

\begin{proof}
As a consequence of the  It\^o Lemma \ref{lem-Ito-final}, we have
\begin{align*}
\frac{\mu^2}2d\,\norma{H}{ v_\mu(t)}{H}^2=&\mu \scalar{H}{ v_\mu(t)}{\Delta u_\mu(t)}\,dt-\mu\gamma \norma{H}{ v_\mu(t)}{H}^2\,dt+\frac 12\Vert \sigma(u_\mu(t))\Vert_{\mathscr{T}_2(K,H)}^2\,dt\\[10pt]
&+\mu\scalar{H}{ \sigma(u_\mu(t))d\WP(t)}{v_\mu(t)}.	\end{align*}
This implies that
\begin{align*}
\frac{\mu^2}2d\,\left(\norma{H}{ v_\mu(t)}{H}^2u_\mu(t)\right)=&\frac{\mu^2}2d\norma{H}{ v_\mu(t)}{H}^2 u_\mu(t)+\frac{\mu^2}2\,\norma{H}{ v_\mu(t)}{H}^2 v_\mu(t)\,dt\\[10pt]
&\hs\hsl=\mu \scalar{H}{ v_\mu(t)}{\Delta u_\mu(t)} u_\mu(t)\,dt-\mu\gamma \norma{H}{ v_\mu(t)}{H}^2 u_\mu(t)\,dt+\frac 12\Vert \sigma(u_\mu(t))\Vert_{\mathscr{T}_2(K,H)}^2 u_\mu(t)\,dt\\[10pt]
&+\frac{\mu^2}2\,\norma{H}{ v_\mu(t)}{H}^2 v_\mu(t)\,dt+\mu\scalar{H}{ \sigma_0(u_\mu(t))d\WP(t)}{v_\mu(t)} u_\mu(t).
\end{align*}
Rearranging all terms, we get
\begin{align*}
\mu\gamma \norma{H}{ v_\mu(t)}{H}^2 u_\mu(t)\,dt=&
\frac 12\Vert \sigma(u_\mu(t))\Vert_{\mathscr{T}_2(K,H)}^2 u_\mu(t)\,dt-\frac{\mu^2}2d\,\left(\norma{H}{ v_\mu(t)}{H}^2u_\mu(t)\right)\\[10pt]
&\hs+\mu \scalar{H}{ v_\mu(t)}{\Delta u_\mu(t)} u_\mu(t)\,dt\\[10pt]
&+\frac{\mu^2}2\,\norma{H}{ v_\mu(t)}{H}^2 v_\mu(t)\,dt+\mu\scalar{H}{ \sigma_0(u_\mu(t))d\WP(t)}{v_\mu(t)} u_\mu(t),
\end{align*}
and the lemma follows once we divide both sides above by $\gamma$ and integrate with respect to time.

\end{proof}

\begin{Lemma}
\label{lemma3}
	Under Hypotheses \ref{hyp-H3}, \ref{hyp-H4} and \ref{hyp-H5}, for every $(u_0,v_0) \in\,\mathscr{H}_2\cap \mathscr{M}$ and $T>0$ we have
	\begin{equation}
	\label{sz35}
	\lim_{\mu\to 0}\mathbb{E}\sup_{t \in\,[0,T]}\norma{H}{ R_\mu(t)}{H}=0.	
	\end{equation}

\end{Lemma}

\begin{proof}
We use here the same notations as in Lemma \ref{lemma2} and we write
\[R_\mu(t)=\frac{\mu^2}{2\gamma}\norma{H}{ v_0}{H}^2u_0 +\sum_{j=1}^4 I_{\mu,k}(t),\;\;\; t \in\,[0,T].\]	

For $I_{\mu,1}(t)$  we have
\[\norma{H}{ I_{\mu, 1}(t)}{H}=\frac{\mu^2}{2\gamma}\,\norma{H}{ \partial_t u_\mu(t)}{H}^2,\]
and, thanks to \eqref{sz24}, we get
\begin{equation}
\label{sz40}
\lim_{\mu\to 0}\,\mathbb{E}\,\sup_{t \in\,[0,T]}\norma{H}{ I_{\mu,1}(t)}{H}=0.	
\end{equation}

For $I_{\mu,2}(t)$, due to \eqref{sz5} we have
\begin{align*}\mathbb{E}\sup_{t \in\,[0,T]}\norma{H}{ I_{\mu,2}(t)}{H}&\leq c\,\mu\left(\int_0^T \mathbb{E}\Vert \sigma_0(u_\mu(s))\Vert_{\mathscr{T}_2(K,H)}^2\norma{H}{ \partial_t u_\mu(s)}{H}^2\,ds\right)^{\frac 12}\\[10pt]
&=c\,\sqrt{\mu}\left(\mu \int_0^T \mathbb{E}\,\norma{H}{ \partial_t u_\mu(s)}{H}^2\,ds\right)^{\frac 12},\end{align*}
and  \eqref{sz154} allows to conclude that
\begin{equation}
\label{sz41}
\lim_{\mu\to 0}\,\mathbb{E}\,\sup_{t \in\,[0,T]}\norma{H}{ I_{\mu,2}(t)}{H}=0.	
	\end{equation}

For $I_{\mu,3}(t)$, we have
\begin{align*}
\mathbb{E}\sup_{t \in\,[0,T]}\norma{H}{ I_{\mu,3}(t)}{H}&\leq c\,\mu^2\int_0^T\mathbb{E}\,\norma{H}{ \partial_t u_\mu(s)}{H}^3\,dt\leq c_T\sqrt{\mu}\left(\mu^2\int_0^T\mathbb{E}\norma{H}{ \partial_t u_\mu(s)}{H}^4\,ds\right)^{\frac 34},
	\end{align*}
so that, in view of \eqref{sz29} and \eqref{sz70-bis}, we have
\begin{equation}
\label{sz42}
\lim_{\mu\to 0}\,\mathbb{E}\,\sup_{t \in\,[0,T]}\norma{H}{ I_{\mu,3}(t)}{H}=0.	
	\end{equation}

Finally,  since $\vert u_\mu\vert_H=1$, for $I_{\mu,4}(t)$, we have
\begin{align*}
\mathbb{E}\sup_{t \in\,[0,T]}\norma{H}{ I_{\mu,4}(t)}{H}\leq& c\,\mu\int_0^T\mathbb{E} \vert \partial_t u_\mu(s)\vert_{H^1}\vert u_\mu(s)\vert_{H^1}\,ds\\[10pt]
&\hs\leq c\left(\mu\int_0^T\mathbb{E} \vert u_\mu(s)\vert^4_{H^1}\,ds\right)^{\frac 14}\left(\mu\int_0^T\mathbb{E} \vert \partial_t u_\mu(s)\vert^{\frac 43}_{H^1}\,ds\right)^{\frac 34}\\[10pt]
&\hs\leq c_T\left(\mu\int_0^T\mathbb{E} \vert u_\mu(s)\vert^4_{H^1}\,ds\right)^{\frac 14}\left(\mu\int_0^T\mathbb{E} \vert \partial_t u_\mu(s)\vert^2_{H^1}\,ds\right)^{\frac 12}\mu^{\frac 14},
	\end{align*}
and then, according to \eqref{sz16} and \eqref{sz70} and to  \eqref{sz29} and \eqref{sz70-bis},  we have
 \begin{equation}
\label{sz43}
\lim_{\mu\to 0}\,\mathbb{E}\,\sup_{t \in\,[0,T]}\norma{H}{ I_{\mu,4}(t)}{H}=0.	
	\end{equation}

As a result of \eqref{sz40}, \eqref{sz41}, \eqref{sz42} and \eqref{sz43} we obtain \eqref{sz35}.
\end{proof}

	\subsection{Proof of Theorem \ref{teo8.3}}
	In Lemma \ref{lemma7.2} we have proven that the family $\{\mathscr{L}(u_\mu)\}_{\mu \in\,(0,1)}$ is tight in $L^q(0,T;H^\alpha)$, for every $\alpha \in\,[0,2)$ and $q<4/\alpha$. Here we take $\a \in\,[1,2)$. Thanks to \eqref{sz24}, this implies that
	$\{\mathscr{L}(u_\mu, \mu\,\partial_t u_\mu) \}_{\mu \in\,(0,1)}$ is tight in   $L^q(0,T;H^\alpha)\times L^\infty(0,T;H)$, so that, due to the Prohorov Theorem, there exists a weak limit point in the same space. Now, let us     define
	\[\mathscr{K}_T:=\left[L^q(0,T;H^\alpha)\times L^\infty(0,T;H)\right]^2\times C([0,T];U),\]
	where $U$ is a Hilbert space containing the reproducing kernel $K$ with Hilbert-Schmidt embedding. Thanks to the Skorokhod theorem for any two sequences $\{\mu^1_k\}_{j \in\,\mathbb{N}}$ and  $\{\mu^2_k\}_{j \in\,\mathbb{N}}$, both converging to zero, there exist two subsequences, still denoted by $\{\mu^1_k\}_{j \in\,\mathbb{N}}$ and  $\{\mu^2_k\}_{j \in\,\mathbb{N}}$, a sequence of random variables
	\[\mathscr{Y}_k:=((\varrho^1_k,\vartheta^1_k),(\varrho^2_k,\vartheta^2_k),\hat{w}_k^Q),\ \ \ \ k \in\,\mathbb{N},\]
	in $\mathscr{K}_T$ and a random variable $\mathscr{Y}=(\varrho^1,\varrho^2,\hat{w}^Q))$ in $\mathscr{K}_T$, all defined on some probability space $(\hat{\Omega}, \hat{\mathscr{F}}, \hat{\mathbb{P}})$, such that
	 \begin{equation}
	\label{sz200}
	\mathscr{L}(\mathscr{Y}_k)=\mathscr{L}((u_{\mu^1_{k}},\ \mu^1_{k}\,\partial_t u_{\mu^1_k}),\,(u_{\mu^2_{k}},\ \mu^2_{k}\,\partial_t u_{\mu^2_k}),\,\WP),\;\;\; k \in\,\mathbb{N},	
	\end{equation}
and, for $i=1, 2$,
\begin{equation}
\label{sz201-bis}
\lim_{k\to\infty} \vert \varrho^i_{k}-\varrho^i\vert_{ L^q(0,T;H^\alpha)}+	\vert \vartheta^i_{k}\vert_{L^\infty(0,T;H)}+\vert \hat{w}_k^Q-\hat{w}^Q\vert_{C([0,T];U)}=0,\;\;\;\mathbb{P}-\text{a.s.}	\end{equation}
Notice that 	this implies that $\varrho^i(t) \in\, M$, $\hat{\mathbb{P}}$-a.s. and, due to \eqref{sz70-bis},
$\varrho^i \in\,L^2(\Omega,L^2(0,T;H^2))$, for $i=1, 2$.

Next, a filtration $(\hat{\mathscr{F}}_t)_{t\geq 0}$ is introduced in $(\hat{\Omega},\hat{\mathscr{F}},\hat{\mathbb{P}})$, by taking the augmentation of the canonical filtration of $(\rho^1, \rho^2,\hat{w}^Q)$, generated by the restrictions of  $(\rho^1, \rho^2,\hat{w}^Q)$ to every interval $[0,t]$. Due to this construction, $\hat{w}^Q$ is a $(\hat{\mathscr{F}}_t)_{t\geq 0}$ Wiener process with covariance $Q^*Q $ (for a proof see \cite[Lemma 4.8]{DHV2016}).

	Now, if we show that $\varrho^1=\varrho^2$, we have that $u_\mu$ converges in probability  in $L^q(0,T;H^\alpha)$ to some $u
\in\, L^2(\Omega; L^2(0,T;H^2))$. Actually, as observed by Gy\"ongy and Krylov in
\cite{gk}, if $E$ is any Polish space equipped with the Borel
$\sigma$-algebra, a sequence $(\xi_n)_{n \in\,\mathbb{N}}$ of $E$-valued random
variables converges in probability if and only if for every pair
of subsequences $(\xi_m)_{m \in\,\mathbb{N}}$ and $(\xi_l)_{l \in\,\mathbb{N}}$ there exists an
$E^2$-valued subsequence $\eta_k:=(\xi_{m(k)},\xi_{l(k)})$
converging weakly to a random variable $\eta$ supported on the
diagonal $\{(h,k) \in\,E^2\ :\ h=k\}$.

In order to show that $\varrho^1=\varrho^2$, we prove that they are both a solution of equation \eqref{limit-eq}, which has pathwise uniqueness. Due to \eqref{sz200}, we have that
both $(\varrho^1_k,\vartheta_k^1)$ and $(\rho^2_k,\vartheta_k^2)$ satisfy equation \eqref{sz51}, with $\WP$ replaced by $\hat{w}_k^Q$. Then, if we first take the scalar product in $H$ of each term in  \eqref{sz51} with an arbitrary but fixed  $\psi\in C^\infty_0(D)$ and then integrate by parts, we get
\begin{align}
\begin{split}
\label{rhoApproxi}
\scalar{H}{ \gamma &\varrho^i_k(t)+\vartheta^i_k(t)}{\psi} =\scalar{H}{ \gamma\,u_0+\mu_k v_0}{\psi} -\int_0^t \scalar{H}{ \nabla\varrho^i_k(s)) }{\nabla \psi}\,ds\\[10pt]
&+\int_0^t  \vert \varrho^i_k(s)\vert_{H^1}^2\scalar{H}{\varrho_k(s)}{\psi}\, ds -\mu_k\int_0^t\vert\vartheta_k^i(s)\vert_{H}^2\scalar{H}{\varrho_k^i(s)}{\psi}\, ds\\[10pt]
&\hsp+\int_0^t\scalar{H}{ \sigma(\rho^i_k(s))d\hat{w}_k^Q (s)}{\psi}
=:\sum_{j=1}^4 I^i_{k,j}+\int_0^t\scalar{H}{ \sigma(\rho^i_k(s))d\hat{w}_k^Q (s)}{\psi}(t).
\end{split}
\end{align}
Clearly
\[\lim_{k\to\infty} I^i_{k,1}=\scalar{H}{ \gamma u_0}{\psi}.\]
and, due to \eqref{sz201-bis}, since $\a\geq 1$, we have
\[\lim_{k\to\infty}I^i_{k,2}(t)=\int_0^t \scalar{H}{ \nabla \varrho^i(s)}{\nabla \psi}\,ds,\;\;\; \ \hat{\mathbb{P}}-\text{a.s}.\]
Moreover
\begin{align*}
\left| I^i_{k,3}(t)\right.&-\left.\int_0^t \vert \varrho^i(s)\vert^2_{H^1}\scalar{H}{ \varrho^i(s)}{\psi}\,ds\right|\\[10pt]
&\hsl\leq \int_0^t\left|\vert \varrho^i_k(s)\vert_{H^1}^2-\vert \varrho^i(s)\vert_{H^1}^2\right|\left|\scalar{H}{\varrho^i_k(s)}{\psi}\,\right|\,ds +\int_0^t\vert \varrho^i(s)\vert_{H^1}^2\left|\scalar{H}{\varrho^i_k(s)-\varrho^i(s)}{\psi}\,\right|\,ds\\[10pt]
&\hsl\leq \norma{H}{\psi}{H}\int_0^t\vert \varrho^i_k(s)-\varrho^i(s)\vert_{H^1}\left(\vert\varrho^i_k(s)\vert_{H^1}+\vert\varrho^i(s)\vert_{H^1}\right)\,ds\\[10pt]
&+\norma{H}{\psi}{H}\int_0^t\vert \varrho^i(s)\vert_{H^1}^2\,\norma{H}{\varrho^i_k(s)-\varrho^i(s)}{H}\,ds\\[10pt]
&\hsl\leq \norma{H}{\psi}{H}\,\vert \varrho^i_k-\varrho^i\vert_{L^2(0,T;H^1)}\left(\vert\varrho^i_k\vert_{L^2(0,T;H^1)}+\vert\varrho^i\vert_{L^2(0,T;H^1)}\right)\\[10pt]
& \hsp \hsp+\norma{H}{\psi}{H}\vert \varrho^i\vert_{L^2(0,T;H^2)}\vert \varrho^i_k-\varrho^i\vert_{L^2(0,T;H)}.\end{align*}
Thanks again to \eqref{sz201-bis} this allows to conclude that
\[\lim_{k\to\infty} I^i_{k,3}(t)=\int_0^t \vert \varrho^i(s)\vert^2_{H^1}\scalar{H}{ \varrho^i(s)}{\psi}\,ds,\;\;\; \ \hat{\mathbb{P}}-\text{a.s}.\]
Next, as a consequence of \eqref{sz200}, due to Lemmas \ref{lemma2} and \ref{lemma3} we have
\[\lim_{k\to\infty} I^i_{k,4}=-\frac 12\int_0^\cdot\Vert\sigma(\varrho^i(s))\Vert_{\mathscr{T}_2(K,H)}^2 \varrho^i(s)\,ds,\;\;\; \ \text{in}\ \  L^2(\Omega;L^\infty(0,T;H)).\]
Now, for $i=1, 2$ and $t \in\,[0,T]$, we define
\begin{align}
\label{rev1}M^i(t):=&\scalar{H}{ \varrho^i(t)}{\psi}-\scalar{H}{ \gamma\,u_0}{\psi}+\int_0^t \scalar{H}{ \nabla\varrho^i(s)) }{\nabla \psi}\,ds-\int_0^t  \vert \varrho^i(s)\vert_{H^1}^2\scalar{H}{\varrho(s)}{\psi}\, ds \\[10pt]
&+\frac 12\int_0^t\Vert\sigma(\varrho^i(s))\Vert_{\mathscr{T}_2(K,H)}^2 \varrho^i(s)\,ds.
\end{align}
By proceeding as in the proof of \cite[Lemma 4.9]{DHV2016}, thanks to \eqref{sz201-bis} and the limits above for $I_{k,j}^i$, $j=1,2,3,4$,  we have that for every $t \in\,[0,T]$
\[\le<M^i-\int_0^\cdot \scalar{H}{ \sigma (\varrho^i(s))d\hat{w}^Q(s)}{\psi}    \r>_t=0,\ \ \ \ \ \ \hat{\mathbb{P}}-\text{a.s},\]
where $\langle\cdot\rangle_t$ is the quadratic variation process. 
In particular, this implies that for $i=1, 2$ the martingale $M^i$ coincides with the stochastic integral 
\[\int_0^\cdot \scalar{H}{ \sigma (\varrho^i(s))d\hat{w}^Q(s)}{\psi},\]
and if we replace such stochastic integral in \eqref{rev1} we conclude that $\rho^i$ is a solution of  equation \eqref{limit-eq}, with $w^Q$ replaced by $\hat{w}^Q$.
Hence, since
both $\rho^1$ and $\rho^2$ satisfy the same equation	 \eqref{limit-eq},  as we have explained above the pathwise uniqueness of equation \eqref{limit-eq} allows to conclude that   \eqref{sz-fine} holds.

\vfill\newpage

 \appendix

\section{Ito Lemma}\label{sec-Ito Lemma-final}
In this section we formulate and prove a version of the It\^o Lemma that we use twice throughout the paper. Our approach  follows 
the papers \cite{Brz+Ondr_2007} and \cite{Brz+Zhu_2016}. 

To make this section self-contained let us remind the framework. 
Assume that $H$ is a real separable Hilbert space and that $A_0$ is a non-negative self-adjoint linear operator on $H$. We denote 
\[
\mathscr{H} := D(A_0)\times H,\]
and  endow it with a norm (and the corresponding scalar product) 
\[
\norma{\mathscr{H}}{z}{}^2:= \norma{H}{A_0x}{}^2+\norma{H}{x}{}^2+\norma{H}{y}{}^2, \;\; z=(x,y) \in \mathscr{H}.
\]
Note that $D(A_0)$ is a Hilbert, and hence a Banach, space when endowed with the graph norm.
Moreover, we define  a linear operator $\mathscr{A}$ on $\mathscr{H}$ by \eqref{eqn-mathscr A}, i.e. 
\begin{equation}\label{eqn-mathscr A-app}
\begin{aligned}
D(\mathscr{A})&:=D(A_0^2)\times D(A_0),\ \ \ \ \ \ \ \ 
\mathscr{A} z
:=(v,-A_0^2u), \;\;\; \ z=(u,v)\in D(\mathscr{A}).
\end{aligned}
\end{equation}
 It is well known that $\mathscr{A}$ generates a $C_0$ group (of exponential growth) $\mathscr{S}=(\mathscr{S}(t))_{t\in  \mathbb{R}}$ on $\mathscr{H}$, see e.g. \cite{Brz+Masl+Seidl_2005} and references therein.
Let us point out that it does not matter which equivalent norm on $\mathscr{H}$ we choose.

Next, we introduce the function $\Phi$ about which we will formulate our It\^o Lemma\begin{equation}\label{eqn-}
  \Phi(z)= \frac12 \left( \norma{H}{ A_0x}{}^2 + \norma{H}{y}{}^2 + \delta \norma{H}{x}{}^2 \right)  + \beta \scalar{H}{x}{y}
+F(x)  , \quad z = \left(x,y\right)
\in \mathscr{H},
\end{equation}
for some $F:D(A_0) \to \mathbb{R}$.
The function  $\Phi$ satisfies the following properties.
\begin{lemma}\label{ref-Phi}
Assume that the function 
$F:D(A_0) \to \mathbb{R}$ is of $C^2$-class in the Frech\'et sense. Then, for every $\delta \geq0 $ and $\beta \in \mathbb{R}$  the function 
$\Phi: \mathscr{H}\to\mathbb{R}$ is well defined and of $C^2$-class in the Frech\'et sense and its second derivative is bounded on balls. 
 Moreover, for  every  $z = (x,y), h = (h_1,h_2)$ and $ k =
(k_1, k_2)\in \mathscr{H}$, it holds
\begin{align}\label{ito1}
D\Phi(z)\,h &= \scalar{H}{A_0 x}{A_0 h_1}+ \scalar{H}{ y}{h_2}
+ \delta \scalar{H}{x}{h_1} + \beta \left(\scalar{H}{x}{h_2} + \scalar{H}{y}{h_1} \right)
+ DF (x)\,h_1
\end{align}
and \begin{equation}\label{ito2}
\begin{aligned}
D^2\Phi(z)(h,k) &= 
\scalar{H}{A_0 h_1}{A_0 k_1}+ \scalar{H}{h_2}{k_2}
+\delta \scalar{H}{h_1}{k_1} 
\\[10pt]
&\quad \quad \quad+\beta \left(\scalar{H}{h_2}{k_1} + \scalar{H}{h_1}{k_2} \right)
 + D^2F(x)(h_1,k_1).
\end{aligned}
\end{equation}
\end{lemma}

Assume now that $\tau$ is an accessible stopping time  with  approximating sequence $(\tau_n)_{n=1}^\infty$. 
Moreover, assume that $f=(f(t): t \in [0,\tau))$  is an $H$-valued process and  $g=(g(t): t \in [0,\tau))$
is a $\mathscr{T}_2(K,H)$-valued process, both progressively measurable and such that 
 for every $k \in \mathbb{N}$ and every $t \geq 0$, 
\begin{equation}\label{eqn-g+f}
\mathbb{E} \int_0^{t\wedge \tau_k} \left( \Vert g(s)\Vert^2_{\mathscr{T}_2(K,H)}+\norma{H}{f(s)}{}\right)  \,ds < \infty.
\end{equation}
Next we introduce the  progressively measurable processes $\mathbbm{f}=(\mathbbm{f}(t): t \in [0,\tau))$  and  $\mathbbm{g}=(\mathbbm{g}(t): t \in [0,\tau))$
which take values in $\mathscr{H}$ and   $\mathscr{T}_2(K,\mathscr{H})$, respectively, defined by 
\begin{align}
\mathbbm{f}(t)&=(0,f(t)), \ \ \ \ \ \mathbbm{g}(t)=(0,g(t)), \;\; t \in [0,\tau).
\end{align}

\begin{lemma}\label{lem-Ito-final}
Assume that  an $\mathscr{H}$-valued continuous  local process $z(t)=(x(t),y(t))$, $t \in [0,\tau)$, is a mild solution to \eqref{eq-3.10}, i.e.  for every $k \in \mathbb{N}$, 
\begin{equation}
\label{eq-3.10-app}
z(t\wedge \tau_k)=z_0+  I_{\tau_k}(t\wedge \tau_k)+  \int^t_0 \mathbbm{1}_{\left[0,\tau_k\right)}(r)\mathscr{S}(t-r)\,
\mathbbm{f}(r)\d r, \;\; t\geq 0,
\end{equation}
where $z_0=(x_0,y_0) \in \mathscr{H}$ and the process $I_{\tau_k}=\left( I_{\tau_k}(t): t\geq 0\right)$ is defined by 

\begin{equation}\label{eqn-T_tau_kt} \begin{aligned}
I_{\tau_k}(t) &:= \int^t_0 \mathbbm{1}_{\left[0,\tau_k\right)}(r)
\mathscr{S}(t-r)\mathbbm{g}(r)\d W(r).
\end{aligned}
\end{equation}
Then
\begin{equation}\label{eqn-Ito-final} \begin{aligned}
\Phi(z(t))&=\Phi(z_0)+ \int_0^t \Big( \delta \scalar{H}{x(s)}{y(s)}+ \beta\left( \norma{H}{y(s)}{}^2- \norma{H}{A_0x(s)}{}^2 \right)
\\
 &  \quad \quad \quad \quad +\scalar{H}{\beta x(s)+y(s)}{f(s)}     +DF (x(s))\,y(s) +\frac12 \Vert g(t)e_i \Vert_{\mathscr{T}_2(K,H)}^2
 \Big)\, ds 
 \\
 &\quad \quad \quad \quad \quad \quad \quad \quad+\int_0^t  \scalar{H}{ \beta x(s)+ y(s)}{g(t) \, dW(s)}, \ \ \ \ t \in [0,\tau).  
\end{aligned}
\end{equation}
\end{lemma}
It is important to emphasize that   It\^o's formula \eqref{eqn-Ito-final}  should be understood in the following stopped way. For every $k \in \mathbb{N}$, for every $t \geq 0$, 
\begin{equation}\label{eqn-Ito-strong} \begin{aligned}
\Phi(z(t\wedge \tau_k))&=\Phi(z_0)+ \int_0^{t \wedge \tau_k} \Big( \delta \scalar{H}{x(s)}{y(s)}+ \beta\left( \norma{H}{y(s)}{}^2- \norma{H}{A_0x(s)}{}^2 \right)
\\
 & \quad \quad \quad \quad +  \scalar{H}{\beta x(s)+y(s)}{f(s)}     +DF (x(s))\,y(s)x +\frac12 \Vert g(t)e_i \Vert_{\mathscr{T}_2(K,H)}^2
 \Big)\, ds 
 \\
 &\quad \quad \quad \quad \quad \quad \quad \quad+\int_0^{t \wedge \tau_k}  \scalar{H}{ \beta x(s)+ y(s)}{g(t) \, dW(s)}.
\end{aligned}
\end{equation}

\begin{proof}[Proof of Lemma \ref{lem-Ito-final}]
Assume that $n \in \mathbb{N}$ is big enough so that  $n \in \rho(\mathscr{A})$, i.e.  $(nI+A_0^2)^{-1}$ exists and is bounded. Denote 
\[
\mathscr{R}_n=\left( \begin{array}{cc} (nI+A_0^2)^{-1} & 0 \\ 0 & (nI+A_0^2)^{-1}  \end{array}
 \right).
\]
Since $ (nI+A_0^2)^{-1}$  maps boundedly  $H$ into $D(A_0)$ and   $D(A_0^2)$, the operator $\mathscr{R}_n$ is a bounded linear map in $\mathscr{H}$  and it maps 
boundedly  $\mathscr{H}$ into $D(\mathscr{A})$. Moreover, by direct calculations one can verify that $\mathscr{R}_n$ commutes with $\mathscr{A}$ and therefore, it also commutes with the  group 
$\left(\mathscr{S}(t): t\in \mathbb{R}\right)$. 

From now on we assume that the processes $z$, $f$, $g$ etc are as in the fomulation of the Lemma. We define a set of new processes as follows,  for every $t \in [0,\tau)$,

\begin{equation}
\label{eqn-A.301}
\begin{aligned}
&z_n(t):= \mathscr{R}_n\,z(t), \ \ \ \ \ 
x_n(t):=(nI+A_0^2)^{-1}\,x(t), \ \ \ \ \ 
y_n(t):=(nI+A_0^2)^{-1}\,y(t), \\[10pt]
&f_n(t):=(nI+A_0^2)^{-1}\,f(t), \ \ \ \ \ g_n(t):=(nI+A_0^2)^{-1}\,g(t), \\[10pt]
&\mathbbm{f}_n(t):=\mathscr{R}_n\,\mathbbm{f}_n(t),\ \ \ \ \ 
\mathbbm{g}_n(t):=\mathscr{R}_n\,\mathbbm{g}_n(t).
\end{aligned}
\end{equation}

We note then that for $t \in [0,\tau)$, 
\begin{equation}
\label{eqn-A.302}
\begin{aligned}
&z_n(0):= \mathscr{R}_n\,z_0=(x_n(0),y_n(0)), \ \ \ \ \ 
z_n(t)=(x_n(t),y_n(t)),
\\[10pt]
&\mathbbm{f}_n(t):=(0,f_n(t)),\ \ \ \ \ 
\mathbbm{g}_n(t):=(0,g_n(t)).
\end{aligned}
\end{equation}
We also observe that $z_n$ is an $D(A_0^3)\times D(A_0^2)$-valued continuous  local process, 
$f_n$ is an $D(A_0^2)$-valued  local process and  $g_n$ is an 
is $\mathscr{T}_2(K,D(A_0^2))$-valued process. Both  $f_n$ and $g_n$  are 
  progressively measurable  and 
 for every $k \in \mathbb{N}$ and every $t \geq 0$, 
\begin{equation}\label{eqn-g+f_n}
\mathbb{E} \int_0^{t\wedge \tau_k} \left[ \Vert g_n(s)\Vert^2_{\mathscr{T}_2(K,D(A_0^2))}+\norma{D(A_0^2)}{f_n(s)}{}\right]  \,ds < \infty.
\end{equation}
Moreover, by the commutativity property stated earlier, the process $z_n$ satisfies identities \eqref{eq-3.10-app} in $D(A_0^3)\times D(A_0^2)$ with appropriate and obvious modifications, i.e.  
 for every $k \in \mathbb{N}$, 
\begin{equation}
\label{eq-3.10-app-_n}
z_n(t\wedge \tau_k)=z_n(0)+  I_{n,\tau_k}(t\wedge \tau_k)+  \int^t_0 \mathbbm{1}_{\left[0,\tau_k\right)}(r)\mathscr{S}(t-r)
\mathbbm{f}_n(r)\d r, \;\; t\geq 0,
\end{equation}
where $z_0=(x_0,y_0) \in \mathscr{H}$ and the process $I_{\tau_k}=\left( I_{\tau_k}(t): t\geq 0\right)$ is defined by 

\begin{equation}\label{eqn-T_tau_kt} \begin{aligned}
I_{n,\tau_k}(t) &:= \int^t_0 \mathbbm{1}_{\left[0,\tau_k\right)}(r)
\mathscr{S}(t-r)\,\mathbbm{g}_n(r)\d W(r), \;\; t \geq 0.
\end{aligned}
\end{equation}
 
Moreover, by the Chojnowska-Michalik theorem, see \cite{Chojnowska_1979} or  \cite[Theorem 12]{Ondr_2004}, see also the proof of   \cite[Proposition 6.1]{Brz+Ondr_2007}, 
for all $t \geq 0$,
\begin{equation}
\label{eq-3.10-app-_n}
z_n(t\wedge \tau_k)=z_n(0)+
\int^{t\wedge \tau_k}_0 \left[ \mathscr{A} z_n(s)+ \mathbbm{f}_n(r) \right]\d r+
 \int^{t\wedge \tau_k}_0\mathscr{A} \mathbbm{g}_n(r)\d W(r) , \;\; t\geq 0.
\end{equation}
Next, by the classical strong version of the It\^o Lemma we infer that 
\begin{equation}\label{eqn-Ito-n} 
\begin{aligned}
\Phi(z_n(t))&=\Phi(z_0)+ \int_0^t D\Phi(z_n(s))(\mathscr{A} z_n(s)+ \mathbbm{f}_n(r) )\, dr + \frac12 \int_0^t \tr_{K} [D^2\Phi(z_n(s))(\circ \mathbbm{g}(s),\circ \mathbbm{g}(s))]\, ds
 \\
 &\quad \quad \quad \quad \quad \quad\quad \quad \quad \quad+\int_0^t D\Phi(z_n(s) )\circ \mathbbm{g}_n(s)  \, dW(s). 
\end{aligned}
\end{equation}
Applying Lemma \ref{ref-Phi} as well as one of the identities in \eqref{eqn-A.302},  we infer that  for $s \in [0,\tau)$

\begin{equation} \label{eqn-bms2.6}
\begin{aligned}
D\Phi(z_n(s))(\mathscr{A}z_n(s)&+\mathbbm{f}_n(s))=
\delta \scalar{H}{x_n(s)}{y_n(s)}+ \beta\left( \norma{H}{y_n(s)}{}^2- \norma{H}{A_0x_n(s)}{}^2 \right)
\\[10pt]
&
 \quad \quad  + DF (x_n(s))\,y_n(s)
\scalar{H}{ y}{f_n(s)}
+\beta \scalar{H}{x_n(s)}{f_n(s)}   
\end{aligned}
\end{equation}
and, with   $\{e_i\}_{i\in I}$ being  an  arbitrary orthonormal basis  in $K$, 
\begin{align}
&\tr_{K} D^2\Phi(z_n(s))(\circ \mathbbm{g}_n(s), \circ \mathbbm{g}_n(s))
= \sum_{ i \in I} \scalar{H}{g_n(s)e_i}{g_n(s)e_i}
\\[10pt]
&\quad \quad \quad \quad \quad =\sum_{ i \in I}  \norma{H}{g_n(s)e_i}{}^2= \Vert g_n(s)e_i \Vert_{\mathscr{T}_2(K,H)}^2.
\end{align}
Therefore, we deduce that $\Phi(z_n(t))$ satisfies the desired identity \eqref{eqn-Ito-strong}, i.e.  for every $ t\geq 0$, 
\begin{equation}\label{eqn-Ito-strong-_n} \begin{aligned}
\Phi(z_n(t\wedge \tau_k))&=\Phi(z_n(0))+ \int_0^{t \wedge \tau_k} \Big( \delta \scalar{H}{x_n(s)}{y_n(s)}+ \beta\left( \norma{H}{y_n(s)}{}^2- \norma{H}{A_0x_n(s)}{}^2 \right)
\\
 & +\scalar{H}{\beta x_n(s)+y_n(s)}{f_n(s)}     +DF (x_n(s))\,y_n(s) +\frac12 \Vert g_n(s) \Vert_{\mathscr{T}_2(K,H)}^2
 \Big)\, ds 
 \\
 &\quad \quad \quad +\int_0^{t \wedge \tau_k}  \scalar{H}{ \beta x_n(s)+ y_n(s)}{g_n(s) \, dW(s)}.
\end{aligned}
\end{equation}

Observe that  $\mathbb{P}$- a.s. for every compact interval $[0,T] \subset [0,\tau)$, 
the following convergences are satisfied uniformly on $[0,T]$ 
\begin{align*}
&z_n(t) \to z(t) \mbox{ in } \mathscr{H}, \ \ \ x_n(t) \to x(t) \mbox{ in } D(A_0), \ \ \ 
y_n(t) \to y(t) \mbox{ in } H, 
\end{align*}
and, for every $t \geq 0$, 
\begin{equation}\label{eqn-g+f-limit}
\mathbb{E} \int_0^{t\wedge \tau_k} \left[ \Vert g_n(s)-g(s)\Vert^2_{\mathscr{T}_2(K,D(A_0^2))}+\norma{D(A_0^2)}{f_n(s)-f(s)}{}\right]  \,ds < \infty.
\end{equation}
Thus, we conclude the proof of \eqref{eqn-Ito-strong} by taking the limit as $n \to \infty$ of 
equalities \eqref{eqn-Ito-strong-_n}. Compare with the proofs of  \cite[Lemma 5.2]{Brz+Zhu_2016} and/or \cite[Proposition 6.1]{Brz+Ondr_2007}.

\end{proof}

\end{document}